\documentclass[11pt]{amsart}

\usepackage{amsmath}
\usepackage{amssymb}
\usepackage{enumerate}
\usepackage{euscript}

\newtheorem{thm}{Theorem}[section]
\newtheorem{lem}[thm]{Lemma}
\newtheorem{prop}[thm]{Proposition}
\newtheorem{cor}[thm]{Corollary}
\theoremstyle{definition}
\newtheorem{defn}[thm]{Definition}
\newtheorem{rem}[thm]{Remark}

\newtheorem{asmp}[thm]{Assumption}

\numberwithin{equation}{section}
\theoremstyle{remark}

\newcommand{\bba}{{\mathbb A}}

\newcommand{\bbc}{{\mathbb C}}

\newcommand{\bbq}{{\mathbb Q}}
\newcommand{\bbr}{{\mathbb R}}

\newcommand{\bbz}{{\mathbb Z}}

\font\tenscr=rsfs10 

\newcommand{\sS}{\hbox{\tenscr S}}

\newcommand{\esA}{{\EuScript{A}}}
\newcommand{\esB}{{\EuScript{B}}}
\newcommand{\esC}{{\EuScript{C}}}

\newcommand{\cM}{{\mathcal M}}

\newcommand{\co}{{\mathcal O}}
\newcommand{\cD}{{\mathcal D}}

\newcommand{\aut}{{\operatorname{Aut}}\,}

\newcommand{\Disc}{{\operatorname{Disc}}}

\newcommand{\aff}{{\operatorname {Aff}}}

\newcommand{\n}{{\operatorname {N}}}
\newcommand{\m}{{\operatorname {M}}}

\newcommand{\gl}{{\operatorname{GL}}}

\newcommand{\spl}{{\operatorname{SL}}}

\newcommand{\ord}{{\operatorname{ord}}}

\newcommand{\cl}{{\operatorname{Cl}}}

\newcommand{\res}{\operatorname{Res}}

\newcommand{\A}{\bba}
\newcommand{\Z}{\bbz}
\newcommand{\Q}{\bbq}
\newcommand{\R}{\bbr}
\newcommand{\C}{\bbc}

\newcommand{\md}{d^{\times}}

\newcommand{\twtw}[4]
{\begin{pmatrix}{#1}&{#2}\\{#3}&{#4}\\\end{pmatrix}}
\newcommand{\F}{\mathbb F}
\newcommand{\rf}{{\rm f}}

\newcommand{\ww}{x}

\begin{document}

\title[Orbital $L$-function]
{Orbital $L$-functions for the space of
binary cubic forms}

\author[Takashi Taniguchi]
{Takashi Taniguchi}
\address{
Department of Mathematics,
Graduate School of Science, Kobe University,
1-1, Rokkodai, Nada-ku, Kobe 657-8501, Japan}
\address{
Department of Mathematics, Princeton University,
Fine Hall, Washington Road, Princeton, NJ 08540}
\email{tani@math.kobe-u.ac.jp}
\author[Frank Thorne]
{Frank Thorne}
\address{Department of Mathematics,
University of South Carolina,
1523 Greene Street, Columbia, SC 29208
}
\email{thorne@math.sc.edu}
\date{\today}


\maketitle

\begin{abstract}
We introduce the notion of orbital $L$-functions
for the space of binary cubic forms
and investigate their analytic properties.
We study their functional equations and residue formulas in some detail.
Aside from the intrinsic interest,
results from this paper are used to 
prove the existence of secondary terms in counting
functions for cubic fields.
This is worked out in a companion paper.
\end{abstract}

\tableofcontents

\section{Introduction}\label{sec:intro}

The theory of {\em prehomogeneous vector spaces}
was initiated by M. Sato in early 1960s.
A finite dimensional representation of complex algebraic group
$(G,V)$ is called a prehomogeneous vector space if
there exists a Zariski open orbit.
One arithmetic significance of this is,
if $(G,V)$ is defined over a number field,
then there exist {\em zeta functions} associated to $(G,V)$
which have analytic continuations and satisfy functional equations.
This was discovered by M. Sato and
Shintani \cite{sash} and numerous number theoretic applications
have been given (see e.g., \cite{dawrb}, \cite{ibsaa}, \cite{nakagawa},
\cite{shintania}, \cite{shintanib}, \cite{zps}, \cite{wryu}, \cite{yukiec}.)

Let $(G,V)$ be the space of binary cubic forms:
\begin{equation*}
G:=\gl_2,
\qquad
V:=\{x(u,v)=x_1u^3+x_2u^2v+x_3uv^2+x_4v^3
\}.
\end{equation*}
The discriminant
$
\Disc(x(u,v))=x_2^2x_3^2+18x_1x_2x_3x_4-4x_1x_3^3-4x_2^3x_4-27x_1^2x_4^2
$
is relatively invariant under the action of $G$,
i.e., $\Disc(gx)=(\det g)^2\Disc(x)$.
We denote by $V^\ast$ the dual
representation of $G$, which is similar to $V$ but
has a slightly different integral structure.

This $(G,V)$ is an interesting example of a prehomogeneous vector space,
and the associated zeta functions were studied extensively
by Shintani \cite{shintania}.
He introduced the Dirichlet series
\[
\xi_\pm(s)
:=\sum_{\substack{x\in\spl_2(\Z)\backslash V(\Z),\ \pm\Disc(x)>0}}
\frac{|{\rm Stab}(x)|^{-1}}{|\Disc(x)|^s}
\]
associated to the positive and negative subsets of $V(\Z)$,
and he similarly associated $\xi^\ast_\pm(s)$ to $V^\ast(\Z)$.
(Here $|{\rm Stab}(x)|$ is the order of the stabilizer
of $x$ in $\spl_2(\Z)$.)
Then he established their notable analytic properties.

\begin{thm}[Shintani]\label{thm:intro_Shintnai}
The four Dirichlet series
$\xi_\pm(s)$ and $\xi^\ast_\pm(s)$
have holomorphic continuations
to the whole complex plane except for simple poles at
$s=1,5/6$, and we have explicit formulas for their residues.
Moreover, these Dirichlet series satisfy the functional equation
\[
\begin{pmatrix}\xi_+(1-s)\\\xi_-(1-s)\end{pmatrix}
=
\frac{3^{3s-2}}{2\pi^{4s}}
\Gamma(s)^2\Gamma\left(s-\tfrac16\right)\Gamma\left(s+\tfrac16\right)
\begin{pmatrix}
\sin 2\pi s&\sin \pi s\\
3\sin \pi s&\sin2\pi s\\
\end{pmatrix}
\begin{pmatrix}\xi^\ast_+(s)\\\xi^\ast_-(s)\end{pmatrix}.
\]
\end{thm}

The purpose of this paper is to study $L$-functions
corresponding to Shintani's zeta functions, extending
Datskovsky and Wright's work \cite{dawra}.
Let us briefly explain our formulation.
We fix a positive integer $N$.
For any $a \in \Z$, the usual 
partial function $\zeta(s,a)$ is defined by the 
formula $\zeta(s, a) = \sum_{n \in a + N\Z; \ n > 0} n^{-s}$.
Extending this idea, for any $a\in V(\Z)$
we define the {\em partial zeta function} by
\[
\xi_\pm(s,a)
:=\frac{1}{[\spl_2(\Z):\Gamma(N)]}
\sum_{\substack{x\in \Gamma(N)\backslash(a+NV(\Z))\\\pm\Disc(x)>0}}
\frac{|{\rm Stab}(x)|^{-1}}{|\Disc(x)|^s}.
\]
Here $\Gamma(N)$ is the principal congruence subgroup of $\spl_2(\Z)$,
and now $|{\rm Stab}(x)|$ denotes the size of the group of stabilizers
of $x$ in $\Gamma(N)$.
Since $\xi_\pm(s,a')$ counts the same orbits
if $a'\equiv a\mod N$,
it is natural to regard $a$ of $\xi_\pm(s,a)$
as an element of $V(\Z/N\Z)$
rather than of $V(\Z)$. We can easily check that
$\xi(s)=\sum_{a\in V(\Z/N\Z)}\xi(s,a)$, as expected.

Recall that the group $G(\Z/N\Z)$ acts on $V(\Z/N\Z)$.
We may now define the {\em orbital $L$-function} by
\[
\xi_\pm(s,\chi,a):=\sum_{g\in G(\Z/N\Z)}\chi(\det g)\xi_\pm(s,ga)
\]
for a Dirichlet character $\chi$ modulo $N$.
We hope the analogy to the Dirichlet $L$-function
$L(s,\chi)=\sum_{t\in(\Z/N\Z)^\times}\chi(t)\zeta(s,t)$ is clear\footnote{
As we will see in Lemma \ref{lem:zeta_basic} (3),
strong approximation for $\spl_2$ implies
that the action of $\spl_2(\Z/N\Z)$ on (the space of) partial zeta functions is
trivial. Hence it is enough to work with one dimensional representations
$\chi \circ \det$ of $\gl_2(\Z/N\Z)$; we can isolate each $\xi(s,a)$
by the orthogonality of characters on $(\Z/N\Z)^\times$.}.
This orbital $L$-function seems to be a natural class of
$L$-functions in the theory of prehomogeneous
vector spaces, and we focus on this $\xi_\pm(s,\chi,a)$.
We note that certain $L$-functions are introduced and studied in detail
in the extensive works of Datskovsky and Wright \cite{wright, dawra},
and our orbital $L$-functions are closely related to theirs.

\bigskip

In this paper we prove three main results.
The first one establishes fundamental analytic properties
for our zeta functions.
\begin{thm}\label{thm:introL}
For any congruence $N$,
the orbital $L$-functions $\xi_\pm(s,\chi,a)$
and the partial zeta functions $\xi_\pm(s,a)$
have meromorphic continuations to whole complex plane
and satisfy certain functional equations.
They are holomorphic except for possible simple poles at $s=1$ and $s=5/6$,
and their residues are described in terms of certain sums
over the $G(\Z/N\Z)$-orbit of $a\in V(\Z/N\Z)$.
We have explicit formulas of those residues in various cases.
In particular, we have residue formulas when $a \in V(\Z)$ detects cubic rings maximal at
all primes dividing $N$, or when $N$ is cube free.
\end{thm}
Hence, in principle we can understand the contributions
of an arbitrary subset $X\subset V(\Z)$ to the zeta function,
as long as $X$ is defined by
finite number of congruence conditions in $V(\Z)$.
We note that the first statement in this theorem
is due to F. Sato \cite{fsatoe}.
This theorem is a combination of various results in this paper,
and we give the proof in Remark \ref{rem:orbitalL}.
This result implies that there is a bias
for the class numbers of integral binary cubic forms
in arithmetic progressions, which seems to have been missed
in the literature.
We prove it in Theorem \ref{thm:bias_classnumber}.

For a function $f$ on $V(\Z/N\Z)$, its finite Fourier transform
$\widehat f$ is defined by
\[
\widehat f(b):=N^{-4}\sum_{a\in V(\Z/N\Z)}f(a)
	\exp\left(2\pi\sqrt{-1}\cdot\frac{[a,b]}{N}\right),
\qquad b\in V^\ast(\Z/N\Z),
\]
where $[*,*]$ is the canonical pairing between $V$ and $V^\ast$.
Our second result is explicit formulas of the Fourier transforms
of certain functions on $V(\Z/p^2\Z)$.
\begin{thm}[Theorems \ref{thm:hat(fp^2)}, \ref{thm:hat(fp^2')}]
\label{thm:intro_fourier}
Let $p$ be a prime not equal to $2$ and $3$.
We have the explicit formulas of the Fourier transforms of
$\Phi_{p}$ and $\Phi_{p}'$, where these are functions
over $V(\Z/p^2\Z)$ detecting nonmaximal and
nonmaximal-or-totally-ramified cubic rings at $p$, respectively.
\end{thm}
See Theorems \ref{thm:hat(fp^2)} and \ref{thm:hat(fp^2')}
for the exact formulas. These Fourier transforms occur in the
explicit formulas for the zeta and $L$-functions 
dual to $\xi_{\pm}(s, a)$ and $\xi_{\pm}(s, \chi, a)$, and
Theorem \ref{thm:intro_fourier} allows us to write down these
explicit formulas. This leads to an improved analytic understanding
of  $\xi_{\pm}(s, a)$ and $\xi_{\pm}(s, \chi, a)$.

Theorem \ref{thm:introL}, in combination with explicit
results such as Theorem \ref{thm:intro_fourier}, has fruitful arithmetic applications.
To explain our motivation, we quote
the main results of our companion paper \cite{scc}.
Our primary purpose in proving Theorems \ref{thm:hat(fp^2)}
and \ref{thm:hat(fp^2')} is to obtain the following
density theorems:
\begin{thm}[\cite{scc}]\label{thm:introscc}
\begin{enumerate}[{\rm (1)}]
\item
The number of cubic fields $K$ with $0<\pm\Disc(K)<X$ is
\[
{N_3^\pm(X)}
=\frac{C^\pm}{12\zeta(3)}X
+K^\pm\frac{4\zeta(1/3)}{5\Gamma(2/3)^3\zeta(5/3)}X^{5/6}
+O(X^{7/9+\epsilon}),
\]
where $C^+=K^+=1$ and $C^-=3,K^-=\sqrt3$.
\item
For a quadratic field $F$, let $\cl_3(F)$ denote the $3$-torsion subgroup
of the ideal class group of $F$. Then
\[
\sum_{\substack{[F:\Q]=2\\0<\pm\Disc(F)<X}}\!\!\!\!\!\!\#\cl_3(F)
=\frac{3+C^\pm}{\pi^2}X
+K^\pm\frac{8\zeta(1/3)}{5\Gamma(2/3)^3}
\prod_p\left(1-\frac{p^{1/3}+1}{p(p+1)}\right)X^{5/6}
+O(X^{18/23+\epsilon}),
\]
where the product in the secondary term is over all primes.
\end{enumerate}
\end{thm}

Part (1) of Theorem \ref{thm:introscc} was also proved in independent work of
Bhargava, Shankar, and Tsimerman \cite{BST}. Their paper studies the 
space $(G, V)$ geometrically, and does not apply the theory of the associated zeta
or $L$-functions.

We generalize this theorem to count these discriminants
in arbitrary arithmetic progressions as well.
In this case we discover a curious bias in the secondary term.
For example, when the modulus $m$ is not divisible by $4$ we prove:
\begin{thm}[\cite{scc}]\label{thm:introsccchi}
Suppose $m$ is a positive integer with $4\nmid m$ and $a\in\Z$ arbitrary.
\begin{enumerate}[{\rm (1)}]
\item
The number of cubic fields $K$ with $0<\pm\Disc(K)<X$
and $\Disc(K)\equiv a\mod m$ is
\[
N_3^\pm(X;m,a)
=C_1^\pm(m,(m,a))X
+K_1^\pm(m,a)X^{5/6}
+O(X^{7/9+\epsilon}m^{8/9}).
\]
The constant
$C_1^\pm$ depends only on $m$ and the greatest common divisor
$(m,a)$ of $m$ and $a$,
but 
$K_1^\pm$ may be different for different values of $a$, even when
$m$ and $(m,a)$ are fixed,
if there exist any nontrivial cubic characters modulo $m/(m,a)$.
\item
We have
\[
\sum_{\substack{[F:\Q]=2,\ 0<\pm\Disc(F)<X\\\Disc(F)\equiv a\!\!\!\mod m}}\!\!\!\!\!\!\!\!\#\cl_3(F)
=C_2^\pm(m,(m,a))X
+K_2^\pm(m,a)X^{5/6}
+O(X^{18/23+\epsilon}m^{{20}/{23}}),
\]
where $C_2^\pm$ and $K_2^\pm$ have the same properties as
$C_1^\pm$ and $K_1^\pm$, respectively.
\end{enumerate}
\end{thm}
We refer to \cite{scc} for more explicit and general statements, including
an explicit evaluation of the constants  $C_i^\pm, K_i^\pm$, as well
as associated numerical data. When $4\mid m$ the statements change slightly (the
discriminant of any field is $\equiv 0,1\mod 4$), but
we treat this case as well. 

Concerning Theorem \ref{thm:introsccchi},
Datskovsky and Wright \cite{dawra} proved that certain $L$-functions
may have a pole if the character is cubic but are otherwise entire, and
this is the origin of subtle behaviours of $K_i^\pm$.
In Section \ref{sec:residue} we refine their significant
residue formulas \cite{dawra} for $\xi(s,\chi,a)$,
and in particular we complete all
the cases where $a$ detects cubic rings maximal at all primes dividing $N$.
These formulas are used in \cite{scc}
to obtain the explicit formulas for $K_i^\pm$.
In Section \ref{sec:example},
we briefly discuss how Theorem \ref{thm:introsccchi}
relates to these residue formulas.
\bigskip

Besides these arithmetic applications, our explicit construction
of $L$-functions is also motivated by work of Ohno and Nakagawa.
In 1997, Ohno \cite{ohno} conjectured the remarkably simple relations
$\xi^\ast_+(s)=\xi_-(s)$ and $\xi^\ast_-(s)=3\xi_+(s)$, and 
these relations were proved by Nakagawa \cite{nakagawa}.
As a consequence, 
Shintani's functional equation in Theorem \ref{thm:intro_Shintnai}
takes the following self dual form:
\begin{thm}[Ohno-Nakagawa]\label{thm:intro_ON}
Let
$\theta_\pm(s):=\sqrt3\xi_+(s)\pm\xi_-(s)$
for each sign. Then
\begin{equation*}
\Delta_\pm(1-s)\theta_\pm(1-s)
=\Delta_\pm(s)\theta_\pm(s),
\end{equation*}
where
\begin{align*}
\Delta_+(s)&:=
\left(\frac{2^43^3}{\pi^4}\right)^{s/2}
		\Gamma\left(\frac s2\right)\Gamma\left(\frac s2+\frac12\right)
		\Gamma\left(\frac{s}{2}-\frac1{12}\right)
		\Gamma\left(\frac{s}{2}+\frac1{12}\right),\\
\Delta_-(s)&:=
\left(\frac{2^43^3}{\pi^4}\right)^{s/2}
		\Gamma\left(\frac s2\right)\Gamma\left(\frac s2+\frac12\right)
		\Gamma\left(\frac{s}{2}+\frac5{12}\right)
		\Gamma\left(\frac{s}{2}+\frac7{12}\right).
\end{align*}
\end{thm}
If we put
$\Delta(s)=\left(\begin{smallmatrix}\Delta_+(s)&0\\0&\Delta_-(s)\end{smallmatrix}\right)$,
$T=\left(\begin{smallmatrix}\sqrt3&1\\\sqrt3&-1\end{smallmatrix}\right)$
and
$\xi(s)=\left(\begin{smallmatrix}\xi_+(s)\\\xi_-(s)\end{smallmatrix}\right)$,
then the formula is
\[
\Delta(1-s)\cdot T\cdot \xi(1-s)=\Delta(s)\cdot T\cdot \xi(s).
\]
Although the formulas
$\xi^\ast_+(s)=\xi^\ast_-(s)$ and $\xi^\ast_-(s)=3\xi_+(s)$ are very simple,
Nakagawa's proof is quite technical; 
in particular, he used class field theory in a sophisticated manner.
Analogous formulas were proved by Ohno, the first author, and Wakatsuki
\cite{yt,sty} for zeta functions associated with other integral models for
$(G,V)$, leading to similar functional equations.

In this paper we will prove the following.
\begin{thm}[Theorem \ref{thm:FE_thetaN}]\label{thm:intro_divisible}
For a positive integer $m$,
let
\[
\xi_{m,\pm}(s)
:=
\sum_{\substack{x\in \spl_2(\Z)\backslash V(\Z)\\m\mid\Disc(x),\ \pm\Disc(x)>0}}
\frac{|{\rm Stab}(x)|^{-1}}{|\Disc(x)|^s},
\qquad
\xi_m(s)
:=
\left(
\begin{array}{l}
\xi_{m,+}(s)\\
\xi_{m,-}(s)
\end{array}
\right),
\]
and for a square free integer $N$, write
\begin{equation*}
\theta_N(s):=\sum_{m\mid N}\mu(m)m\xi_m(s).
\end{equation*}
With this notation, $\theta_N(s)$
satisfies the functional equation
\[
N^{2(1-s)}\Delta(1-s)\cdot T\cdot \theta_N(1-s)
=N^{2s}\Delta(s)\cdot T\cdot \theta_N(s).
\]
\end{thm}
Here $\mu(m)$ is the M\"obius function.
In Theorem \ref{thm:FE_thetaN},
we also describe the residues of $\theta_N(s)$.
The case $N=1$ is Theorem \ref{thm:intro_ON}, and
we will in fact
reduce the proof of this theorem to Ohno-Nakagawa's original formula.
In the proof, we use Mori's explicit formulas \cite{mori} for
certain orbital Gauss sums over $\Z/p\Z$.

\bigskip

As the terminology ``orbital $L$-function'' indicates, understanding
the $G(\Z/N\Z)$-orbit structure of $a\in V(\Z/N\Z)$
is fundamental for the analysis of $\xi_\pm(s,\chi,a)$.
We pursue the theory from this viewpoint.
As we noted earlier, our $L$-functions
are closely related to those studied by
Datskovsky and Wright \cite{dawra}.
Although our work overlaps with theirs to some extent,
our orbital $L$-functions have elementary and explicit
descriptions, and we hope we developed the theory
to a point where it is quite usable in applications.

Although we focus on the particular example of the space of binary cubic forms,
the definitions and basic properties studied in
Sections \ref{sec:L} and \ref{sec:FE} are applicable
to general irreducible regular prehomogeneous vector spaces
with a fixed integral model.
Also, our arguments in Sections \ref{sec:orbit},
\ref{sec:gausssum} and \ref{sec:residue},
regarded as arguments over $\Z_p$, 
generalize to the integer ring of other local fields, and
a version of Theorem \ref{thm:intro_fourier}
holds over an arbitrary local field with residue
characteristic not $2$ or $3$.
In a forthcoming paper (joint with Manjul Bhargava), this will be used to improve the error term
in the function counting cubic extensions
of any base number field, previously studied by Datskovsky and
Wright \cite{dawrb}.

We also discuss some related results and problems.
When a Dirichlet series $\xi(s)=\sum a_n/n^s$ is given,
it is natural to 
twist by a Dirichlet character $\chi$, yielding the $L$-function
$\xi(s,\chi)=\sum a_n\chi(n)/n^s$.
$L$-functions of this type associated to 
prehomogeneous vector spaces have been previously discussed in the literature; 
see, e.g., \cite{ibsad}, \cite{hsaitoe},
\cite{hsaitod}, \cite{fsatoe}.
For our case of the space of binary cubic forms,
we discuss this $\xi(s,\chi)$ in Section \ref{sec:example}.
Since this $\xi(s,\chi)$ is expressed
in terms of linear combinations of
orbital $L$-functions of the form $\xi(s,\chi^2,a)$,
in principle our theory contains the theory of $\xi(s,\chi)$.
However, there are some rich stories
involving $\xi(s,\chi)$ for which we do not yet have good analogues for $\xi(s,\chi,a)$.
Among others, we mention the significant work of Denef and Gyoja \cite{degy},
who proved an explicit formula for a certain Gauss sum.
As a result, the functional equation of $\xi(s,\chi)$ turns out
to have a nice simple form, as observed in \cite{fsatoe}.
Their work is notable since the formula is proved
for a general prehomogeneous vector space.
In this direction, although we obtain
explicit formulas of orbital Gauss sums
\[
W(\chi,a,b)
:=\sum_{g\in G(\Z/N\Z)}\chi(\det g)
	\exp\left(2\pi\sqrt{-1}\cdot\frac{[ga,b]}{N}\right)
\]
for some special $a\in V(\Z/N\Z), b\in V^\ast(\Z/N\Z)$,
it would be very interesting to further investigate the general case.

We also remark on the secondary pole of zeta functions
for ``cubic cases''. What principle underlies the fact
that the space of binary cubic forms $(G,V)$ describes the family
of cubic extensions?
In 1992, Wright and Yukie \cite{wryu}
clarified that this is because the component group of the generic stabilizer
is isomorphic to $\mathfrak S_3$, the permutation group of degree $3$,
and they studied the relationship of this fact to geometric interpretations of rational orbits.
Among $29$ types of irreducible regular prehomogeneous vector spaces
classified by M. Sato and Kimura \cite{saki}, $4$ of them
share this property and hence they should be regarded as
``cubic cases''. One such cubic case is the representation
$(\gl_2\times\gl_3^2, \aff^2\otimes \aff^3\otimes\aff^3)$,
and the global theory for a non-split form of this representation
was given by the first author \cite{zps}.
Interestingly, as with the $(G,V)$ studied in this paper,
the secondary pole of the zeta function does not vanish
under a twist by cubic characters.
It is likely that this property is shared for all cubic case zeta functions,
and ultimately this would reflect phenomena
similar to Theorem \ref{thm:introsccchi}.

\bigskip

This paper is organized as follows:
In Section \ref{sec:GVdefn}
we introduce the space of binary cubic forms $V$
and its dual space $V^\ast$, with natural actions of
$G=\gl_2$. We also recall the Delone-Faddeev correspondence.
In Section \ref{sec:L}
we introduce the orbital $L$-functions $\xi(s,\chi,a)$
and the orbital Gauss sums $W(\chi,a,b)$.
In Section \ref{sec:FE}
we discuss the functional equations.
These functional equations were obtained by F. Sato \cite{fsatoe}
in a general setting and we apply his result.
We also introduce zeta functions
$\xi(s,f)$ associated to $G_N$-relative invariant functions $f$, and 
express them in terms of
orbital $L$-functions.

Later sections develop the more specific theory for $(G,V)$.
Let $p$ be a prime.
In Section \ref{sec:orbit} we give orbit descriptions
over $\Z/p\Z$ and $\Z/p^2\Z$ which have arithmetic meanings.
In Section \ref{sec:gausssum}, we discuss the orbital Gauss sums
over $\Z/p\Z$ and $\Z/p^2\Z$ in detail. Over $\Z/p\Z$, this was
studied by S. Mori \cite{mori} and we recall his result.
After that, we compute various orbital Gauss sums
over $\Z/p^2\Z$ for $p\neq 2,3$, which is the main technical contribution
of this paper, and we prove Theorem \ref{thm:intro_fourier} as a consequence.

In Section \ref{sec:ON} we study ``divisible'' zeta functions
and prove Theorem \ref{thm:intro_divisible}.

In Section \ref{sec:residue} we study the residues of
the orbital $L$-functions $\xi(s,\chi,a)$.
With a natural
choice of the test function $\Phi_a$, Wright's
adelic zeta function \cite{wright}
gives an integral expression for $\xi(s,\chi,a)$.
Using the residue formula in \cite{wright}
we compute the residues of $\xi(s,\chi,a)$ for the cases of our interest.
The method as well as many of these results are due to
Datskovsky and Wright \cite{dawra}, and
our results refine theirs as needed for our application to Theorem \ref{thm:introsccchi}.
These computations lead to an explicit version of Theorem \ref{thm:introL}.
In Section \ref{sec:example}, we apply these results to prove
residue formulas for $\xi(s,\chi)$ (Theorem \ref{thm:standardL})
and we prove that class numbers of binary cubic forms are biased
in arithmetic progressions (Theorem \ref{thm:bias_classnumber}).
We also compare our results to Theorem \ref{thm:introsccchi}.

\bigskip
\noindent
{\bf Acknowledgement.}
The first author thanks Shingo Mori for many discussions,
which became the starting point of this work.

\bigskip
\noindent
{\bf Notation.}
For a finite set $X$, we denote by $|X|$ its cardinality.
For a variety $V$ defined over $\Z$ and a ring $R$,
the set of $R$-rational points is denoted by $V_R$
rather than $V(R)$. The trivial Dirichlet character
is denoted by $\bf1$. Our notation mostly matches
the companion paper \cite{scc}, but there are a few exceptions:
the dual vector space $V^\ast$ and
the zeta function $\xi^\ast(s)$ for $V^\ast$
in this paper are denoted by $\widehat V$ and $\widehat\xi(s)$
in \cite{scc}, respectively.
Also, $\xi_m(s)$ in this paper denotes the $m$-divisible zeta function,
while $\xi_q(s)$ in \cite{scc} denotes the $q$-nonmaximal zeta function.
We hope this does not confuse the reader.

\section{The space of binary cubic forms and cubic rings}
\label{sec:GVdefn}
In this section, we introduce the space of binary cubic forms $V$
and its dual space $V^\ast$, with natural actions of
$G=\gl_2$. We regard all of these spaces over $\Z$.
After discussing their basic properties, we recall the
Delone-Faddeev correspondence relating cubic forms to cubic rings.

Let $G=\gl_2$ and
\[
V=\{x(u,v)=x_1u^3+x_2u^2v+x_3uv^2+x_4v^3\mid x_1,x_2,x_3,x_4\in\aff\}.
\]
We identify $V$ with the four dimensional affine space $\aff^4$ via
$x(u,v)=x_1u^3+x_2u^2v+x_3uv^2+x_4v^3\leftrightarrow x=(x_1,x_2,x_3,x_4)$.
We consider the following twisted action of $G$ on $V$:
\[
(g\cdot x)(u,v)=\frac{1}{\det g}x(\alpha u+\gamma v, \beta u+\delta v),
\qquad
x\in V,
\quad
g=\begin{pmatrix}\alpha&\beta\\\gamma&\delta\\\end{pmatrix}\in G.
\]
We note that the scalar matrices act on $V$
by the usual scalar multiplication.
In terms of coordinates, the action is given by
\[
{}^t(g\cdot x)=
\frac{1}{\alpha \delta -\beta \gamma }
\begin{pmatrix}
\alpha ^3		&\alpha ^2\beta 		
&\alpha \beta ^2		&\beta ^3		\\
3\alpha ^2\gamma 		&\alpha ^2\delta +2\alpha \beta \gamma 	
&\beta ^2\gamma +2\alpha \beta \delta 	&3\beta ^2\delta 		\\
3\alpha \gamma ^2		&\beta \gamma ^2+2\alpha \gamma \delta 	
&\alpha \delta ^2+2\beta \gamma \delta 	&3\beta \delta ^2		\\
\gamma ^3		&\gamma ^2\delta 		
&\gamma \delta ^2		&\delta ^3		\\
\end{pmatrix}
\begin{pmatrix} x_1\\x_2\\x_3\\x_4\end{pmatrix},
\quad
g=\begin{pmatrix}\alpha&\beta\\\gamma&\delta\end{pmatrix}.
\]
We usually omit $\cdot$ from $g\cdot x$ and write $gx$ instead.
Let
\[
P(x)=x_2^2x_3^2+18x_1x_2x_3x_4-4x_1x_3^3-4x_2^3x_4-27x_1^2x_4^2,
\]
which is the discriminant of $x(u,v)$.
Then $P(gx)=(\det g)^2P(x)$.

The dual representation
$V^\ast$ is the space of linear forms on $V$.
In the dual coordinate system on $V^\ast$, we
express elements of $V^\ast$ as $y=(y_1,y_2,y_3,y_4)$.
We denote the canonical pairing of $V$ and $V^\ast$
by $[x,y]$, so that $[x,y]=x_1y_1+x_2y_2+x_3y_3+x_4y_4$.
We define the left action of $G$ on $V^\ast$
via $[x,g\ast y]=[(\det g)g^{-1}\cdot x,y]$.
Then the scalar matrices act
by the usual scalar multiplication on $V^\ast$ also.
In terms of coordinates, we have
\[
{}^t(g\ast y)=
\frac{1}{\alpha \delta -\beta \gamma }
\begin{pmatrix}
\delta ^3		&-3\gamma \delta ^2		
&3\gamma ^2\delta 		&-\gamma ^3		\\
-\beta \delta ^2		&\alpha \delta ^2+2\beta \gamma \delta 	
&-(\beta \gamma ^2+2\alpha \gamma \delta )	&\alpha \gamma ^2		\\
\beta ^2\delta 		&-(\beta ^2\gamma +2\alpha \beta \delta )	
&\alpha ^2\delta +2\alpha \beta \gamma 	&-\alpha ^2\gamma 		\\
-\beta ^3		&3\alpha \beta ^2		
&-3\alpha ^2\beta 		&\alpha ^3		\\
\end{pmatrix}
\begin{pmatrix} y_1\\y_2\\y_3\\y_4\end{pmatrix},
\quad
g=\begin{pmatrix}\alpha&\beta\\\gamma&\delta\end{pmatrix}.
\]
We usually omit $\ast$ from $g\ast y$ and write $gy$ instead.
Let
\[
P^\ast(y)=3y_2^2y_3^2+6y_1y_2y_3y_4-4y_1y_3^3-4y_2^3y_4-y_1^2y_4^2.
\]
Then $P^\ast(gy)=(\det g)^{2}P^\ast(y)$.

We recall an embedding of $V^\ast$ into $V$
which is compatible with
the action of $G$. Let
\[
\iota\colon V^\ast\ni (y_1,y_2,y_3,y_4)\mapsto (y_4,-3y_3,3y_2,-y_1)
\in V.
\]
Then we have $\iota (g\ast y)=g\cdot\iota(y)$.
Hence if $3$ is not a zero divisor in a ring $R$,
we can realize $V^\ast_R$ as a $G_R$-submodule of $V_R$.
Moreover, if $3$ is invertible in $R$ then
we can identify $V^\ast_R$ with $V_R$ in terms of $\iota$.
We use this identification in Section \ref{sec:gausssum}.
Under this identification, 
the bilinear form on $V$ induced by the pairing $[\ast,\ast]$
is given as
\begin{equation}\label{eq:bilin_form}
[x,x']=x_4x_1'-\frac13x_3x_2'+\frac13x_2x_3'-x_1x_4',
\qquad
x,x'\in V,
\end{equation}
and this satisfies $[gx,gx']=(\det g)[x,x']$.
We also note that $P(\iota(y))=27P^\ast(y)$.

We now recall the {\em Delone-Faddeev correspondence},
which gives a ring theoretic interpretation
of rational orbits of $(G,V)$.
This was first given by Delone and Faddeev \cite{defa}
for certain restricted cases and recently by Gan, Gross and Savin
\cite{ggs} in full generality.
Let $R$ be an arbitrary ring. A finite $R$-algebra $S$
is called a {\em cubic ring over $R$} if $S$ is
free of rank $3$ as an $R$-module.

\begin{thm}[\cite{defa}, \cite{ggs}]\label{thm:defa}
There is a canonical discriminant preserving bijection
between the set of orbits $G_R\backslash V_R$
and the set of isomorphism classes of cubic rings over $R$.
If $x\in V_R$ corresponds to a cubic ring $S$ over $R$,
then the group of stabilizers $G_{R,x}$ of $x$ in $G_R$ is
isomorphic to $\aut_R(S)$, the group of automorphisms
of $S$ as an $R$-algebra.
Moreover if $x\in V_{R}$ is of the form
$x(u,v)=u^3+bu^2v+cuv^2+dv^3$, then the corresponding
cubic ring is $R[X]/(X^3+bX^2+cX+d)$.
Also if $x$ is of the form $x(u,v)=v(u^2+cuv+dv^2)$,
then the corresponding cubic ring is $R\times R[X]/(X^2+cX+d)$.
\end{thm}

The proof of this theorem is well known.
We simply recall the construction here.
For $x=(x_1,x_2,x_3,x_4)\in V_R$,
the corresponding cubic ring is
the $R$-module $R1\oplus R\omega\oplus R\theta$
with the commutative multiplicative structure
so that $1$ is the multiplicative identity and that
\[
\omega^2=-x_1x_3-x_2\omega+x_1\theta,
\quad
\theta^2=-x_2x_4-x_4\omega+x_3\theta,
\quad
\omega\theta=-x_1x_4.
\]
For more details, see \cite{ggs} for example.

\section{Orbital $L$-functions and orbital Gauss sums}
\label{sec:L}
In this section, we introduce the notion of
{\em orbital $L$-functions} and {\em orbital Gauss sums},
and discuss their most basic properties.
Further analytic properties are studied in later sections.

Let $N$ be a positive integer. We put
\[
G_N:=G_{\Z/N\Z}=\gl_2(\Z/N\Z),
\qquad
V_N:=V_{\Z/N\Z}\cong V_\Z/N V_\Z.
\]
Then $G_N$ acts on $V_N$.
For $a\in V_N$, let $G_{N,a}:=\{g\in G_N\mid ga=a\}$,
the group of stabilizers.
For each $a\in V_N$, $a+NV_\Z$ is invariant
under the action of the principal congruence subgroup
$\Gamma(N):=\ker(\spl_2(\Z)\rightarrow \spl_2(\Z/N\Z))$
of $\spl_2(\Z)$.
Let $\chi$ be a Dirichlet character whose conductor
$m=m(\chi)$ is a divisor of $N$.
As usual, we regard $\chi$ as a character modulo $N$
via the reduction map
$(\Z/N\Z)^\times\rightarrow (\Z/m\Z)^\times$.

For a congruence subgroup $\Gamma$ of ${\rm SL}_2(\Z)$
and a $\Gamma$-invariant subset $X$ of $V_\Z$,
we put
\[
\xi_{\pm}(s,X,\Gamma):=\frac{1}{[{\rm SL}_2(\Z): \Gamma]}\!
\sum_{\substack{x\in{\Gamma}\backslash X\\\pm P(x)>0}}\!
\frac{|\Gamma_x|^{-1}}{|P(x)|^s},
\qquad
\xi(s,X,\Gamma)
:=
\left(
\begin{array}{l}
\xi_+(s,X,\Gamma)\\
\xi_-(s,X,\Gamma)
\end{array}
\right),
\]
where 
$\Gamma_x=\{\gamma\in\Gamma\mid \gamma x=\gamma\}$.
We put
\[
\xi(s):=\xi(s,V_\Z,\spl_2(\Z)).
\]
This is the zeta function introduced and studied
by Shintani \cite{shintania}.
As a generalisation, we introduce the following.
\begin{defn}\label{def:orbitalL}
We define
\begin{align*}
\xi(s,a)&:=\xi(s,a+NV_\Z,\Gamma(N)),\\
\xi(s,\chi,a)&:=\sum_{g\in G_N}\chi(\det g)\xi(s,ga).
\end{align*}
We call $\xi(s,a)$ a {\em partial zeta function}
and $\xi(s,\chi,a)$ an {\em orbital $L$-function}
for the space of binary cubic forms.
\end{defn}
\begin{rem}
We later observe in Propositions \ref{prop:f_inv_splZ} and
\ref{prop:f_inv_L_ogs}
that we can define $\xi(s,\chi,a)$ in terms of only $\spl_2(\Z)$
and not $\Gamma(N)$.
\end{rem}
For these zeta functions, following basic properties hold.
\begin{lem}\label{lem:zeta_basic}
\begin{enumerate}[{\rm (1)}]
\item
For $g\in G_N$, $\xi(s,\chi,ga)=\chi(\det g)^{-1}\xi(s,\chi,a)$.
\item
We have $\sum_{a\in V_N}\xi(s,a)=\xi(s)$.
\item
If $g\in \spl_2(\Z/N\Z)$, then $\xi(s,ga)=\xi(s,a)$.
\end{enumerate}
\end{lem}

\begin{proof}
These are not too difficult to
verify from the definitions. Alternatively, one can apply
Proposition 3.3 of \cite{yt} as follows:
(1) immediately follows from definition,
while (2) follows from Proposition 3.3 (1) and (4) of \cite{yt}.
For (3), first note that the canonical map
$\spl_2(\Z)\rightarrow \spl_2(\Z/N\Z)$ is surjective.
Take any lift $\gamma\in\spl_2(\Z)$ of $g\in\spl_2(\Z/N\Z)$,
then $\gamma(a+NV_\Z)=ga+NV_\Z$ by definition.
By Proposition 3.3 (2) of \cite{yt}, we have
\begin{align*}
\xi(s,ga)&=\xi(s,ga+NV_\Z,\Gamma(N))
=\xi(s,\gamma(a+NV_\Z),\Gamma(N))\\
&=\xi(s,\gamma(a+NV_\Z),\gamma\Gamma(N)\gamma^{-1})
=\xi(s,a+NV_\Z,\Gamma(N))=\xi(s,a).
\end{align*}
Note that $\Gamma(N)$ is a normal subgroup of $\spl_2(\Z)$.
\end{proof}

We put
\[
T:=\left\{t=\begin{pmatrix}1&0\\0&{t_1}\end{pmatrix}\ \vrule\ t_1\in\gl_1\right\}\cong\gl_1
\]
and $T_N=T_{\Z/N\Z}\cong(\Z/N\Z)^\times$.
Then since $G=\spl_2\rtimes T$, we can write
$\xi(s,\chi,a)$ as
\begin{equation}\label{eq:sum_TN}
\xi(s,\chi,a)=\frac{|G_N|}{|T_N|}\sum_{t\in T_N}\chi(\det t)\xi(s,ta).
\end{equation}
Since $T_N$ is an abelian group, by the orthogonality
of characters we have the following.
\begin{prop}\label{prop:partial_orbital}
We have
\[
\xi(s,a)
=|G_N|^{-1}\sum_\chi\xi(s,\chi,a).
\]
Here $\chi$ runs through all the Dirichlet characters
of conductors dividing $N$.
\end{prop}
Hence the study of partial zeta functions is equivalent to that of
orbital $L$-functions. Since orbital $L$-functions are theoretically
more natural to study, we concentrate on these for the rest of this paper.

We put $V^\ast_N=V^\ast_{\Z/N\Z}\cong V^\ast_\Z/NV^\ast_\Z$.
Then $G_N$ acts on $V^\ast_N$ also.
We define the
zeta functions for the dual
$\xi^\ast(s,b), \xi^\ast(s,\chi,b)$
for each $b\in V^\ast_N$
in exactly the same way; letting
\[
\xi^\ast_{\pm}(s,Y,\Gamma):=\frac{1}{[\spl_2(\Z) : \Gamma]}\!
\sum_{\substack{y\in{\Gamma}\backslash Y\\\pm P^\ast(y)>0}}\!
\frac{|\Gamma_y|^{-1}}{|P^\ast(y)|^s},
\qquad
\xi^\ast(s,Y,\Gamma)
:=
\left(
\begin{array}{l}
	\xi^\ast_+(s,Y,\Gamma)\\
	\xi^\ast_-(s,Y,\Gamma)\\
\end{array}\right),
\]
we define
\begin{align*}
\xi^\ast(s,b)&:=\xi^\ast(s,b+NV^\ast_\Z,\Gamma(N)),\\
\xi^\ast(s,\chi,b)&:=\sum_{g\in G_N}\chi(\det g)\xi^\ast(s,gb).
\end{align*}
They satisfy the same properties in Lemma \ref{lem:zeta_basic}.
Namely, $\sum_{b\in V_N^\ast}\xi^\ast(s,b)=\xi^\ast(s)$,
$\xi^\ast(s,gb)=\xi^\ast(s,b)$ for $g\in\spl_2(\Z/N\Z)$,
and $\xi^\ast(s,\chi,gb)=\chi(\det g)^{-1}\xi^\ast(s,\chi,b)$.
Here we put $\xi^\ast(s)=\xi^\ast(s,V^\ast_\Z,\Gamma(1))$.

We note that
$\xi^\ast(s,Y,\Gamma)=27^{s}\xi(s,\iota(Y),\Gamma)$
where $\iota\colon V^\ast_\Z\rightarrow V_\Z$
is the embedding introduced in Section \ref{sec:GVdefn}.
The factor $27^s$ comes from the relation
$P(\iota(y))=27P^\ast(y)$ for $y\in Y\subset V^\ast_\Z$.
Also if $Y$ is defined by congruence conditions modulo $N$ in $V^\ast_\Z$,
then $\iota(Y)$ is determined by congruence conditions modulo $3N$ in $V_\Z$.
Hence it is possible to write $\xi^\ast(s,b)$
(and hence $\xi^\ast(s,\chi,b)$)
in terms of linear combinations of $\xi(s,a)$.

We conclude this section with the definition of
the orbital Gauss sum.
For $a\in V_N, b\in V^\ast_N$, we put
$\langle a,b\rangle:=\exp(2\pi i[a,b]/N)$.
If we emphasize the dependence on $N$,
we write $\langle a,b\rangle_N$ also.
\begin{defn}\label{def:orbitalGausssum}
For $a\in V_N$, $b\in V_N^\ast$ we define
\[
W(\chi,a,b)
:=\sum_{g\in G_N}\chi(\det g)\langle ga,b\rangle
=\sum_{g\in G_N}\chi(\det g)\langle a,gb\rangle
\]
and call it the {\em orbital Gauss sum}.
\end{defn}
The second equality holds because $[a,gb]=[(\det g)g^{-1}a,b]$.
If we emphasize the dependence on $N$, we write $W_N(\chi,a,b)$ also.
The following immediately follows from the definition.
\begin{lem}
For $g_1,g_2\in G_N$,
\[
W(\chi,g_1a,g_2b)=\chi(\det g_1)^{-1}\chi(\det g_2)^{-1}W(\chi,a,b).
\]
In particular, if $\chi\circ\det$ is non-trivial
either on $G_{N,a}$ or $G_{N,b}$, then $W(\chi,a,b)=0$.
\end{lem}
The significance of this character sum will be
clarified in the next section. In particular,
this appears in the functional equation
satisfied by $\xi(s,\chi,a)$ and $\xi^\ast(s,\chi^{-1},b)$.

\section{Functional equation}
\label{sec:FE}
In this section we discuss the functional equation
of the zeta functions.
To begin, we recall Shintani's functional equation \cite{shintania}.
\begin{thm}[Shintani]\label{thm:FEShintani}
Let
\[
M(s):=
\frac{3^{3s-2}}{2\pi^{4s}}
\Gamma(s)^2\Gamma\left(s-\tfrac16\right)\Gamma\left(s+\tfrac16\right)
\begin{pmatrix}
\sin 2\pi s&\sin \pi s\\
3\sin \pi s&\sin2\pi s\\
\end{pmatrix}.
\]
Then
\[
\xi(1-s)=M(s)\cdot\xi^\ast(s).
\]
\end{thm}
An extension necessary for us was given by F. Sato \cite{fsatoe}
in a general setting.
We recall his formula and apply it to our orbital $L$-functions.
Let $C(V_N)$ and $C(V^\ast_N)$ be the space of
$\C$-valued functions on $V_N$ and $V^\ast_N$, respectively.
\begin{defn}\label{defn:zetaf}
For $f\in C(V_N)$, $f^\ast\in C(V^\ast_N)$ we define
the associated zeta functions as follows:
\[
\xi(s,f):=\sum_{a\in V_N}f(a)\xi(s,a),
\quad
\xi^\ast(s,f^\ast):=\sum_{b\in V^\ast_N}f^\ast(b)\xi^\ast(s,b).
\]
\end{defn}
This is the most general class of the zeta functions
from our viewpoint and this class
contains partial zeta functions and orbital $L$-functions
as special cases.
Our main interest is the orbital $L$-function but the functional
equation is described most naturally for this class.
For $f\in C(V_N)$, we define its Fourier transform
$\widehat f\in C(V^\ast_N)$ by
\[
\widehat f (b) =N^{-4}\sum_{a\in V_N}f(a)\langle a,b\rangle.
\]
By the Fourier inversion formula, we have
$f(a)=\sum_{b\in V^\ast_N}\widehat f(b)\langle -a,b\rangle$.

By \cite[Theorem Q]{fsatoe}, we have the following
functional equation. 
\begin{thm}[F. Sato]\label{thm:FESato}
We have
\[
\xi(1-s,f)=N^{4s}M(s)\cdot\xi^\ast(s,\widehat f).
\]
\end{thm}
\begin{proof}[Sketch of proof]
For the convenience of the reader, we briefly review Sato's proof.
Let $\sS(V_\R)$ be the space of Schwarz-Bruhat functions on $V_\R$.
We put $G_\R^+=\{g\in G_\R\mid \det g>0\}$ and denote by
$dg_\infty$ a fixed Haar measure on it.
We define the (vector valued) local zeta function
\begin{equation}\label{eqn:vv_zeta}
\Gamma_\infty(\Phi_\infty,s):=
\left(
\int_{G_\R^+}|P(g_\infty x_+)|_\infty^s\Phi_\infty(g_\infty x_+)dg_\infty,
\int_{G_\R^+}|P(g_\infty x_-)|_\infty^s\Phi_\infty(g_\infty x_-)dg_\infty
\right),
\end{equation}
where $\Phi_\infty\in\sS(V_\R)$.
Here $x_\pm\in V_\R^\pm=\{x\in V_\R\mid \pm P(x)>0\}$ is arbitrary.

As in the proof of \cite[Theorem 5]{shintanib},
we can take $\Phi_\infty\in\sS(V_\R)$ such that
$\Phi_\infty$ vanishes on $\{x\in V_\R\mid P(x)=0\}$ and 
$\widehat\Phi_\infty$ vanishes on $\{y\in V^\ast_\R\mid P^\ast(y)=0\}$,
where we put $\widehat\Phi_\infty(y)=\int_{V_\R}\Phi(x)\exp(-2\pi i[x,y])dx$.
For $x\in V_\Z, y\in V^\ast_\Z$,
let $f(x)=f(x\mod N)$ and $\widehat f(y)=\widehat f(y\mod N)$.
Then by the standard unfolding method\footnote{Though
we use another integral expression
for $\xi(s,\chi,a)$ to compute residue formulas in Section \ref{sec:residue},
we find this more convenient to prove the functional equation.}
(see the proof of Proposition \ref{prop:unfold} for detail),
\[
\int_{G_\R^+/\Gamma(N)}
	|\det g_\infty|_\infty^{2s}\sum_{x\in V_\Z}f(x)\Phi_\infty(g_\infty x)dg_\infty
=\Gamma_\infty(\Phi_\infty,s)\xi(s,f).
\]
On the other side, by the Poisson summation formula,
\[
\sum_{x\in V_\Z}f(x)\Phi_\infty(g_\infty x)
=|\det g_\infty|_\infty^{-2}\sum_{y\in V^\ast_\Z}\widehat f(y)\widehat\Phi_\infty((\det g_\infty)^{-1}g_\infty y/N).
\]
The rest of argument is standard in the theory of prehomogeneous vector space
and we omit the details.
\end{proof}

Hence the study of the Fourier transform $\widehat f$
is fundamental for further analysis of
the functional equation.

Let $N=N_1N_2$, where $N_1$ and $N_2$ are coprime integers.
For $f^i\in C(V_{N_i})$ for $i=1,2$,
we define $f^1\times f^2\in C(V_N)$ by
\[
(f^1\times f^2)(a)=f^1(a\mod N_1)f^2(a\mod N_2),
\qquad
a\in V_N.
\]
We use the similar notation for the dual space.
We note that the Fourier transform
$(f^1\times f^2)^\wedge$ of $f^1\times f^2$
does {\em not} coincide with $\widehat{f^1}\times\widehat{f^2}$ in general.
Instead, we have the following.
For $t\in(\Z/N\Z)^\times$, let $f_t(a)=f(ta)$.
\begin{lem}\label{lem:productf}
With the notation above, we have
\[
(f^1\times f^2)^\wedge=(f^{1}_{N_2})^\wedge\times(f^{2}_{N_1})^\wedge.
\]
\end{lem}
\begin{proof}
Let $f=f^1\times f^2$. Then
\[
f^\wedge(b)=N_1^{-4}N_2^{-4}\sum_{a\in V_{N_1N_2}}f(a)\langle a,b\rangle_{N_1N_2}.
\]
By the Chinese remainder theorem, the set of $a\in V_{N_1N_2}$
is equal to the set of $N_2a_1+N_1a_2$ for $a_1\in V_1, a_2\in V_2$.
Therefore, the quantity above is
\begin{align*}
&N_1^{-4}N_2^{-4}\sum_{a_1\in V_{N_1}}\sum_{a_2\in V_{N_2}}f(N_2a_1+N_1a_2)\langle N_2a_1+N_1a_2,b\rangle_{N_1N_2}
\\
&\quad=N_1^{-4}N_2^{-4}\sum_{a_1\in V_{N_1}}\sum_{a_2\in V_{N_2}}f^1(N_2a_1)f^2(N_1a_2)\langle N_2a_1,b\rangle_{N_1N_2}\langle N_1a_2,b\rangle_{N_1N_2}.
\end{align*}
We have $\langle N_2a_1,b\rangle_{N_1N_2}=\langle a_1,b\rangle_{N_1}$
and $\langle N_1a_2,b\rangle_{N_1N_2}=\langle a_2,b\rangle_{N_2}$,
where we reduce $b$ modulo $N_1$ and $N_2$, respectively,
and the result follows.
\end{proof}

Let $N=\prod N_i$ where $N_i$ and $N_j$ are coprime for $i\neq j$
and $f^i\in C(V_{N_i})$.
Then by the repeated use of this lemma, we have
\[
(\prod f^i)^\wedge=\prod (f^{i}_{N/N_i})^\wedge,
\]
where $\prod f^i\in C(V_N)$ is defined similarly.

We now introduce the following space of functions on $V_N$,
which are our main interest.

\begin{defn}
We define
\[
C(V_N,\chi):=
\left\{f\colon V_N\rightarrow\C
	\mid\text{$f(ga)=\chi(\det g)f(a)$ for all $g\in G_N, a\in V_N$}\right\},
\]
the space of $G_N$-relative invariant function with respect to $\chi$.
\end{defn}

Let $f\in C(V_N,\chi)$ and consider the associated zeta function $\xi(s,f)$.
For $x\in V_\Z$, let $f(x)=f(x\mod N)$. We note that for
$\gamma\in\spl_2(\Z)$,
$f(\gamma x)=f(x)$ since
$(\gamma\mod N)\in\spl_2(\Z/N\Z)$ is of determinant $1$.
For $f\in C(V_N,\chi)$, the zeta function $\xi(s,f)$
has the following description.
\begin{prop}\label{prop:f_inv_splZ}
For $f\in C(V_N,\chi)$,
\[
\xi_{\pm}(s,f)
=\sum_{\substack{x\in\spl_2(\Z)\backslash V_\Z\\\pm P(x)>0}}
f(x)\frac{|\spl_2(\Z)_x|^{-1}}{|P(x)|^s}
\quad
\text{where}
\quad
\xi(s,f)=
\begin{pmatrix}\xi_+(s,f)\\\xi_-(s,f)\end{pmatrix}.
\]
\end{prop}
\begin{proof}
For $X\subset V_\Z$, let $X^\pm=\{x\in X\mid \pm P(x)>0\}$.
Then this follows from
\begin{align*}
\sum_{x\in\spl_2(\Z)\backslash V_\Z^\pm}f(x)\frac{|\spl_2(\Z)_x|^{-1}}{|P(x)|^s}
&=[\spl_2(\Z):\Gamma(N)]^{-1}
\sum_{x\in\Gamma(N)\backslash V_\Z^\pm}f(x)\frac{|\Gamma(N)_x|^{-1}}{|P(x)|^s}\\&=[\spl_2(\Z):\Gamma(N)]^{-1}\sum_{a\in V_N}f(a)
\sum_{x\in\Gamma(N)\backslash (a+NV_\Z)^\pm}\frac{|\Gamma(N)_x|^{-1}}{|P(x)|^s}\\
&=\xi_\pm(s,f).
\vspace*{-2cm}
\end{align*}
\end{proof}

All the zeta functions we study in the rest of this paper
will be of the form $\xi(s,f)$ for some $f\in C(V_N,\chi)$.
We explain how $\xi(s,f)$ and $\widehat f$ for these $f$ are related to
the orbital $L$-functions and the orbital Gauss sums
introduced in Section \ref{sec:L}.

For $a\in V_N$, we define $f_{\chi,a}\in C(V_N)$ as follows:
If $\chi\circ\det$ is trivial on $G_{N,a}$, we define
\[
f_{\chi,a}(a'):=
\begin{cases}
|G_{N,a}|\chi(\det g)& a'=ga, g\in G_N,\\
0	&a'\notin G_N\cdot a.
\end{cases}
\]
Otherwise, we put $f_{\chi,a}=0$.
Then one can easily see that
\[
\xi(s,f_{\chi,a})=\xi(s,\chi,a)
\quad\text{and}\quad
\widehat{f_{\chi,a}}(b)=N^{-4}W(\chi,a,b).
\]
We also check that $f_{\chi,a}\in C(V_N,\chi)$,
and moreover if $f\in C(V_N,\chi)$, we have
\[
f=|G_N|^{-1}\sum_{a\in V_N}f(a)f_{\chi,a}.
\]
Hence we have the following.
\begin{prop}\label{prop:f_inv_L_ogs}
Let $f\in C(V_N,\chi)$. Then
\begin{align*}
\xi(s,f)&=|G_N|^{-1}\sum_{a\in V_N}f(a)\xi(s,\chi,a),\\
\widehat f(b)&=N^{-4}|G_N|^{-1}\sum_{a\in V_N}f(a)W(\chi,a,b).
\end{align*}
\end{prop}

So we will study $\xi(s,\chi,a)$ and $W(\chi,a,b)$,
and then apply the results
to prove analytic properties of $\xi(s,f)$.

For its own interest, we describe the functional equation
of the orbital $L$-functions.
\begin{prop}\label{prop:FEorbital}
We have
\begin{align*}
\xi(1-s,\chi,a)
&=N^{4s-4}M(s)\sum_{b\in G_N\backslash V^\ast_N}
\frac{W(\chi,a,b)}{|G_{N,b}|}\xi^\ast(s,\chi^{-1},b)\\
&=\frac{N^{4s-4}}{|G_N|}M(s)\sum_{b\in V^\ast_N}
{W(\chi,a,b)}\xi^\ast(s,\chi^{-1},b).
\end{align*}
\end{prop}
\begin{proof}
By applying $f_{\chi,a}\in C(V_N)$ to Theorem \ref{thm:FESato},
we have
\begin{align*}
\xi(1-s,\chi,a)
&=\frac{N^{4s-4}}{|G_N|}M(s)\sum_{b\in V^\ast_N}
{W(\chi,a,b)}\xi^\ast(s,b)\\
&=\frac{N^{4s-4}}{|G_N|}M(s)\sum_{b\in G_N\backslash V^\ast_N}
\frac{1}{|G_{N,b}|}\sum_{g\in G_N}{W(\chi,a,gb)}\xi^\ast(s,gb).
\end{align*}
Since $W(\chi,a,gb)=\chi(\det g)^{-1}W(\chi,a,b)$,
we have the first equality.
Since the product $W(\chi,a,b)\cdot\xi^\ast(s,\chi^{-1},b)$
depends only on the $G_N$-orbits of $b$, we have the second identity.
\end{proof}

We discuss some
general properties of the orbital Gauss sums.
The following is easy to prove.
\begin{lem}\label{lem:ogs_reduction}
Let $m\mid N$, $a\in V_{N/m}$, $b\in V_N$
and let $\chi$ be a Dirichlet character whose conductor
is a divisor of $N/m$.
Regarding $ma \in V_N$, we have
\[
W_N(\chi,ma,b)=\frac{|G_{N}|}{|G_{N/m}|}W_{N/m}(\chi,a,b).
\]
\end{lem}
\begin{proof}
Since the action of $G_{N}$ on $ma\in V_N$ factors
through $G_{N}\rightarrow G_{N/m}$, we have
\[
W_{N}(\chi,ma,b)=\frac{|G_{N}|}{|G_{N/m}|}
\sum_{g\in G_{N/m}}\chi(\det g)\langle gma,b\rangle_N.
\]
Since $\langle gma,b\rangle_N=\langle ga,b\rangle_{N/m}$,
we have the formula.
\end{proof}

We give two further properties of orbital Gauss sums,
which are generalizations of those of the classical Gauss sum
$\tau_N(\chi)=\sum_{t\in(\Z/N\Z)^\times}\chi(t)\exp(2\pi i t/N)$.
We do not use these results in this paper,
but mention them because of their own interest.
Recall that $\tau_N(\chi^{-1})\tau_N(\chi)=\chi(-1)N$
if the conductor of $\chi$ is $N$, and $\tau_N(\chi)=0$ otherwise
(see e.g. (3.15) of \cite{iwko}).
As a generalization, our Gauss sum satisfies the following.
\begin{prop}
Assume that $\chi\circ\det$ is trivial on $G_{N,a}$. Then
\[
\frac{1}{N^4}\sum_{b\in G_N\backslash V^\ast_N}
\frac{W(\chi^{-1},-a',b)}{|G_{N,b}|}
\frac{W(\chi,a,b)}{|G_{N,a}|}
=
\begin{cases}
\chi(\det g)&a'=ga,\\
0&a'\not\in G_N\cdot a.
\end{cases}
\]
\end{prop}
\begin{proof}
By the Fourier inversion formula,
\begin{align*}
f_{\chi,a}(a')
&=\sum_{b\in V^\ast_N}\widehat{f_{\chi,a}}(b)
\langle -a',b\rangle
=N^{-4}\sum_{b\in V^\ast_N}W(\chi,a,b)\langle -a',b\rangle\\
&=N^{-4}\sum_{b\in G_N\backslash V^\ast_N}
\frac{1}{|G_{N,b}|}\sum_{g\in G_N}W(\chi,a,gb)\langle -a',gb\rangle\\
&=N^{-4}\sum_{b\in G_N\backslash V^\ast_N}
\frac{W(\chi,a,b)}{|G_{N,b}|}\sum_{g\in G_N}\chi^{-1}(\det g)\langle -a',gb\rangle\\
&=N^{-4}\sum_{b\in G_N\backslash V^\ast_N}
\frac{W(\chi,a,b)}{|G_{N,b}|}W(\chi^{-1},-a',b).
\end{align*}
This is equivalent to the desired formula.
\end{proof}

We next describe the decomposition formula.
Assume $N=\prod_{1\leq i\leq r}N_i$
where $(N_i,N_j)=1$ if $i\neq j$.
Using the
canonical isomorphism
$(\Z/N\Z)^\times\cong\prod (\Z/N_i\Z)^\times$,
we obtain a character $\chi_i$ on $(\Z/N_i\Z)^\times$
by restricting $\chi$.
Then
$\tau_N(\chi)=\prod_{1\leq i\leq r}\chi_i(N/N_i)\tau_{N_i}(\chi_i)$
(see e.g. (3.16) of \cite{iwko}).
This is generalized as follows.
\begin{prop}\label{prop:ogs_decomposition}
For $a\in V_N$ and $b\in V^\ast_N$,
let $a_i=(a\mod N_i)\in V_{N_i}$ and $b_i=(b\mod N_i)\in V^\ast_{N_i}$,
respectively.
Then
\[
W_N(\chi,a,b)=\prod_{1\leq i\leq r}\chi_i(N/N_i)^2W_{N_i}(\chi_i,a_i,b_i).
\]
\end{prop}
\begin{proof}
Let $f=f_{\chi,a}\in C(V_N)$ and $f_i=f_{i,\chi_i,a_i}\in C(V_{N_i})$
for each $i$.
Then $f=\prod_{1\leq i\leq r}f_{i}$. Since
$(f_{i,N/N_i})^\wedge=\chi_i(N/N_i)^2(f_i)^\wedge$,
the result follows from Lemma \ref{lem:productf}.
\end{proof}

\section{Orbit description}\label{sec:orbit}
For the specific study of the orbital $L$-functions
and orbital Gauss sums, it is indispensable to describe
the $G_N$-orbit structure of $V_N$ explicitly.
By the Chinese remainder theorem, this is reduced
to the case when $N$ is a prime power. Let $p$ be a prime.
In this section, we study these orbit structures when $N=p$ and $N=p^2$.
The theory over $\Z/p\Z$ is the base case and is well known, and
the theory over $\Z/p^2\Z$ is a refinement of this. Besides its own interest,
this is significant in the study of cubic fields,
because the maximality criterion
of cubic rings $R$ over $\Z$ at $p$ is given in terms of congruence conditions
of the coefficients modulo $p^2$ of the corresponding integral
binary cubic forms $x\in V_\Z$.
We will in fact prove this criterion in Proposition \ref{prop:maximality}.
\subsection{The case $N=p$}
Let $\Sigma_p$ be the following set of symbols;
\[
\Sigma_p=\{(3), (21), (111), (1^21), (1^3), (0)\}.
\]
For each $(\sigma)\in\Sigma_p$, we define $V_p(\sigma)\subset V_p$
as follows:
\begin{align*}
V_p(3)&=\{a\in V_p\mid \text{$a(u,v)$ has no rational roots in $\mathbb P^1_{\mathbb F_p}$}\},\\
V_p(21)&=\{a\in V_p\mid \text{$a(u,v)$ has only one rational root in $\mathbb P^1_{\mathbb F_p}$}\},\\
V_p(111)&=\{a\in V_p\mid \text{$a(u,v)$ has three distinct rational roots in $\mathbb P^1_{\mathbb F_p}$}\},\\
V_p(1^21)&=\{a\in V_p\mid \text{$a(u,v)$ has one single root and one double root in $\mathbb P^1_{\mathbb F_p}$}\},\\
V_p(1^3)&=\{a\in V_p\mid \text{$a(u,v)$ has one triple root in $\mathbb P^1_{\mathbb F_p}$}\},\\
V_p(0)&=\{0\}.
\end{align*}
So each cubic form in $V_p(3)$ is irreducible over $\mathbb F_p$,
while each cubic form in $V_p(21)$ is
the product of a linear form and an irreducible quadratic
form over $\mathbb F_p$.
We have
$\{a\in V_p\mid P(a)=0\}=V_p(1^21)\sqcup V_p(1^3)\sqcup V_p(0)$.
We say $a\in V_p$ is of type $(\sigma)$ if $a\in V_p(\sigma)$.
The description of rational orbits over a field is well known
(see e.g., \cite[Proposition 2.1]{wright}).
Each $V_p(\sigma)$ consists of a single $G_p$-orbit.
Moreover, under the Delone-Faddeev correspondence in Theorem \ref{thm:defa},
the corresponding cubic rings over $\mathbb F_p$ are
$\mathbb F_{p^3}$, $\mathbb F_{p^2}\times\mathbb F_{p}$,
$(\mathbb F_p)^3$, $\mathbb F_p\times\mathbb F_{p}[X]/(X^2)$,
$\mathbb F_p[X]/(X^3)$ and $\mathbb F_p[X,Y]/(X^2,XY,Y^2)$,
respectively.
Note that $(0,1,0,0)\in V_p(1^21)$ and $(1,0,0,0)\in V_p(1^3)$.
We always use them as orbital representatives of these orbits.
The structure of the stabilizers for each orbit is given as follows.

\begin{lem}\label{lem:stabilizer_p}
The group of stabilizers of $G_{p,a}$ of $a$ of type
$(3), (21), (111), (1^21), (1^3)$ and $(0)$ are
isomorphic to $\Z/3\Z, \Z/2\Z, \mathfrak S_3,
\mathbb F_p^\times, \mathbb F_p^\times\ltimes\mathbb F_p$,
and $G_p$, respectively.
Moreover, for $(0,1,0,0)\in V_p(1^21)$ and
$(1,0,0,0)\in V_p(1^3)$, the stabilizers are respectively given by
\[
\left\{\begin{pmatrix}1&0\\0&t\end{pmatrix}\ \vrule t\in\mathbb F_p^\times\right\}
\quad\text{and}\quad
\left\{\begin{pmatrix}t&x\\0&{t^2}\end{pmatrix}\ \vrule t\in\mathbb F_p^\times, x\in\mathbb F_p\right\}.
\]
\end{lem}
\begin{proof}
If $a\in V_p$ corresponds to a cubic ring $R_a$ over $\mathbb F_p$
under the Delone-Faddeev correspondence,
then $G_{p,a}$ is isomorphic to the automorphism group
$\aut(R_a)$
as an $\mathbb F_p$-algebra. It is easy to see that
$\aut(\mathbb F_{p^3})\cong\Z/3\Z$,
$\aut(\mathbb F_p\times\mathbb F_{p^2})\cong\Z/2\Z$
and $\aut((\mathbb F_p)^3)\cong\mathfrak S_3$.
If $g\in G_p$ stabilizes $a=(0,1,0,0)$,
then $g$ fixes $(1:0), (0:1)\in\mathbb P^1_{\mathbb F_p}$
and hence must be a diagonal matrix. Now an easy computation
determines the stabilizers of $a$. The stabilizers of $(1,0,0,0)$
are similarly determined.
\end{proof}

As a result, we see the cardinality of each orbit.
\begin{lem}\label{lem:n_p(sigma)}
Let $n_p(\sigma)=|V_p(\sigma)|$. Then
\begin{align*}
n_p(3)&=3^{-1}(p^2-1)(p^2-p),
&
n_p(1^21)&=(p^2-1)p,\\
n_p(21)&=2^{-1}(p^2-1)(p^2-p),
&
n_p(1^3)&=(p^2-1),\\
n_p(111)&=6^{-1}(p^2-1)(p^2-p),
&
n_p(0)&=1.
\end{align*}
\end{lem}

Let $\Q_p$ be the $p$-adic field and $\Z_p$ its ring of integers.
For $(\sigma)\in\Sigma_p$, we define
\begin{align*}
V_{\Z_p}(\sigma)&:=\{a\in V_{\Z_p}\mid a\mod p\in V_p(\sigma)\},\\
V_{p^e}(\sigma)&:=\{a\in V_{p^e}\mid a\mod p\in V_p(\sigma)\},
\qquad (e\geq1).
\end{align*}
We say $a\in V_{\Z_p}$ or $V_{p^e}$ is of type $(\sigma)$
if $a\in V_{\Z_p}(\sigma)$ or $V_{p^e}(\sigma)$, respectively.

For the application in Section \ref{sec:ON},
we introduce the following function on $V_p$.

\begin{defn}\label{defn:f_p}
Let $f_p\in C(V_p)$
be the characteristic function of
those $a\in V_p$ with $P(a)=0$.
\end{defn}
Obviously, $f_p\in C(V_p,{\bf 1})$.

\subsection{{The case $N=p^2$}}
In this subsection we study the $G_{p^2}$-orbit structure
of $V_{p^2}$.
We put $R:=\Z/p^2\Z$,
and also $pR:=\{pu\mid u\in R\}$,
$pR^\times:=\{pu\mid u\in R^\times\}=pR\setminus\{0\}$.

Let $\Sigma_{p^2}$ be the following set of symbols:
\[
\Sigma_{p^2}=\{(3),(21),(111),(1^21_{\rm max}),(1^21_\ast),
(1^3_{\rm max}),(1^3_{\ast}),(1^3_{\ast\ast})\}.
\]
We introduce $V_{p^2}(\sigma)$
for $(\sigma)=(1^21_{\rm max}),(1^21_\ast),
(1^3_{\rm max}),(1^3_{\ast}),(1^3_{\ast\ast})$
which we show that
\begin{align*}
V_{p^2}(1^21)&=V_{p^2}(1^21_{\max})\sqcup V_{p^2}(1^21_\ast),\\
V_{p^2}(1^3)&=
V_{p^2}(1^3_{\max})\sqcup V_{p^2}(1^3_\ast)\sqcup V_{p^2}(1^3_{\ast\ast}).
\end{align*}
Let
\begin{align*}
\cD_{p^2}{(1^21_{\rm max})}
&:=\{(0,a_2,a_3,a_4)\in V_{p^2}
	\mid a_2\in R^\times, a_3\in pR,a_4\in pR^\times\},\\
\cD_{p^2}(1^21_\ast)
&:=\{(0,a_2,a_3,a_4)\in V_{p^2}\mid a_2\in R^\times, a_3\in pR, a_4=0\},\\
\cD_{p^2}{(1^3_{\rm max})}
&:=\{(a_1,a_2,a_3,a_4)\in V_{p^2}
	\mid a_1\in R^\times, a_2\in pR,a_3\in pR,a_4\in pR^\times\},\\
\cD_{p^2}{(1^3_{\ast})}
&:=\{(a_1,a_2,a_3,a_4)\in V_{p^2}
	\mid a_1\in R^\times, a_2\in pR,a_3\in pR^\times,a_4=0\},\\
\cD_{p^2}{(1^3_{\ast\ast})}
&:=\{(a_1,a_2,a_3,a_4)\in V_{p^2}
	\mid a_1\in R^\times, a_2\in pR,a_3=a_4=0\},
\end{align*}
and
\[
V_{p^2}(\sigma)
:=G_{p^2}\cdot\cD_{p^2}(\sigma)
\]
for these $(\sigma)$.
We say $a\in V_{p^2}$ is of type
$(\sigma)$ if $a\in V_{p^2}(\sigma)$.
We put $n_{p^2}(\sigma):=|V_{p^2}(\sigma)|$.

One may notice that $\cD_{p^2}{(1^3_{\rm max})}$ is the
set of modulus classes of Eisenstein polynomials multiplied
by units. We will prove in Proposition \ref{prop:maximality}
that orbits containing these modulus classes
correspond to cubic rings those are totally ramified at $p$.
A similar interpretation in terms of partially ramified rings
will be given for $\cD_{p^2}{(1^21_{\rm max})}$ also.
We also give an interpretation of these stratifications
in terms of $p$-adic valuations of $P(x)$
in Proposition \ref{prop:stra_val}.

Let us study these sets.
\begin{lem}
We put
\begin{align*}
\cD_{p^2}(1^21)&=\cD_{p^2}{(1^21_{\rm max})}\sqcup\cD_{p^2}(1^21_\ast),\\
G_{p^2}(1^21)&=
\left\{\begin{pmatrix}s&0\\n&t\end{pmatrix}\in G_{p^2}\ \vrule\ s,t\in R^\times, n\in pR\right\}.
\end{align*}
\begin{enumerate}[{\rm (1)}]
\item We have $V_{p^2}(1^21)=G_{p^2}\cdot\cD_{p^2}(1^21)$.
\item Let $a\in\cD_{p^2}(1^21)$. Then
for $g\in G_{p^2}$, $ga\in\cD_{p^2}(1^21)$
if and only if $g\in G_{p^2}(1^21)$. Moreover,
both $\cD_{p^2}{(1^21_{\rm max})}$ and $\cD_{p^2}(1^21_\ast)$
are $G_{p^2}(1^21)$-invariant.
\item
We have $n_{p^2}(1^21_{\rm max})=p^3(p^2-1)(p^2-p)$,
$n_{p^2}(1^21_{\ast})=p^4(p^2-1)$.
\end{enumerate}
\end{lem}
\begin{proof}
(1) \ 
By definition $V_{p^2}(1^21)\supset G_{p^2}\cD_{p^2}(1^21)$
and we consider the reverse inclusion.
Let $a\in V_p(1^21)$.
Then $(a\mod p)\in V_p$ lies in the $G_p$-orbit of $(0,1,0,0)$.
Hence by changing an element of its $G_{p^2}$-orbit
if necessary, we may assume $a=(a_1,1,a_3,a_4)$ where $a_1,a_3,a_4\in pR$.
Then the cubic form $a(u,v)$ decomposes as
$a(u,v)=(a_1u+v)(u^2+a_3uv+a_4v^2)$ in $R[u,v]$.
Hence for $g=\begin{pmatrix}1&{-a_1}\\0&1\end{pmatrix}$, we have $ga\in\cD_{p^2}(1^21)$.\\
(2) \ Let $a=(0,a_2,a_3,a_4)\in \cD_{p^2}(1^21)$,
$g\in G_{p^2}$ and assume $ga\in\cD_{p^2}(1^21)$.
By considering the reduction modulo $p$,
we see that $g$ must be of the form
\[
g=\begin{pmatrix}s&m\\n&t\end{pmatrix},
\quad s,t\in R^\times, m,n\in pR.
\]
Moreover, since the first coordinate of $ga$ is
$(\det g)^{-1}s^2ma_2$
and also $(\det g)^{-1}s^2a_2\in R^\times$, we have $m=0$. Hence
$g\in G_{p^2}(1^21)$. On the other hand, for this $g$,
we have $ga=(0,sa_2,2na_2+ta_3,s^{-1}t^2a_4)$. Note that $na_3=0$.
Hence
both $\cD_{p^2}{(1^21_{\rm max})}$ and $\cD_{p^2}(1^21_\ast)$
are $G_{p^2}(1^21)$-invariant.\\
(3) \ By (2), we have
\[
n_{p^2}(1^21_{\rm max})
=\frac{|G_{p^2}|}{|G_{p^2}(1^21)|}|\cD_{p^2}(1^21_{\rm max})|
=p^3(p^2-p)(p^2-1).
\]
The number $n_{p^2}(1^21_{\ast})$ is computed similarly.
\end{proof}

\begin{lem}
Let
\begin{align*}
\cD_{p^2}(1^3)&
=\cD_{p^2}{(1^3_{\rm max})}\sqcup\cD_{p^2}(1^3_\ast)\sqcup\cD_{p^2}(1^3_{\ast\ast}),\\
G_{p^2}(1^3)&=
\left\{\begin{pmatrix} s&m\\n&t\end{pmatrix}\in G_{p^2}\ \vrule\ s,t\in R^\times,m\in R,n\in pR\right\}.
\end{align*}
\begin{enumerate}[{\rm (1)}]
\item We have $V_p(1^3)=G_{p^2}\cdot\cD_{p^2}(1^3)$.
\item Let $a\in\cD_{p^2}(1^3)$. Then
for $g\in G_{p^2}$, $ga\in\cD_{p^2}(1^3)$
if and only if $g\in G_{p^2}(1^3)$. Moreover,
all of
$\cD_{p^2}{(1^3_{\rm max})}$,
$\cD_{p^2}(1^3_\ast)$
and $\cD_{p^2}{(1^3_{\ast\ast})}$
are $G_{p^2}(1^3)$-invariant.
\item
We have $n_{p^2}(1^3_{\rm max})=p^2(p^2-1)(p^2-p)$,
$n_{p^2}(1^3_{\ast})=p(p^2-p)(p^2-1)$
and $n_{p^2}(1^3_{\ast\ast})=p^2(p^2-1)$.
\end{enumerate}
\end{lem}
Since the proof is similar to the previous lemma,
we omit the detail.
Note that for $a=(a_1,a_2,a_3,a_4)\in \cD_{p^2}(1^3)$
and $g=\left(\begin{smallmatrix}s&m\\n&t\end{smallmatrix}\right)\in G_{p^2}(1^3)$ (and hence $n\in pR$),
\[
{}^t(ga)=\frac{1}{\det g}
\begin{pmatrix}
s^3a_1+s^2ma_2+sm^2a_3+m^3a_4\\
3s^2na_1+s^2ta_2+2stma_3+3m^2ta_4\\
st^2a_3+3mt^2a_4\\
t^3a_4
\end{pmatrix}.
\]

We summarize the formulas for $n_p^2(\sigma)$
for convenience.
\begin{lem}\label{lem:n_p^2(sigma)}
We have
\[
\begin{array}{rl}
n_{p^2}(3)=&3^{-1}p^4(p^2-1)(p^2-p),\\
n_{p^2}(21)=&2^{-1}p^4(p^2-1)(p^2-p),\\
n_{p^2}(111)=&6^{-1}p^4(p^2-1)(p^2-p),\\
n_{p^2}(1^21_{\rm max})=&p^3(p^2-1)(p^2-p),\\
n_{p^2}(1^3_{\rm max})=&p^2(p^2-1)(p^2-p),
\end{array}
\qquad\text{and}\qquad
\begin{array}{rl}
n_{p^2}(1^21_{\ast})=&p^4(p^2-1),\\
n_{p^2}(1^3_{\ast})=&p(p^2-1)(p^2-p),\\
n_{p^2}(1^3_{\ast\ast})=&p^2(p^2-1).\\
\end{array}
\]
\end{lem}

For $(\sigma)\in\Sigma_{p^2}$ and $e \geq 2$ we put
\begin{align*}
V_{\Z_p}(\sigma):=\{a\in V_{\Z_p}\mid a\mod p^2\in V_{p^2}(\sigma)\},
\quad
V_{p^e}(\sigma):=\{a\in V_{p^e}\mid a\mod p^2\in V_{p^2}(\sigma)\}.
\end{align*}
By the definition of $\cD_{p^2}(\sigma)$, we can easily see the following.
\begin{prop}\label{prop:stra_val}
For $p\neq2$,
\begin{align*}
V_{\Z_p}(1^21_{\rm max})&=\{x\in V_{\Z_p}(1^21)\mid \ord_p(P(x))=1\},\\
V_{\Z_p}(1^21_{\ast})&=\{x\in V_{\Z_p}(1^21)\mid \ord_p(P(x))\geq2\},
\end{align*}
and for $p\neq 2,3$,
\begin{align*}
V_{\Z_p}(1^3_{\rm max})&=\{x\in V_{\Z_p}(1^3)\mid \ord_p(P(x))=2\},\\
V_{\Z_p}(1^3_{\ast})&=\{x\in V_{\Z_p}(1^3)\mid \ord_p(P(x))=3\},\\
V_{\Z_p}(1^3_{\ast\ast})&=\{x\in V_{\Z_p}(1^3)\mid \ord_p(P(x))\geq4\}.
\end{align*}
\end{prop}

We now describe the ring theoretic meaning of these orbits.
We say that a cubic ring over $\Z_p$ is {\em maximal}
if it is not isomorphic to a proper subring of another cubic ring.
We say that a cubic ring $R$ over $\Z$ is {\em maximal at $p$}
if it is not contained in another cubic ring with finite index divisible
by $p$, or equivalently if $R \otimes \Z_p$ is maximal as a cubic
ring over $\Z_p$. We need the following auxiliary lemma.

\begin{lem}
Let $R$ be a cubic ring over $\Z_p$ whose discriminant is zero.
Then $R$ is nonmaximal.
\end{lem}
\begin{proof}
Let $K=R\otimes_{\Z_p}\Q_p$ and regard $R\subset K$.
Since $\Disc(K)=0$, $K$ is isomorphic to
$\Q_p\times\Q_p[X]/(X^2)$, $\Q_p[X]/(X^3)$ or
$\Q_p[X,Y]/(X^2,XY,Y^2)$. In particular, $K$ has a
nilpotent element $x$. Take $y\in R$ such that
$R\cap \Q_p x=\Z_py$ and consider the ring $R[y/p]\subset K$.
Since $y^3=0$, we have $R\subset R[y/p]\subset p^{-2}R$.
This implies that $R[y/p]$ is free of rank $3$ as a $\Z_p$-module.
Since $R$ is a proper subring of $R[y/p]$, $R$ is nonmaximal.
\end{proof}

Applying Theorem \ref{thm:defa}, we have the following interpretation
of the set $V_{p^2}(\sigma)$ for
$(\sigma)=(3),(21),(111),(1^21_{\max})$ and $(1^3_{\max})$
in terms of maximal cubic rings over $\Z_p$.
This also explains why the maximality condition may be detected
modulo $p^2$.

\begin{prop}\label{prop:maximality}
A maximal cubic ring over $\Z_p$ is one of the following:
the integer ring $\co_L$ of the unramified cubic extension $L$ of $\Q_p$,
$\co_F\times\Z_p$ where $\co_F$ is the integer ring of the unramified quadratic
extension $F$ of $\Q_p$, $\Z_p^3$,
the integer ring $\co_{L'}$ of a ramified cubic extension $L'$ of $\Q_p$,
or $\co_{F'}\times\Z_p$ where $\co_{F'}$ is the integer ring of a ramified quadratic extension $F'$ of $\Q_p$.
An element $x\in V_{\Z_p}$ corresponds to the above
$\co_L, \co_F\times\Z_p, \Z_p^3, \co_{L'}$ or $\co_{F'}\times\Z_p$
if and only if $x\in V_{\Z_p}(\sigma)$
where $(\sigma)=(3),(21),(111),(1^3_{\max})$ or
$(1^21_{\rm max})$, respectively.
In particular, $V_{\Z_p}(\sigma)$ is a single $G_{\Z_p}$-orbit
for $(\sigma)=(3),(21)$ and $(111)$.
\end{prop}
\begin{proof}
Let $R$ be a maximal cubic ring over $\Z_p$.
By the lemma above, the discriminant is nonzero.
Hence $R\otimes\Q_p\supset R$ is a separable cubic algebra over $\Q_p$
and so it is a direct product $F_1\times\dots\times F_n$
of field extensions $F_i$ of $\Q_p$ with $\sum_i [F_i:\Q_p]=3$.
Let $\co_{F_i}$ be the integer ring of $F_i$.
Since any elements of $F_i\setminus \co_{F_i}$ generate
$\Z_p$-algebras of infinite rank, any entries of elements
of $R\subset F_1\times\cdots\times F_n$ must be in $\co_{F_i}$.
Hence we have $R\subset\co_{F_1}\times\dots\times\co_{F_n}$.
Since $R$ is maximal, we have
$R=\co_{F_1}\times\dots\times\co_{F_n}$
and the first statement follows.
Let $R=\Z_p\times\co_{F'}$ where $F'$ is a ramified quadratic extension.
Then $\co_{F'}=\Z_p[\theta]$ where $\theta\in F'$ is a root
of an Eisenstein polynomial $X^2+cX+d\in\Z_p[X]$.
Hence $x(u,v)=v(u^2+cuv+dv^2)\in V_{\Z_p}$ corresponds to $R$
by Theorem \ref{thm:defa} and 
we have $(x\mod p^2)\in\cD_{p^2}(1^21_{\rm max})\subset 
V_{p^2}(1^21_{\rm max})$
by the definition of an Eisenstein polynomial.
The other cases are proved similarly.
\end{proof}

Let $V_{p^2}^{\max}\subset V_{p^2}$
be the set of elements of any of the types above.
$V_{p^2}^{\max}$ is defined by
a congruence condition modulo $p^2$ on $V_\Z$,
and it detects cubic rings over $\Z$ maximal at $p$.
Similarly we define $V_{\Z_p}^{\max}\subset V_{\Z_p}$
and $V_{p^e}^{\max}\subset V_{p^e}$.

\begin{defn}\label{defn:f_p^2}
We define the following subsets of $V_{p^2}$:
\begin{align*}
V_{p^2}^{\rm nm}&:
=pV_{p^2}\sqcup V_{p^2}(1^3_{\ast\ast})
\sqcup V_{p^2}(1^3_\ast)\sqcup V_{p^2}(1^21_\ast)=V_{p^2}\setminus V_{p^2}^{\max},\\
\tilde V_{p^2}^{\rm nm}&:
=V_{p^2}^{\rm nm}\sqcup V_{p^2}(1^3_{\rm max}).
\end{align*}
We denote by $\Phi_{p}, \Phi_{p}'\in C(V_{p^2})$
the characteristic functions of $V_{p^2}^{\rm nm}$,
$\tilde V_{p^2}^{\rm nm}$, respectively.
\end{defn}
Obviously $\Phi_{p}, \Phi_{p}'\in C(V_{p^2},{\bf 1})$.
$\Phi_{p}$ detects cubic rings nonmaximal at $p$,
and $\Phi_{p}'$ detects rings nonmaximal or totally ramified at $p$.
The conditions corresponding to these functions were introduced in the seminal
work of Davenport and Heilbronn \cite{DH}, who originally proved the main terms
of Theorem \ref{thm:introscc} by studying the space of binary cubic forms.
These functions play important roles in our proof of Theorem \ref{thm:introscc} as well.

In Section \ref{sec:gausssum}, we will compute the
Fourier transform of $\Phi_{p}$, $\Phi_{p}'$
and hence the related orbital Gauss sums explicitly.
For that purpose,
we study the orbit structure of $V_{p^2}(\sigma)$
for $(\sigma)=(1^21_\ast)$,
$(1^21_{\rm max})$, $(1^3_{\ast\ast})$,
$(1^3_\ast)$, and $(1^3_{\rm max})$ more closely.

\begin{lem}\label{lem:stabilizer_singular_p^2}
\begin{enumerate}[{\rm (1)}]
\item
Assume $p\neq2$.
The stabilizers of elements
$a=(0,1,0,0)\in\cD_{p^2}(1^21_{\ast})$ and
$a'=(0,1,0,a_4)\in\cD_{p^2}(1^21_{\max})$
are respectively given by
\begin{align*}
G_{p^2,a}&=\left\{ \begin{pmatrix}1&0\\0&t\end{pmatrix}\ 
	\vrule\ t\in R^\times\right\},\\
G_{p^2,a'}&=\left\{ \begin{pmatrix}1&0\\0&t\end{pmatrix}\ 
	\vrule\ t\in R^\times,t^2\equiv 1\mod p\right\}.
\end{align*}
Moreover,
$b'=(0,1,0,b_4)\in\cD_{p^2}(1^21_{\max})$ lies in the orbit
$G_{p^2}\cdot a'$ if and only if $b_4=t^2a_4$ for some $t\in R^\times$.
\item
Assume $p\neq3$.
The stabilizers of elements
$a=(1,0,0,0)\in\cD_{p^2}(1^3_{\ast\ast})$,
$a'=(1,0,a_3,0)\in\cD_{p^2}(1^3_{\ast})$ and
$a''=(1,0,0,a_4)\in\cD_{p^2}(1^3_{\max})$
are respectively given by
\begin{align*}
G_{p^2,a}&=\left\{ \twtw tm0{t^2}\ 
	\vrule\ t\in R^\times, m\in R\right\},\\
G_{p^2,a'}&=\left\{ \twtw tm{-2tma_3/3}{t^2+m^2a_3/3}\ 
	\vrule\ t\in R^\times,m\in R,t^2\equiv 1\mod p \right\},\\
G_{p^2,a''}&=\left\{\twtw tm0{t^2}\ 
	\vrule\ t\in R^\times, m\in pR,
t^3\equiv 1\mod p\right\}.
\end{align*}
Moreover,
$b'=(1,0,b_3,0)\in\cD_{p^2}(1^3_{\ast})$ lies in the orbit
$G_{p^2}\cdot a'$ if and only if $b_3=t^2a_3$ for some $t\in R^\times$,
and
$b''=(1,0,0,b_4)\in\cD_{p^2}(1^3_{\max})$ lies in the orbit
$G_{p^2}\cdot a''$ if and only if $b_4=t^3a_4$ for some $t\in R^\times$.
\end{enumerate}
\end{lem}
\begin{proof}
(1)
Let $g\in G_{p^2}$ satisfy
$ga=a$, $ga'=a'$ or $ga'=b'$.
Then $(g \mod p)\in G_p$ stabilizes
$(0,1,0,0)\in V_p$. Hence by Lemma
\ref{lem:stabilizer_p},
$g$ must be of the form
\[
g=\twtw {1+l}mnt,
\qquad
t\in R^\times, l,m,n\in pR.
\]
For this $g$, we have
\[
ga=\frac{1}{\det g}(m,(1+l)^2t,2nt,0),
\qquad
ga'=\frac{1}{\det g}(m,(1+l)^2t,2nt,a_4t^3).
\]
If $ga=a$, then by comparing the first and third entries,
we have $m=n=0$. Note that $2t\in R^\times$ since $p\neq2$.
Now by comparing the second entries,
we have $(\det g)^{-1}(1+l)^2t=1+l=1$ and hence $l=0$.
If $ga'=b'$, by a similar argument
we have $m=n=l=0$ and $b_4=t^2a_4$.
In particular, if $b_4=a_4$,
this holds if and only if
$m=n=l=0$ and $t^2\equiv 1\mod p$.
This proves (1).

(2)
Let $g\in G_{p^2}$ satisfy
$ga=a$, $ga'=a'$, $ga'=b'$, $ga''=a''$, or $ga''=b''$.
Then $(g \mod p)\in G_p$ stabilizes
$(1,0,0,0)\in V_p$. Hence by Lemma
\ref{lem:stabilizer_p},
$g$ must be of the form
\[
g=\twtw tmn{t^2+l},
\qquad
t\in R^\times, m\in R, n,l\in pR.
\]
For this $g$, we have
\begin{align*}
ga&=\frac{1}{\det g}(t^3,3t^2n,0,0),\\
ga'&=\left(
	\frac{t^3+tm^2a_3}{t^3+tl-mn},
	\frac{3t^2n+2t^3ma_3}{t^3+tl-mn},
	\frac{t^5a_3}{t^3+tl-mn},
	0
\right)\\
&=\left(
	1+\frac{tm^2a_3+mn-tl}{t^3},\frac{3n+2ta_3m}{t},t^2a_3,0
\right),\\
ga''&=\frac{1}{\det g}(t^3+m^3a_4,3t^2(n+m^2a_4),3mt^4a_4,t^6a_4).
\end{align*}
By the first formula, $ga=a$ holds if and only if $n=0$ and $l=0$.
Note that $3t^2\in R^\times$ since $p\neq3$.
Similarly by the second formula, $ga'=b'$ holds if and only if
\[
n=-2ta_3m/3,
\quad
l=m^2a_3+mnt^{-1}=m^2a_3/3,
\quad b_3=t^2a_3.
\]
Let $ga''=b''$. Then by comparing the third and second entries,
we have $m\in pR$ and $n=0$. Also since $t^3/(\det g)=1$, $l=0$.
Now comparing the last coefficient, we have $b_4=t^3 a_4$.
Hence we have (2).
\end{proof}
Now we are ready to prove:
\begin{prop}\label{prop:singular_orbits}
\begin{enumerate}[{\rm (1)}]
\item
Let $p\neq2$.
\begin{enumerate}[{\rm (i)}]
\item We have
	$V_{p^2}(1^21_{\ast})=G_{p^2}\cdot(0,1,0,0)$.
\item	The set	$V_{p^2}(1^21_{\max})$ consists of two $G_{p^2}$-orbits and
	each orbit has the same cardinality $2^{-1}|V_{p^2}(1^21_{\max})|$.
	For any $u_1, u_2\in\mathbb F_p^\times$ such that
	$u_1/u_2$ is not a square,
	$(0,1,0,pu_1)$ and $(0,1,0,pu_2)$ are representatives of
	the two orbits.
\end{enumerate}
\item
Let $p\neq3$.
\begin{enumerate}[{\rm (i)}]
\item We have
	$V_{p^2}(1^3_{\ast\ast})=G_{p^2}\cdot(1,0,0,0)$.
\item	If $p\equiv2\mod 3$, 
	$V_{p^2}(1^3_{\rm max})$ consists of single $G_{p^2}$-orbit.
\item	Let $p\equiv1\mod 3$.
	Then $V_{p^2}(1^3_{\rm max})$ consists of three $G_{p^2}$-orbits and
	each orbit has the same cardinality $3^{-1}|V_{p^2}(1^3_{\rm max})|$.
	Let $\{u_1,u_2,u_3\}\subset\mathbb F_p^\times$
	be a set of representatives of
	$\mathbb F_p^\times/(\mathbb F_p^\times)^3$.
	Then we can take $(1,0,0,pu_i), i=1,2,3$ as
	representatives of the three orbits.
\end{enumerate}
\item
Let $p\neq2,3$.
	The set	$V_{p^2}(1^3_{\ast})$ consists of two $G_{p^2}$-orbits and
	each orbit has the same cardinality $2^{-1}|V_{p^2}(1^3_{\ast})|$.
	If $u_1, u_2\in\mathbb F_p^\times$ are such that
	$u_1/u_2$ is not a square, then
	$(1,0,pu_1,0)$ and $(1,0,pu_2,0)$ are representatives of
	the two orbits.
\end{enumerate}
\end{prop}
\begin{proof}
(1) Let $a=(0,1,0,0)\in V_{p^2}(1^21_{\ast})$.
By Lemma \ref{lem:stabilizer_singular_p^2}, we have
$|G_{p^2,a}|=p^2-p$. Hence
\[
|G_{p^2}\cdot a|
=\frac{|G_{p^2}|}{|G_{p^2,a}|}
=\frac{p^4(p^2-p)(p^2-1)}{(p^2-p)}
=p^4(p^2-1)=|V_{p^2}(1^21_{\ast})|
\]
and we have $V_{p^2}(1^21_{\ast})=G_{p^2}\cdot a$.
Let $a'=(0,1,0,a_4)\in V_{p^2}(1^21_{\max})$.
Then we have $|G_{p^2,a'}|=2p$ and hence
we have $|G_{p^2}\cdot a'|=2^{-1}p^3(p^2-p)(p^2-1)
=2^{-1}|V_{p^2}(1^21_{\max})|$.
Hence
$V_{p^2}(1^21_{\ast})$ consists of $2$ orbits.
Combined with Lemma \ref{lem:stabilizer_singular_p^2} (1),
we have (ii).

(2), (3) are proved in the same way.
Let $p\neq 3$ and let $a,a',a''$ be as in 
Lemma \ref{lem:stabilizer_singular_p^2} (2).
Then
\begin{gather*}
|G_{p^2,a}|=p^2(p^2-p),\quad
|G_{p^2,a''}|=
\begin{cases}
p^2 	& p\equiv 2\mod 3,\\
3p^2 	& p\equiv 1\mod 3.\\
\end{cases}
\end{gather*}
If further $p\neq 2$, then $|G_{p^2,a'}|=2p^3$.
The rest of the argument proceeds similarly.
\end{proof}

We conclude this section with supplementary results which 
we use in residue computations.

\begin{prop}\label{prop:orbitalvolume}
Let $R$ be a non-degenerate cubic ring over $\Z_p$
and let $V_{\Z_p,R}\subset V_{\Z_p}$ be the set of elements
corresponding to $R$ under the Delone-Faddeev correspondence.
Normalize the Haar measure on $V_{\Z_p}$ such that
the total volume is $1$.
Then the volume of $V_{\Z_p,R}$ is
$|\aut_{\Z_p}(R)|^{-1}|\Disc(R)|^{-1}(1-p^{-1})(1-p^{-2})$,
where $|\Disc(R)|$ is the discriminant of $R$ as a power of $p$.
\end{prop}
\begin{proof}
Recall that for any $x\in V_{\Q_p}$
with $P(x)\neq0$, $G_{\Q_p}\cdot x$ is an open orbit in $V_{\Q_p}$.
Let $dg$ be the Haar measure on $G_{\Q_p}$ such that
the volume of $G_{\Z_p}$ is $1$.
Then by the computation of the Jacobian determinant
\cite[p.38]{dawra}, we have
\[
\int_{G_{\Q_p}\cdot x} \phi(y)\frac{dy}{|P(y)|_p}
=\frac{(1-p^{-1})(1-p^{-2})}{|G_{\Q_p,x}|}
\int_{G_{\Q_p}}\phi(gx)dg
\]
for an integrable function
$\phi$ on $G_{\Q_p}\cdot x\subset V_{\Q_p}$.
Hence the same computation shows that
\[
\int_{G_{\Z_p}\cdot x} \phi(y)\frac{dy}{|P(y)|_p}
=\frac{(1-p^{-1})(1-p^{-2})}{|G_{\Z_p,x}|}
\int_{G_{\Z_p}}\phi(gx)dg.
\]
Now let $x\in V_{\Z_p,R}$ and let $\phi$ be the characteristic
function of $V_{\Z_p,R}=G_{\Z_p}\cdot x$.
Then since $G_{\Z_p,x}\cong \aut_{\Z_p}(R)$
and $|P(y)|_p=|{\rm Disc}(R)|^{-1}$ for all
$y\in V_{\Z_p,R}$, we have the result.
\end{proof}

\begin{prop}\label{prop:e_enough}
\begin{enumerate}[{\rm (1)}]
\item
For $a\in V_p$ of type $(3),(21),(111)$,
$G_pa+pV_{\Z_p}$ is a single $G_{\Z_p}$-orbit.
\item
For $p\neq2$
and $a\in V_{p^2}$ of type $(1^21_{\rm max})$,
$G_{p^2}a+p^2V_{\Z_p}$ is a single $G_{\Z_p}$-orbit.
\item
For $p\neq3$
and $a\in V_{p^2}$ of type $(1^3_{\rm max})$,
$G_{p^2}a+p^2V_{\Z_p}$ is a single $G_{\Z_p}$-orbit.
\end{enumerate}
\end{prop}
\begin{proof}
(1) follows from Proposition \ref{prop:maximality}.
Alternatively, we can prove this as follows.
Let $\tilde a\in V_{\Z_p}$ be a lift of $a$
and $R$ the corresponding (unramified) cubic ring.
Then obviously $G_{\Z_p}\tilde a\subseteq G_pa+pV_{\Z_p}$.
On the other hand,
by Lemma \ref{lem:stabilizer_p} and Proposition \ref{prop:orbitalvolume}
we see that the volumes of these sets are equal:
\[
\int_{G_pa+pV_{\Z_p}}dx
=p^{-4}|G_{p}a|
=\frac{p^{-4}|G_p|}{|G_{p,a}|}
=\frac{(1-p^{-1})(1-p^{-2})}{|\aut_{\Z_p}(R)|}
=\int_{G_{\Z_p}\tilde a}dx.
\]
Since both $G_{\Z_p}\tilde a$ and $G_pa+pV_{\Z_p}$ are open, they coincide.

(2) and (3) are proved similarly. We recall the classification of
ramified quadratic and cubic extensions of $\Q_p$.
If $p\neq2$, then there are two ramified quadratic extensions
of discriminant $p$. If $p\equiv1\mod 3$, then there are three
cyclic ramified cubic extensions, and if $p\equiv 2\mod 3$
there is a unique non-cyclic ramified cubic extension,
and the discriminants of all these extensions are $p^2$.
Let $R$ be the (maximal ramified) cubic ring
corresponding to a lift $\tilde a$ of $a$.
Since $R$ is maximal in $R\otimes\Q_p$,
$\aut_{\Z_p}(R)\cong\aut_{\Q_p}(R\otimes\Q_p)$.
Lemma \ref{lem:stabilizer_singular_p^2} asserts that
$|G_{p^2,a}|=|\aut_{\Q_p}(R\otimes\Q_p)||\Disc(R)|^{-1}$ for each case
and hence we have the result.
\end{proof}

We will use the following result in Section \ref{subsec:rm}.
We prove this with the aid of PARI/GP \cite{pari}.
\begin{prop}\label{prop:e_enough_3}
For $p=3$
and $a\in V_{p^3}$ of type $(1^3_{\rm max})$,
$G_{p^3}a+p^3V_{\Z_p}$ is a single $G_{\Z_p}$-orbit.
\end{prop}
\begin{proof}
It is known that there are $9$ ramified cubic extensions of $\Q_3$
(see e.g. \cite{joro})
and hence $V_{\Z_p}(1^3_{\max})$ consists of $9$ $G_{\Z_p}$-orbits.
We list a set of representatives in the table of
Proposition \ref{prop:residue_rm_max}.
Using PARI/GP \cite{pari} to explicitly calculate the stablizer group in
$G_{\Z/27\Z}$ for each representative $a\in V_{\Z/27\Z}(1^3_{\max})$,
we confirm the identity
$|G_{p^3,a}|=|\aut_{\Q_p}(R\otimes\Q_p)||\Disc(R)|^{-1}$ as above.
This finishes the proof.
\end{proof}

\section{Computation of singular Gauss sums}
\label{sec:gausssum}
In this section, we explicitly compute the Fourier transforms of
the functions $f_p\in C(V_p,{\bf 1})$ and
$\Phi_{p},\Phi_{p}'\in C(V_{p^2},{\bf 1})$ introduced
in Definitions \ref{defn:f_p} and \ref{defn:f_p^2}.
By Proposition \ref{prop:f_inv_L_ogs}, it suffices 
to evaluate the orbital Gauss sum $W({\bf 1},a,b)$
for all $b$ and for $a$ in the support of these functions.
We call these Gauss sums singular because $P(a)$ is not invertible
for any such $a$.

For $N=p$, these Gauss sums were computed by S. Mori \cite{mori}.
We first review his results and
then study the case $N=p^2$ by extending Mori's approach.

\subsection{The case $N=p$}
In this subsection we will prove the following.
\begin{prop}\label{prop:hat(fp)}
The Fourier transform of $f_p\in C(V_p)$
is given by
\[
\widehat{f_p}(b)
=
\begin{cases}
-p^{-3}			& P^\ast(b)\neq0,\\
p^{-2}-p^{-3}		& P^\ast(b)=0,b\neq0,\\
p^{-1}+p^{-2}-p^{-3}	& b=0.\\
\end{cases}
\]
\end{prop}
By Proposition \ref{prop:f_inv_L_ogs},
\[
\widehat{f_p}(b)
=p^{-4}|G_{p}|^{-1}\sum_{a\in V_{p}}f_{p}(a)W({\bf 1},a,b)
=p^{-4}|G_{p}|^{-1}\sum_{P(a)=0}W({\bf 1},a,b)
\]
and so we are interested in $W({\bf 1},a,b)$ for those
$a$ with $P(a)=0$.
Let $p\neq3$. Then the map $\iota\colon V^\ast_p\rightarrow V_p$
is an isomorphism of $G_p$-modules,
so we may and do identify $V^\ast_p$ with $V_p$.
Hence the orbital Gauss sum $W(\chi,a,b)$
is defined for a character $\chi$
on $\mathbb F_p^\times$ and $a,b\in V_p$.
Since the bilinear form \eqref{eq:bilin_form}
is alternating,
we have $W(\chi,b,a)=\chi(-1)W(\chi,a,b)$.
As mentioned above, explicit formulas for these Gauss sums
were proved by S. Mori \cite{mori}.
Here we quote some of his results with his permission.
\begin{prop}[Mori]\label{prop:ogs_p}
Let $p\neq 3$.
Assume $\chi={\bf 1}$ and $P(a)=0$, $a\neq0$. Then 
$W({\bf 1},a,b)$ is given by the following table.
\[
\begin{array}{c||c|c}
\hline
\text{type of $b$} 
& \text{$a$ : of type $(1^3)$}
& \text{$a$ : of type $(1^21)$}
\\
\hline
(0)	&(p^2-p)(p^2-1)	&(p^2-p)(p^2-1)\\
(1^3)	&-p(p-1)	&p(p-1)^2\\
(1^21)	&p(p-1)^2	&p(p-1)(p-2)\\
(111)	&p(p-1)(2p-1)	&-3p(p-1)\\
(21)	&-p(p-1)	&-p(p-1)\\
(3)	&-p(p-1)(p+1)	&0\\
\hline
\end{array}
\]
\end{prop}
Since Mori's preprint is not yet available, we will give a brief
outline of his proof in Remark \ref{rem:mori} in the next subsection.
\begin{proof}[Proof of Proposition \ref{prop:hat(fp)}]
The case $p=3$ can be checked numerically, 
so we assume $p\neq3$. As above, we identify $V^\ast_p$ with $V_p$.
When $b$ is of type $(111)$, by Lemma \ref{lem:n_p(sigma)}
and Propositions \ref{prop:f_inv_L_ogs} and \ref{prop:ogs_p},
\begin{align*}
\widehat{f_p}(b)
&=p^{-4}|G_p|^{-1}\sum_{a\in  V_p(0)\sqcup V_p(1^3)\sqcup V_p(1^21)}W({\bf 1},a,b)\\
&=\frac{(p^2-1)(p^2-p)\cdot 1+p(p-1)(2p-1)\cdot (p^2-1)-3p(p-1)\cdot p(p^2-1)}
{p^4(p^2-1)(p^2-p)}=-p^{-3}.
\end{align*}
All the other cases are computed similarly.
\end{proof}

\subsection{The case $N=p^2$}

In this subsection we assume $p\neq2,3$.
We identify $V^\ast_{p^2}$ with $V_{p^2}$ via $\iota$.
Recall that we defined $\Phi_{p},\Phi_{p}'\in C(V_{p^2},{\bf 1})$
in Definition \ref{defn:f_p^2}.
In this subsection we will prove:
\begin{thm}\label{thm:hat(fp^2)}
The Fourier transform of $\Phi_{p}$ is given as follows:
\begin{enumerate}[{\rm (1)}]
\item
Let $b\in pV_{p^2}$.
We write $b=pb'$ and regard $b'$ as an element of $V_p$.
Then
\[
\widehat{\Phi_{p}}(pb')=
\begin{cases}
p^{-2}+p^{-3}-p^{-5}	& \text{$b'$ : of type $(0)$},\\
p^{-3}-p^{-5}		& \text{$b'$ : of type $(1^3),(1^21)$},\\
-p^{-5}			& \text{$b'$ : of type $(111),(21),(3)$}.\\
\end{cases}
\]
\item
For $b\in V_{p^2}\setminus pV_{p^2}$,
\[
\widehat{\Phi_{p}}(b)=
\begin{cases}
p^{-3}-p^{-5}	& \text{$b$ : of type $(1^3_{\ast\ast})$},\\
-p^{-5}		& \text{$b$ : of type $(1^3_{\ast})$, $(1^3_{\rm max})$},\\
0		& \text{otherwise}.
\end{cases}
\]
\end{enumerate}
\end{thm}
\begin{thm}\label{thm:hat(fp^2')}
The Fourier transform of $\Phi_{p}$ is given as follows:
\begin{enumerate}[{\rm (1)}]
\item
Let $b\in pV_{p^2}$ and
write $b=pb'$ as above. Then
\[
\widehat{\Phi_{p}'}(pb')=
\begin{cases}
2p^{-2}-p^{-4}	& \text{$b'$ : of type $(0)$},\\
p^{-3}-p^{-4}	& \text{$b'$ : of type $(1^3)$},\\
2p^{-3}-2p^{-4}	& \text{$b'$ : of type $(1^21)$},\\
2p^{-3}-3p^{-4}	& \text{$b'$ : of type $(111)$},\\
-p^{-4}		& \text{$b'$ : of type $(21)$},\\
-p^{-3}		& \text{$b'$ : of type $(3)$}.
\end{cases}
\]
\item
For $b\in V_{p^2}\setminus pV_{p^2}$,
\[
\widehat{\Phi_{p}'}(b)=
\begin{cases}
p^{-3}-p^{-4}	& \text{$b$ : of type $(1^3_{\ast\ast})$},\\
-p^{-4}		& \text{$b$ : of type $(1^3_{\ast})$},\\
0		& \text{otherwise}.\\
\end{cases}
\]
\end{enumerate}
\end{thm}
\begin{rem}
We may check that our results are 
consistent with the
Parseval formula
\[
\textstyle
\sum_{b\in V_{p^2}}|\widehat{\Phi_{p}}(b)|^2
=p^{-8}\sum_{a\in V_{p^2}}|\Phi_{p}(a)|^2,
\]
and similarly for $\Phi_{p}'$.
\end{rem}

To prove these theorems, we evaluate
$W({\bf 1},a,b)$ for $a\in \tilde V_{p^2}^{\rm nm}$:

\begin{prop}\label{prop:singulargausssum}
Assume $b\in V_{p^2}\setminus pV_{p^2}$.
Then for $a\in \tilde V_{p^2}^{\rm nm}$, the orbital Gauss sum
$W({\bf 1},a,b)$ is given by the following table:
\[
\begin{array}{c||c|c|c|c}
\hline
\text{type of $b$}
& \text{$a$ : of type $(1^3_{\ast\ast})$}
& \text{$a$ : of type $(1^3_\ast)$}
& \text{$a$ : of type $(1^3_{\rm max})$}
& \text{$a$ : of type $(1^21_\ast)$}\\
\hline
\hline
(1^3_{\ast\ast})& p^5(p-1)^2 & p^5(p-1)^2 & -p^5(p-1) & p^5(p-1)^2\\
(1^3_{\ast})	& p^5(p-1)^2 & p^5(p-1)^2 & -p^5(p-1) & -p^5(p-1)\\
(1^3_{\rm max})	& -p^5(p-1)  & -p^5(p-1)  & \text{$p^5$ in average}&0\\
(1^21_\ast)	& p^5(p-1)^2 & -p^5(p-1)  &0		&0\\
(1^21_{\rm max})& -p^5(p-1)  &\text{$p^5$ in average} &0&0\\
(111)		&0&0&0&0\\
(21)		&0&0&0&0\\
(3)		&0&0&0&0\\
\hline
\end{array}
\]
Here, when we say ``in average'', we fix $b$ and take the average
of $W({\bf 1},a,b)$ over all $a$ of the given type.
For example, if $b$ is of type $(1^21_{\rm max})$,
``$p^5$ in average'' in the second entry means
$|V_{p^2}(1^3_\ast)|^{-1}\sum_{a\in V_{p^2}(1^3_\ast)}W({\bf 1},a,b)=p^5$.
(The individual values are described in the proof also.)
\end{prop}
The theorems follow from Propositions \ref{prop:ogs_p}
and \ref{prop:singulargausssum}:
\begin{proof}[Proof of Theorems \ref{thm:hat(fp^2)} and \ref{thm:hat(fp^2')}]
Assume $b=pb'$, $b'\in V_p$.
Then we have
$W_{p^2}(1,a,pb')=p^4W_p(1,a,b')$
by Lemma \ref{lem:ogs_reduction}.
There are respectively $p^4$, $p^3(p^2-1)$ and $p^4(p^2-1)$
elements in $V_{p^2}^{\rm nm}$
whose reductions modulo $p$ are of type $(0)$, $(1^3)$ and $(1^21)$,
respectively.
Similarly there are respectively $p^4$, $p^4(p^2-1)$ and $p^4(p^2-1)$
in $\tilde V_{p^2}^{\rm nm}$.
Hence by Propositions \ref{prop:f_inv_L_ogs} and
\ref{prop:ogs_p}, if $b'$ is of type $(111)$, we have
\begin{align*}
\widehat{\Phi_{p}}(pb')=
\frac{1}{p^{4}|G_{p^2}|}\sum_{a\in V_{p^2}^{\rm nm}}W_p(1,a,b')
&=\frac{p^3(p^2-p)(p^2-1)\{p+(2p-1)-3p\}}
{p^4|G_{p^2}|}=-p^{-5},\\
\widehat{\Phi_{p}'}(pb')=
\frac{1}{p^{4}|G_{p^2}|}\sum_{a\in\tilde V_{p^2}^{\rm nm}}W_p(1,a,b')
&=\frac{p^4(p^2-p)(p^2-1)\{1+(2p-1)-3\}}
{p^4|G_{p^2}|}=2p^{-3}-3p^{-4}.
\end{align*}
Other cases of $b\in pV_{p^2}$ are obtained similarly.
Now let $b\in V_{p^2}\setminus pV_{p^2}$.
Then since $(b\mod p)\neq0$, we have
\[
\sum_{a\in pV_{p^2}}W_{p^2}({\bf 1},a,b)
=p^4\sum_{a'\in V_{p}}W_{p}({\bf 1},a',b)=0.
\]
Hence by Proposition \ref{prop:f_inv_L_ogs},
\[
\widehat{\Phi_{p}}(b)=\frac{1}{p^{8}|G_{p^2}|}
\sum_{a\in V_{p^2}^{\rm nm}\setminus pV_{p^2}}
W_{p^2}({\bf 1},a,b)
\quad{\text and}\quad
\widehat{\Phi_{p}'}(b)
=\frac{1}{p^{8}|G_{p^2}|}
\sum_{a\in \tilde V_{p^2}^{\rm nm}\setminus pV_{p^2}}
W_{p^2}({\bf 1},a,b).
\]
Our formulas now follow from Proposition \ref{prop:singulargausssum}
and Lemma \ref{lem:n_p^2(sigma)}.
\end{proof}

We now come to the proof of Proposition \ref{prop:singulargausssum}.
\begin{proof}[Proof of Proposition \ref{prop:singulargausssum}]
We prove this by case by case direct computations.
We begin by introducing some notation
and explaining our approach.
Since $W({\bf 1},a,b)$ depends only on the $G_{p^2}$-orbits
of $a,b$, we will take a specific representative from each orbit
and compute for those $a,b$. We will choose $b$ as follows
according as its type:
\[
\begin{array}{c|c|c}
\hline
\text{type of $b$} & b & \text{condition on the coefficients}\\
\hline
\hline
(1^3_{**})	&(1,0,0,0)	&\text{--}	\\
\hline
(1^3_\ast)	&(1,0,l,0)	&l\in pR^\times	\\
\hline
(1^3_{\rm max})	& (1,0,k,-l)	&k\in pR, l\in pR^\times\\
\hline
(1^21_*)	& (0,-1,0,0)	&\text{--}	\\
\hline
(1^21_{\rm max})& (0,-1,0,-l)	&l\in pR^\times	\\
\hline
(111)		&(0,-1,1,0)	&\text{--}	\\
\hline
(21)		& (0,-1,0,-l)
	&\text{$u^2+l\in R[u]$ is irreducible}\\
\hline
(3)		& (1,0,k,-l)
	&\text{$u^3+ku-l\in R[u]$ is irreducible}\\
\hline
\end{array}
\]
For $b$ of type $(1^3_{\rm max})$, we may let $k=0$
but we sometimes leave it as is
to treat types $(1^3_{\rm max})$ and (3) simultaneously.
We will choose $a\in V_{p^2}$ later.

We put
\begin{align*}
G_{p^2,1}&:=\left\{g_1:=t(1-mn)\twtw {s}001\twtw 1nm1\ \vrule\ s,t\in R^\times, n\in R,m\in pR\right\}\subset G_{p^2},\\
G_{p^2,2}&:=\left\{g_2:=-t\twtw s001\twtw m{1+mn}1n\ \vrule\ s,t\in R^\times, n,m\in R\right\}\subset G_{p^2}.
\end{align*}
Then $G_{p^2}=G_{p^2,1}\sqcup G_{p^2,2}$.
We also drop $p^2$ and write $G_1=G_{p^2,1}, G_2=G_{p^2,2}$.
We write $\langle t\rangle :=\exp(2\pi i t/p^2)$,
hence $\langle a,b\rangle=\langle[a,b]\rangle$.
We put
\[
W_i(a,b):=\sum_{g_i\in G_{i}}\langle [g_ia,b]\rangle
\qquad i=1,2,
\]
and compute this value for each 
choice of $a,b$. We recall the formula \eqref{eq:bilin_form},
\[
[a,b]=a_4b_1-\frac13a_3b_2+\frac13a_2b_3-a_1b_4,
\]
for $a=(a_1,a_2,a_3,a_4), b=(b_1,b_2,b_3,b_4)\in V_{p^2}$,
which will be used throughout.
When convenient, we write an element of $V_{p^2}$
as a row vector via its transpose.

We immediately see that
\[
\sum_{t\in R^\times}\langle t\rangle=0,
\quad
\sum_{t\in R^\times}\langle pt\rangle=-p,
\quad
\sum_{n\in R}\langle n\rangle=
\sum_{n\in R}\langle pn\rangle=
\sum_{n\in pR}\langle n\rangle=0.
\]
We use these formulas and their variations very often.
For example, if $f(s,m,n)\in R$ is a function
independent of $t$ and $f(s,m,n)\in R^\times$ for all $s,m,n$, then
\[
\sum_{t\in R^\times}\sum_{s,m,n}\langle t\cdot f(s,m,n)\rangle
=\sum_{s,m,n}\sum_{t\in R^\times}\langle t\rangle=0.
\]
Also, if $g(s,m)\in pR$ is independent of $n$
and $\alpha\in pR^\times$ is a constant, then
\[
\sum_{n\in R}\sum_{t\in R^\times}\sum_{s,m}\langle t(g(s,m)+\alpha n)\rangle=
\sum_{t\in R^\times}\sum_{s,m}\sum_{n\in R}\langle \alpha n\rangle=0.
\]
These are typical examples of change of variables,
and we omit the explanation of such modifications
when they are easy and natural.
Finally, note that $W({\bf 1},a,b)=W({\bf 1},b,a)$.

We now carry out the computation.
In what follows,
we see that $W_2(a,b)=0$ and $W_1(a,b)$ is given
by the table of the proposition for all of our chosen
representatives.\\

\noindent
\underline{{\bf (I)} $a$ is of type $(1^3_{\ast\ast})$}.
We choose $a=(1,0,0,0)$.
To make the computation easier we replace the variables $t,m$ of $g_1$
by $s^{-2}t,st^{-1}m$ and
$t$, $m$ of $g_2$ by $st$ and $s^{-1}m$.
Then since the variable $m$ of $g_1$ is in $pR$ and hence $m^2=0$, we have
\[
g_1a=\begin{pmatrix}t\\3m\\0\\0\end{pmatrix}
,\qquad
g_2a=t\begin{pmatrix}m^3\\3m^2\\3m\\1\end{pmatrix}.
\]

\noindent
If $b=(1,0,0,0)$, then
since $|G_{1}|=p^5(p-1)^2$,
\[
W_1(a,b)=\sum_{g_1\in G_1}1=p^5(p-1)^2,
\quad
W_2(a,b)
	=\sum_{g_2\in G_2}\langle t\rangle
	=\sum_{s,m,n}\sum_{t\in R^\times}\langle t \rangle=0.
\]
 If $b=(0,-1,0,0)$, then
\[
W_1(a,b)=\sum_{g_1\in G_1}1=p^5(p-1)^2,
\quad
W_2(a,b)
	=\sum_{g_2\in G_2}\langle tm \rangle
	=\sum_{s,t,n}\sum_{m\in R}\langle tm \rangle=0.
\]
If $b=(1,0,l,0)$ is of type $(1^3_\ast)$,
then since $1+lm^2\in R^\times$ for any $m\in R$, we have
\[
W_1(a,b)=\sum_{g_1\in G_1}1=p^5(p-1)^2,
\quad
W_2(a,b)=\sum_{g_2\in G_2}\langle t(1+lm^2) \rangle
	=\sum_{g_2\in G_2}\langle t\rangle=0.
\]
If $b=(0,-1,1,0)$, then $W_1(a,b)=\sum_{g_1\in G_1}\langle m\rangle=0$ and
\begin{align*}
W_2(a,b)
&=\sum_{g_2\in G_2}\langle t(m+m^2)\rangle
=\left(\sum_{m+m^2\in R^\times}
	+\sum_{m+m^2\in pR^\times}
	+\sum_{m+m^2=0}\right)
\langle t(m+m^2)\rangle\\
&=0+2(p-1)\sum_{n\in R,s\in R^\times}(-p)+2\sum_{n\in R,s,t\in R^\times}1
=-2p^4(p-1)^2+2p^4(p-1)^2=0.
\end{align*}
Let $b=(0,-1,0,-l)$ be either of type $(1^21_{\rm max})$ or $(21)$.
Then $l\in pR^\times$ if $b$ is of type $(1^21_{\rm max})$ and
$l\in R^\times$ if $b$ is of type $(21)$.
Also $1+lm^2\in R^\times$ for any $m$, since $1+lX^2\in R[X]$ is an
irreducible polynomial. Hence we have
\begin{align*}
W_1(a,b)
&=\sum_{g_1\in G_1}\langle tl \rangle
=\begin{cases}\sum_{s,m,n}(-p)=-p^5(p-1) & (1^21_{\rm max}),\\
\sum_{s,m,n}0=0 & (21),
\end{cases}\\
W_2(a,b)&=\sum_{g_2\in G_2}\langle tm(1+lm^2)\rangle
=\sum_{g_2\in G_2}\langle tm \rangle
=\sum_{s,t,n}\sum_{m\in R}\langle tm \rangle
=0.
\end{align*}
Finally, if $b=(1,0,k,-l)$ is of type $(1^3_{\rm max})$ or $(3)$, then
by a similar consideration as above,
\begin{align*}
W_1(a,b)&=\sum_{s,m,n}\langle km \rangle \sum_t\langle tl \rangle
=\begin{cases}\sum_{s,m,n}(-p)=-p^5(p-1) & (1^3_{\rm max}),\\
\sum_{s,m,n}0=0 & (3),
\end{cases}\\
W_2(a,b)&=\sum_{g_2\in G_2}\langle t(1+km^2+lm^3)\rangle=0.
\end{align*}

\noindent
\underline{{\bf (II)} $a$ is of type $(1^21_{\ast})$}.
We consider this case next.
We choose $a=(0,1,0,0)$.
Then
\[
g_1a=t\begin{pmatrix}s^2n\\s(1+2nm)\\2m\\0\\\end{pmatrix},
\quad
g_2a=t\begin{pmatrix}s^2(m^2+nm^3)\\s(2m+3nm^2)\\1+3nm\\s^{-1}n\end{pmatrix}.
\]
If $b=(0,-1,0,0)$, then
\begin{align*}
W_1(a,b)&=\sum_{g_1\in G_1}\langle\frac{2tm}{3}\rangle
=\sum_{s,t,n}\sum_{m\in pR}\langle m\rangle=0,\\
W_2(a,b)&=\sum_{g_2\in G_2}\langle t(\frac13+nm)\rangle
=\sum_{t,s}\langle\frac t3\rangle\sum_{m,n}\langle tnm\rangle
=\sum_{t,s}\langle t\rangle\sum_{m,n}\langle nm\rangle
=\sum_{m,n,s}\langle nm\rangle\sum_{t\in R^\times}\langle t\rangle=0.
\end{align*}
Let $b=(1,0,l,0)$ be of type $(1^3_*)$.
Since $l\in pR^\times$ and $1+ls^2m^2$ is always a unit, we have
\begin{align*}
W_1(a,b)&
	=\sum_{g_1\in G_1}\langle\frac{tsl}{3}\rangle
	=\sum_{s,m,n}\sum_{t\in R^\times}\langle tl\rangle
	=\sum_{s,m,n}(-p)=-p^5(p-1),\\
W_2(a,b)&
=\sum_{g_2\in G_2}\langle t\{\frac{2mls}3+n(\frac 1s+lsm^2)\}\rangle
=\sum_{s,t,r}\langle\frac{2mlts}3\rangle\sum_n\langle\frac ts n(1+ls^2m^2)\rangle=0.
\end{align*}
Let $b=(0,-1,1,0)$ be of type $(111)$.
Then
\[
W_1(a,b)=\sum_{g_1\in G_1}\langle \frac t3(2m+s(1+2nm))\rangle
=\sum_{g_1\in G_1}\langle t\rangle=0.
\]
For $W_2(a,b)$, we have
\[
W_2(a,b)
=\sum_{s,t,m}\sum_{n}\langle\frac t3(1+2ms+3n(m+sm^2))\rangle.
\]
We divide the sum
according as $m+sm^2=m(1+sm)\in R^\times$ or not.
The former is
\[
\sum_{m+sm^2\in R^\times}\langle
\frac t3(1+2ms)\rangle\langle tn(m+m^2s)\rangle
= \sum_{m+sm^2\in R^\times}
	\langle\frac t3(1+2ms)\rangle\langle n \rangle=0.
\]
Let $m+sm^2\in pR$. Then either $m\in pR$ or $ms\in -1+pR$.
Hence $(1+2ms+3n(m+sm^2))\in R^\times$ and so the latter sum
is $0$ as well. Hence $W_2(a,b)=0$.

Let $b=(0,-1,0,-l)$ be of type $(1^21_{\rm max})$ or $(21)$.
If $b$ is of type $(1^21_{\rm max})$ then $s^2nl\in pR$ and
if $b$ is of type $(21)$ then $s^2l\in R^\times$ for any $s,n$. Hence
\[
W_1(a,b)=\sum_{g_1\in G_1}\langle t(\frac{2m}3+s^2nl)\rangle
=\begin{cases}
\sum_{t,s,n}\sum_{m\in pR}\langle tm\rangle=0 & (1^21_{\rm max}),\\
\sum_{t,s,m}\sum_{n\in R}\langle tn\rangle=0 & (21).\\
\end{cases}
\]
For $W_2(a,b)$, since $1+s^2m^2l\in R^\times$, we have
\[
W_2(a,b)
=\sum_{g_2\in G_2}
\langle t(\frac13+s^2m^2l)\rangle
\langle tnm(1+s^2m^2l)\rangle
=\sum_{t,s,m}
\langle t(\frac13+s^2m^2l)\rangle
\sum_{n}\langle nm\rangle.
\]
Since $\sum_{n\in R}\langle nm\rangle=0$ unless $m=0$, we have
$W_2(a,b)=\sum_{t,s,n}\langle t/3\rangle=0$.

Let $b=(1,0,k,-l)$ be of type $(1^3_{\rm max})$ or $(3)$.
If $b$ is of type $(1^3_{\rm max})$, since $l\in pR^\times$,
\[
W_1(a,b)=\sum_{g_1\in G_1}\langle t(\frac{sk}3+s^2ln) \rangle
=\sum_{t,s,m}\langle\frac{tsk}3\rangle
\sum_{n}\langle ts^2ln\rangle=0.
\]
If $b$ is of type $(3)$, then
since $s^2l+\frac{2sk}3m\in R^\times$,
\[
W_1(a,b)
=\sum_{g_1\in G_1}\langle tn(s^2l+\frac{2sk}3m)\rangle
\langle\frac{tsk}3\rangle
=\sum\langle n\rangle\langle\frac{tsk}3\rangle=0.
\]
Moreover for both types,
\[
W_2(a,b)
=\sum_{g_2\in G_2}\langle tn(s^{-1}+sm^2k+s^2m^3l)\rangle
\langle t(\frac{2msk}3+s^2m^2l)\rangle.
\]
Since $s^{-1}+sm^2k+s^2m^3l\in R^\times$,
$W_2(a,b)=\sum_{g_2} \langle n\rangle\langle t(\frac{2msk}3+s^2m^2l)\rangle=0$.\\

\noindent
\underline{{\bf (III)} $a$ is of type $(1^3_{\ast})$}.
Let $a=(1,0,\alpha,0)$, where $\alpha\in pR^\times$.
By replacing $m$ of $g_1$ with $s^{-1}m-2\alpha n/3$, we have
\[
g_1a=
t\begin{pmatrix}
	s^2(1+\alpha n^2)\\
	3m		\\
	\alpha		\\
	0		\\
\end{pmatrix},
\quad
g_2a=
t\begin{pmatrix}
	s^2	(m^3+\alpha(n^2m^3+2nm^2+m))\\
	s	(3m^2+\alpha(3n^2m^2+4nm+1))\\
		3m+\alpha(3n^2m+2n)\\
	s^{-1}	(1+\alpha n^2)
\end{pmatrix}.
\]
Let $b=(1,0,l,0)$ be of type $(1^3_\ast)$.
Then
\[
W_1(a,b)=\sum_{g_1\in G_1}1=p^5(p-1)^2,
\quad
W_2(a,b)=\sum_{g_2\in G_2}\langle t(\frac1s(1+\alpha n^2)+slm^2)\rangle
=\sum_{g_2\in G_2}\langle t \rangle=0.
\]
Let $b=(0,-1,1,0)$ be of type $(111)$.
Then $W_1(a,b)=\sum_{t,s,n}\langle t\alpha/3 \rangle
\sum_{m}\langle tm \rangle=0$, and
\begin{align*}
W_2(a,b)
&=\sum_{g_2\in G_2}\langle \frac t3(3m+\alpha(3n^2m+2n)) \rangle
\langle \frac {ts}3(3m^2+\alpha(3n^2m^2+4nm+1)) \rangle\\
&=\sum_{t,s,n}
\left(
	\sum_{m\in R^\times}\langle t \rangle\langle ts \rangle
+	\sum_{m\in pR}\langle t(m+\frac{2\alpha n}3) \rangle\langle\frac{ts\alpha}3 \rangle
\right)\\
&=\sum_{m\in R^\times}\sum_{s,n}\sum_{t\in R^\times}\langle s \rangle\langle t \rangle+\sum_{m \in pR}\sum_{s,t}\sum_{n\in R}\langle m \rangle\langle t \rangle
=0+0=0.
\end{align*}
Let $b=(1,0,k,-l)$ be either of type $(1^3_{\rm max})$ or $(3)$.
Then
\begin{align*}
W_1(a,b)&=
\begin{cases}
\sum_{g_1\in G_1}\langle ts^2l \rangle=-p^5(p-1) & (1^3_{\rm max}),\\
\sum_{g_1\in G_1}\langle t(ls^2+(km+\alpha s^2n^2l))\rangle
=\sum_{g_1\in G_1} \langle t \rangle=0 & (3),
\end{cases}\\
W_2(a,b)&=
\sum_{g_2\in G_2}
\langle t\{s^{-1}(1+ks^2m^2+ls^3m^3)+\alpha f(s,n,m)\} \rangle
=\sum_{g_2\in G_2}\langle t \rangle=0.
\end{align*}
Note that $1+ks^2m^2+ls^3m^3=b(1,sm)\in R^\times$ for all $s,m$.

Let $b=(0,-1,0,-l)$ be of type $(1^21_{\rm max})$ or $(21)$.
Then
\[
W_2(a,b)
=\sum_{G_2}\langle t\{m(1+lm^2s^2)+\frac\alpha3 f(s,m,n)\} \rangle,
\]
where $f(s,m,n)=3n^2m+2n+s^2(n^2m^3+2nm^2+m)l$.
Note that $1+lm^2s^2\in R^\times$. If $m\in pR$, then
$\alpha f(s,m,n)=2n\alpha$. Hence
\[
W_2(a,b)
=\sum_{m\in R^\times}\langle t \rangle
	+\sum_{m\in pR}\langle t(m+\frac{2n\alpha}3)\rangle
=0+0=0.
\]
If $b$ is of type $(21)$,
\[
W_1(a,b)=\sum_{G_1}\langle t(ls^2+\frac\alpha3+ls^2n^2\alpha) \rangle=0.
\]
If $b$ is of type $(1^21_{\rm max})$,
\[
W_1(a,b)
=\sum_{g_1\in G_1}\langle t(\frac\alpha3+ls^2) \rangle
=\sum_{s,m,n}\sum_{t\in R^\times}\langle t(\alpha+3ls^2) \rangle.
\]
We consider whether $\alpha+3ls^2\in pR$
is $0$ or not for $s\in R^\times$.
Note that since $\alpha,3l\in pR^\times$,
we can regard $\alpha/3l\in\mathbb F_p^\times$.
If $-\alpha/3l$ is not a square in $\mathbb F_p^\times$,
then $\alpha+3ls^2\in pR^\times$ for any $s\in R^\times$.
If $-\alpha/3l$ is a square in $\mathbb F_p^\times$,
then $\alpha+3ls^2=0$ for $2p$ choices of $s\in R^\times$,
and $\alpha+3ls^2\in pR^\times$ otherwise.
Hence
\[
W_1(a,b)=
\begin{cases}
\sum_{s,m,n}(-p)=p^5-p^6
	& -\alpha/3l\notin (\mathbb F_p^\times)^2,\\
2p\sum_{m,n}(p^2-p)+(p^2-3p)\sum_{m,n}(-p)=p^5+p^6
	& -\alpha/3l\in (\mathbb F_p^\times)^2.
\end{cases}
\]
By Proposition \ref{prop:singular_orbits} (2),
for $a\in V_{p^2}(1^3_\ast)$,
$-\alpha/3l\in (\mathbb F_p^\times)^2$
or $-\alpha/3l\notin (\mathbb F_p^\times)^2$
occurs with equal probability.
Hence the average value of $W_1(a,b)$ with respect to
$a\in V_{p^2}(1^3_\ast)$ is $p^5$.\\

\noindent
\underline{{\bf (IV)} $a$ is of type $(1^3_{\max})$}.
By Proposition \ref{prop:singular_orbits},
we can take $a=(1,0,0,\alpha)$ for
some $\alpha\in pR^\times$.
Then
\[
g_1a=
t\begin{pmatrix}
	s^2(1+\alpha n^3)\\
	3s(m+\alpha n^2)\\
	3\alpha n\\
	s^{-1}\alpha
\end{pmatrix},
\quad
g_2a=
t\begin{pmatrix}
	s^2(m^3+\alpha(mn+1)^3)\\
	3s(m^2+\alpha n(mn+1)^2)\\
	3(m+\alpha n^2(mn+1))\\
	s^{-1}(1+\alpha n^3).
\end{pmatrix}.
\]
Let $b=(0,-1,1,0)$. Then
\begin{align*}
W_1(a,b)
&=\sum_{s,t,n}\sum_{m\in pR}
\langle t(sm+\alpha n+\alpha sn^2) \rangle
=\sum_{s,t,n}\sum_{m\in pR}
\langle m \rangle=0,\\
W_2(a,b)
&=\sum_{g_2\in G_2}
	\langle t(m+\alpha n^2(mn+1)) \rangle
	\langle st(m^2+\alpha n(mn+1)^2) \rangle\\
&=\sum_{s,t,n}\sum_{m\in R^\times}\langle t \rangle\langle ts \rangle
+\sum_{s,t,n}\sum_{m\in pR}\langle t(m+\alpha n^2) \rangle
\langle \alpha tsn \rangle
=0+0=0.
\end{align*}
Let $b=(0,-1,0,-l)$ be of type $(1^21_{\rm max})$ or $(21)$. Then
\begin{align*}
W_1(a,b)&=\sum_{g_1\in G_1}\langle t(\alpha n+ls^2(1+\alpha n^3)\rangle
=\begin{cases}
\sum_{s,t,m}\sum_{n}\langle t(\alpha n+ls^2) \rangle=0	& (1^21_{\rm max}),\\
\sum_{s,m,n}\sum_{t}\langle t \rangle=0			& (21),
\end{cases}\\
W_2(a,b)
&=\sum_{g_2\in G_2}
\langle t\{m(1+ls^2m^2)+\alpha(n^2(mn+1)+ls^2(mn+1)^3)\} \rangle\\
&=\sum_{s,t,n}\sum_{m\in R^\times}\langle t \rangle
	+\sum_{s,t,n}\sum_{m\in pR}\langle t(m+\alpha(n^2+s^2l)) \rangle
=0+0=0.
\end{align*}
Let $b=(1,0,k,-l)$ be of type $(3)$.
Then
\begin{align*}
W_1(a,b)&=\sum_{g_1\in G_1}
\langle t(ls^2+ksm+\alpha(s^{-1}+kn^2+ls^2n^3)) \rangle
=\sum_{s,m,n}\sum_{t\in R^\times}\langle t \rangle=0,\\
W_2(a,b)&=\sum_{g_2\in G_2}
\langle t\{s^{-1}(1+ks^2m^2+ls^3m^3)+\alpha f(s,m,n)\} \rangle
=\sum_{s,m,n}\sum_{t\in R^\times}\langle t \rangle=0.
\end{align*}
Note that $1+ks^2m^2+ls^3m^3\in R^\times$ for any $s,m$.

Let $b=(1,0,0,-l)$ be of type $(1^3_{\rm max})$.
(By Proposition \ref{prop:singular_orbits}, we may assume $k=0$.)
Similarly to as above, we have $W_2(a,b)=0$. For
$W_1(a,b)$, we have
\[
W_1(a,b)
=\sum_{g_1\in G_1}\langle \frac ts(\alpha+s^3l) \rangle
=\sum_{t,m,n}\sum_s\langle t(\alpha+s^3l) \rangle.
\]
We consider whether $\alpha+s^3l\in pR$ is $0$ or not for $s\in R^\times$.

First, let $p\equiv 2\mod 3$. Then
$\alpha+s^3l=0$ for $p$ choices of $s$ and 
$\alpha+s^3l\in pR^\times$ otherwise.
Hence
\[
W_1(a,b)
=p\sum_{m,n,t}1+(p^2-2p)\sum_{m,n}(-p)
=p(p^2-p)p^3-(p^2-2p)p^4=p^5.
\]
Next let $p\equiv 1\mod 3$.
If $-\alpha/l$ is not a cube in $\mathbb F_p^\times$,
then $\alpha+s^3l\in pR^\times$ for all $s$.
If $-\alpha/l$ is a cube in $\mathbb F_p^\times$,
$\alpha+s^3l=0$ for $3p$ choices of $s$ and 
$\alpha+s^3l\in pR^\times$ otherwise.
Hence
\[
W_1(a,b)
=\begin{cases}
(p^2-p)\sum_{m,n}(-p)=p^5-p^6
	& -\alpha/l\notin (\mathbb F_p^\times)^3,\\
3p\sum_{m,n,t}1+(p^2-4p)\sum_{m,n}(-p)=p^5+2p^6
	& -\alpha/l\in (\mathbb F_p^\times)^3.\\
\end{cases}
\]
By Proposition \ref{prop:singular_orbits} (4), 
the average value of $W_1(a,b)$ with respect to
$a\in V_{p^2}(1^3_{\rm max})$ is
$\frac23 (p^5-p^6)+\frac 13(p^5+2p^6)=p^5$.
\end{proof}
\begin{rem}\label{rem:mori}
We briefly review Mori's proof \cite{mori} of Proposition \ref{prop:ogs_p}.
The outline is quite similar to the proof above.
We fix a prime $p\neq3$.
Let
\begin{align*}
G_{p,1}&:=\left\{g_1:=t\twtw {s}001\twtw 1n01\ \vrule\ s,t\in \F_p^\times, n\in \F_p\right\}\subset G_{p},\\
G_{p,2}&:=\left\{g_2:=-t\twtw s001\twtw m{1+mn}1n\ \vrule\ s,t\in \F_p^\times, n,m\in \F_p\right\}\subset G_{p}.
\end{align*}
Then $G_{p}=G_{p,1}\sqcup G_{p,2}$.
We drop $p$ and write $G_1=G_{p,1}, G_2=G_{p,2}$.
We write $\langle t\rangle :=\exp(2\pi i t/p)$,
hence $\langle a,b\rangle=\langle[a,b]\rangle$.
We put
$W_i(a,b):=\sum_{g_i\in G_{i}}\langle [g_ia,b]\rangle$
and compute this value.
Note that $W({\bf 1},a,b)=W_1(a,b)+W_2(a,b)$.
In this case of $\F_p=\Z/p\Z$, we have
\[
\sum_{t\in \F_p^\times}\langle t\rangle=-1,
\quad
\sum_{n\in\F_p}\langle n\rangle=0.
\]
Let $a=(1,0,0,0)$ be of type $(1^3)$.
Then a change of
variables similar to that in (I) in the previous proof, we have
$
g_1a=(t,0,0,0)$,
$g_2a=t(m^3,3m^2,3m,1)$.
If $b=(1,0,0,0)$, then
\[
W_1(a,b)=\sum_{g\in G_1}1=|G_1|=p(p-1)^2,
\quad
W_2(a,b)=\sum_{g\in G_2}\langle t\rangle=\sum_{s,m,n}(-1)=-p^2(p-1),
\]
and hence $W({\bf 1},a,b)=-p(p-1)$. If $b=(0,1,0,0)$, then
\[
W_1(a,b)=\sum_{g\in G_1}1=p(p-1)^2,
\quad
W_2=\sum_{g\in G_2}\langle tm\rangle=
\sum_{s,t,n}\sum_{m}\langle m\rangle=0
\]
and $W({\bf 1},a,b)=p(p-1)^2$ follows. 
If $b=(1,0,k,-l)$ is of type $(3)$, then $l\in\F_p^\times$
and also $1+m^2k+m^3l\in\F_p^\times$ for all $m\in\F_p$. Hence
\[
W_1(a,b)=\sum_{g\in G_1}\langle tl\rangle=-p(p-1),
\quad
W_2(a,b)=\sum_{g\in G_2}\langle t(1+km^2+lm^3)\rangle=-p^2(p-1)
\]
and so $W({\bf 1},a,b)=-p(p^2-1)$.
Let $a=(0,1,0,0)$. Then as in (II) in the previous proof,
\[
g_1a=t\begin{pmatrix}s^2n\\s\\0\\0\\\end{pmatrix},
\quad
g_2a=t\begin{pmatrix}s^2(m^2+nm^3)\\s(2m+3nm^2)\\1+3nm\\s^{-1}n\end{pmatrix}.
\]
Let $b=(0,-1,1,0)$ be of type $(111)$. Then
$W_1(a,b)=\sum_{g\in G_1}\langle ts/3\rangle=-p(p-1)$.
For $W_2(a,b)$, by exactly the same consideration as in (II)
in the previous proof,
\[
W_2(a,b)
=
\!\!\!\!\sum_{m+sm^2\in\F_p^\times}
\!\!	\langle\frac t3(1+2ms)\rangle\langle n \rangle
	+\sum_{m=0}\langle\frac t3 \rangle
	+\!\!\sum_{1+ms=0}\!\!\langle-\frac t3 \rangle
=0-p(p-1)-p(p-1)=-2p(p-1).
\]
Hence we have $W({\bf 1},a,b)=-3p(p-1)$.
Let $b=(0,-1,0,-l)$ be of type $(21)$. Then similarly,
\[
W_1(a,b)=\sum_{t,s,n,m}\langle ts^2nl\rangle =0,
\qquad
W_2(a,b)=\sum_{t,s,n,m}\langle t/3\rangle=-p(p-1)
\]
and hence $W({\bf 1},a,b)=-p(p-1)$.
The other cases are obtained similarly
and we omit the detail.
\end{rem}

\section{Ohno-Nakagawa's relation}
\label{sec:ON}

After 25 years after Shintani
introduced the zeta functions $\xi(s)$ and $\xi^\ast(s)$,
Ohno \cite{ohno} conjectured a remarkably
simple formula satisfied by $\xi(s)$ and $\xi^\ast(s)$.
This was proved by Nakagawa \cite{nakagawa}.
\begin{thm}[Ohno-Nakagawa]\label{thm:ON}
We have
\[
\xi^\ast(s)=A\cdot\xi(s),
\qquad
A=\twtw 0130.
\]
\end{thm}
By plugging this formula into Shintani's functional equation
$\xi(1-s)=M(s)\xi^\ast(s)$ in Theorem \ref{thm:FEShintani},
the functional equation was rewritten
in the following symmetric form:
\begin{thm}[Ohno-Nakagawa]\label{thm:FEON}
We have
\[
\Delta(1-s)\cdot T\cdot \xi(1-s)=\Delta(s)\cdot T\cdot \xi(s),
\]
where $\Delta(s)$ and $T$ are as in Theorem \ref{thm:intro_ON}.
\end{thm}
Theorem \ref{thm:FEON} follows from Theorem \ref{thm:ON}
simply because $\Delta(1-s)TM(s)=\Delta(s)TA^{-1}$.
We note that this ``diagonalization'' of $M(s)$ is due to
Datskovsky and Wright \cite[Proposition 4.1]{dawra}.
Recently, Ohno, Wakatsuki and the first author \cite{sty, yt}
classified all $\spl_2(\Z)$-invariant lattices $L$ in $V_\Z$,
and showed that the associated zeta function $\xi(s,L,\spl_2(\Z))$
for each $L$ satisfies a functional equation in a self dual form
as in Theorem \ref{thm:FEON}.
In this section we establish a similar formula for the
``$N$-divisible zeta function'' when $N$ is square free.
We also state residue formulas which we will prove in Section \ref{sec:residue}.

In this section we assume that $N$ is a
square free integer.
Let
$f_N\in C(V_N, {\bf 1})$ be
the characteristic function of
$\{a\in V_N\mid P(a)=0\}$.
This definition agrees with Definition \ref{defn:f_p},
and $f_N=\prod_{p\mid N}f_p$.
\begin{defn}
We define the {\em $N$-divisible zeta function}
\[
\xi_N(s):=\xi(s,f_N)
=\sum_{a\in V_N, P(a)=0}\xi(s,a).
\]
\end{defn}
This function counts the $V_\Z$-orbits
whose discriminants are multiples of $N$, 
and we study its analytic properties.

We first describe the residues; these formulas follow from
Proposition \ref{prop:residue_f} and Corollary \ref{cor:f_p}.
\begin{prop}\label{prop:divisibleRF}
We have
\[
\underset{s=1}{\res}\ \xi_N(s)
=\prod_{p\mid N}\left(\frac1p+\frac{1}{p^2}-\frac1{p^3}\right)
\cdot(\alpha+\beta),
\quad
\underset{s=5/6}{\res}\ \xi_N(s)
=\prod_{p\mid N}\left(\frac1p+\frac{1}{p^{4/3}}-\frac1{p^{7/3}}\right)
\cdot\zeta(\frac13)\gamma,
\]
where $\alpha$, $\beta$ and $\gamma$ are as in Definition \ref{defn:albtgm}.
\end{prop}
We next prove the functional equation.
For a square free integer $m$, we put
\[
\xi^\ast_m(s):=\sum_{b\in V^\ast_m, P^\ast(b)=0}\xi^\ast(s,b).
\]
\begin{prop}\label{prop:divisibleFE}
If $N$ is square free,
we have
\[
\xi_N(1-s)=N^{4s-3}M(s)
\sum_{m_1m_2m_3=N}\mu(m_1)m_2m_3^{2-4s}\xi^\ast_{m_2}(s).
\]
\end{prop}
\begin{proof}
By Theorem \ref{thm:FESato} and Proposition \ref{prop:hat(fp)}, we have
\[
\xi_N(1-s)=N^{4s}M(s)\sum_{b\in V^\ast_N}\widehat f_N(b)\xi^\ast(s,b)
\]
where $\widehat f_N(b)$ is multiplicative in $N$,
and for a prime $p$ is equal to
$p^{-1}+p^{-2}-p^{-3}$ if $b=0$,
$p^{-2}-p^{-3}$ if $b\neq0$ but $P^\ast(b)=0$,
and $-p^{-3}$ otherwise. We rewrite $\widehat f_p(b)$ as
\[
\widehat f_p(b)=-p^{-3}+g_{p,2}(b)p^{-2}+g_{p,3}(b)p^{-1},
\]
where $g_{p,3}(b)$ is the characteristic function of $b=0$,
and $g_{p,2}(b)$ is the characteristic function of those
$b$ with $P^\ast(b)=0$.
Then we have
\begin{align*}
\xi_N(1-s)
&=N^{4s}M(s)\sum_{b\in V^\ast_N}\prod_{p\mid N}(-p^{-3}+g_{p,2}(b)p^{-2}+g_{p,3}(b)p^{-1})
\xi^\ast(s,b)\\
&=N^{4s}M(s)\sum_{m_1m_2m_3=N}\mu(m_1)m_1^{-3}m_2^{-2}m_3^{-1}
\biggl(\sum_{\substack{b\in V^\ast_N\\ m_2\mid P^\ast(b),\ b\in m_3V^\ast_N}}
\xi^\ast(s,b)\biggr)\\
&=N^{4s-3}M(s)\sum_{m_1m_2m_3=N}\mu(m_1)m_2m_3^{2}	
\biggl(m_3^{-4s}\sum_{\substack{b\in V^\ast_N\\ m_2\mid P^\ast(b)}}
\xi^\ast(s,b)\biggr),
\end{align*}
as desired.
\end{proof}

Let $\varphi$ and $\mu$ be the Euler function and M\"obius function,
respectively.
Similarly to \cite{yt}, we find a functional equation
in self dual form for certain linear combinations of zeta
functions.
\begin{thm}\label{thm:FE_thetaN}
For a square free integer $N$, let
\[
\theta_N(s):=\sum_{m\mid N}\mu(m)m\xi_m(s).
\]
Then
\[
\underset{s=1}{\res}\ \theta_N(s)
=\mu(N)\frac{\varphi(N)}{N^2}
\cdot(\alpha+\beta),\quad
\underset{s=5/6}{\res}\ \theta_N(s)
=\mu(N)\frac{\varphi(N)\zeta(1/3)}{N^{4/3}}
\cdot\gamma,
\]
and
\[
N^{2(1-s)}\Delta(1-s)\cdot T\cdot \theta_N(1-s)
=N^{2s}\Delta(s)\cdot T\cdot \theta_N(s).
\]
\end{thm}
\begin{proof}
By Proposition \ref{prop:divisibleRF},
\[
\underset{s=1}{\res}\ m\xi_m(s)
=\prod_{p\mid m}\left(1+\frac{1-p^{-1}}p\right)
\cdot(\alpha+\beta),\quad
\underset{s=5/6}{\res}\ m\xi_m(s)
=\prod_{p\mid m}\left(1+\frac{1-p^{-1}}{p^{1/3}}\right)
\cdot\zeta(1/3)\gamma.
\]
Since
\[
\prod_{p\mid N}(1-a_p)=\sum_{m\mid N}\mu(m)\prod_{p\mid m}a_p
\]
for any $a_p$ and $\prod_{p\mid N}(1-p^{-1})=\varphi(N)/N$,
the residue formulas follow. 
We consider the functional equation.
By Proposition \ref{prop:divisibleFE},
\begin{align*}
M(s)^{-1}\cdot\theta_N(1-s)
&=M(s)^{-1}\sum_{m|N}\mu(m)m\xi_m(1-s)\\
&=\sum_{m|N}\mu(m)m^{4s-2}
\sum_{m_1m_2m_3=m}\mu(m_1)m_2m_3^{2-4s}\xi^\ast_{m_2}(s)\\
&=\sum_{m_1m_2m_3m_4=N}\mu(m_2m_3)m_2(m_1m_2)^{4s-2}\xi^\ast_{m_2}(s)\\
&=\sum_{m_1m_2\mid N}\mu(m_2)(m_1m_2)^{4s-2}m_2\xi^\ast_{m_2}(s)
\sum_{m_3m_4=\frac{N}{m_1m_2}}\mu(m_3).
\end{align*}
By the M\"obius inversion formula,
$\sum_{m_3m_4=\frac{N}{m_1m_2}}\mu(m_3)$ is $1$ if $N=m_1m_2$
and $0$ otherwise.
Hence
\[
M(s)^{-1}\cdot\theta_N(1-s)
=N^{4s-2}\sum_{m\mid N}\mu(m)m\xi^\ast_m(s).
\]
By the Ohno-Nakagawa formula,
we have $\xi^\ast_m(s)=A\cdot\xi_m(s)$ for any $m$.
Hence
\[
M(s)^{-1}\cdot\theta_N(1-s)=N^{4s-2}A\cdot\theta_N(s).
\]
Since $\Delta(1-s)TM(s)=\Delta(s)TA^{-1}$, we have the formula.
\end{proof}

\section{Computation of residues}\label{sec:residue}
In this section we compute the residues of our orbital $L$-functions
and related zeta functions. We start by showing that
for a suitable choice $\Phi_a$
of test function, the adelic Shintani zeta function studied
by Wright \cite{wright} gives an integral expression for $\xi(s,\chi,a)$.
Hence its residues are described in terms of certain integrals
which have Euler products. The local analysis is
carried out in later subsections.
We note that a number of results of this section are already
obtained in the extensive work of Datskovsky-Wright \cite{dawra},
and we follow their approach to give
refinements of their results.

Recall that $\xi(s,\chi,a)$ is a vector consisting of
two Dirichlet series.
When we talk about the analytic properties
of $\xi(s,\chi,a)$, we mean so entrywise.
The locations of the poles coincide for the two series,
so we hope our meaning is clear.
In particular, for $z_0\in\C$, we denote by
$\res_{z=z_0}\xi(s,\chi,a)$
the column vector of residues of these Dirichlet series at $z=z_0$.

We fix some notation for adelic analysis.
Let $\R_+^\times=\{t\in\R\mid t>0\}$ and
$\C_1^\times=\{z\in\C^\times\mid|z|=1\}$.
For $t\in\Q_p^\times$, let $\ord_p(t)$ be the unique integer $m$
satisfying $t\in p^m\Z_p^\times$.
Let $|\cdot|_p\colon\Q_p^\times\rightarrow\R^\times_+$
be the normalized absolute value, hence $|t|_p=p^{-\ord_p(t)}$.
We normalize the Haar measure $du$ (resp. $\md t$) on $\Q_p$
(resp. $\Q_p^\times$)
so that the volume of $\Z_p$ is $1$ (resp. of $\Z_p^\times$ is $1$).
We put $\widehat\Z:=\prod_{\text{$p$:finite}}\Z_p$, so that
$\widehat\Z^\times=\prod_{\text{$p$:finite}}\Z_p^\times$.
As usual, let
$\A_\rf:=\widehat\Z\otimes\Q$ and $\A=\A_\rf\times\R$.
Let $\sS(V_\R), \sS(V_{\Q_p}), \sS(V_{\A_{\rf}}), \sS(V_\A)$
be the space of Schwarz-Bruhat functions on each of the
indicated domains.
For a Dirichlet character $\chi$, let $\delta(\chi)$ be $1$ if
$\chi$ is trivial and $0$ otherwise.

For the rest of this section
we fix a (primitive) Dirichlet $\chi$ of conductor $m$.
We introduce notation related to $\chi$.
The usual $L$-function and the local $L$-factors for $\chi$
are defined by
\[
L(s,\chi):=\prod_{p}L_p(s,\chi),
\qquad
L_p(s,\chi):=
\begin{cases}
1 & p\mid m,\\
\left(1-\chi(p)/p^s\right)^{-1} & p\nmid m.
\end{cases}
\]
Recall the direct product decomposition
$\A^\times=\R_+^\times\cdot\widehat\Z^\times\cdot\Q^\times$.
We lift $\chi$ to an idele class character
$\tilde\chi\colon\A^\times/\Q^\times\rightarrow\C_1^\times$
by using the compositum
\[
\tilde\chi\colon\A^\times
=\R_+^\times\cdot\widehat\Z^\times\cdot\Q^\times
\twoheadrightarrow\widehat\Z^\times
\twoheadrightarrow(\widehat\Z/m\widehat\Z)^\times
\cong (\Z/m\Z)^\times\overset{\chi}{\rightarrow}\C_1^\times.
\]
Let $\tilde\chi_p$ be the local character on $\Q_p^\times$
induced from $\tilde\chi$:
\[
\tilde\chi_p:=
\tilde\chi|_{\Q_p^\times}\colon \Q_p^\times\rightarrow\C_1^\times.
\]

Let $m=\prod p^{c_p}$.
Then corresponding to the decomposition
$(\Z/m\Z)^\times=\prod(\Z/p^{c_p}\Z)^\times$,
$\chi$ has a unique decomposition
$\chi=\prod\chi_p$,
where $\chi_p$ is a primitive character on
$(\Z/p^{c_p}\Z)^\times$.
(If $p\nmid m$, this means $c_p=0$ and $\chi_p$ is
the trivial character.)
We put
$\chi_p':=\prod_{p'\neq p}\chi_{p'}$.
This is a primitive character of conductor $m/p^{c_p}$
which is coprime to $p$.
Hence $\chi_p'(p)$ makes sense.

Let us describe $\tilde\chi_p$ explicitly
in terms of $\chi$.
Since $\Q_p^\times=\Z_p^\times\times p^\Z$, it is enough
to describe $\tilde\chi_p|_{\Z_p^\times}$ and $\tilde\chi_p(p)$,
and by definition these are given as follows.
\begin{lem}\label{lem:character}
The character $\tilde\chi_p$ restricted to $\Z_p^\times$ agrees with
the pullback of $\chi_p$ via the canonical surjection
$\Z_p^\times\rightarrow(\Z_p/p^{c_p}\Z_p)^\times
\cong(\Z/p^{c_p}\Z)^\times$. Also we have 
$\tilde\chi_p(p)=\chi_p'(p)^{-1}$.
\end{lem}
If $p\nmid m$, then $\chi_p'=\chi_p$ and so 
$\tilde\chi_p(p)=
\chi(p)^{-1}$.

\subsection{An integral expression for orbital $L$-functions}
We now return to the analysis of zeta functions.
In this subsection we fix a positive integer $N$
which is a multiple of $m$, and we also fix $a\in V_N$.
Let $V'_\Q:=\{x\in V_\Q\mid P(x)\neq0\}$.
For $\Phi\in\sS(V_\A)$,
Wright \cite{wright} introduced the global zeta function
\[
Z(\Phi,s,\chi)
:=\int_{G_\A/G_\Q}|\det g|_\A^{2s}\tilde\chi(\det g)
\sum_{x\in V'_\Q}\Phi(gx)\, dg.
\]
Here $dg$ is a Haar measure on $G_\A$ normalized as in
\cite[p.514]{wright}.

Let $\Phi_a=\Phi_\infty\times\Phi_{\rf,a}$
where $\Phi_\infty\in\sS(V_\R)$ is arbitrary and
$\Phi_{\rf,a}\in\sS(V_{\A_\rf})$ is the characteristic
function of $\tilde a+NV_{\widehat\Z}\subset V_{\widehat\Z}$,
where $\tilde a$ is a lift of $a$ under the surjection
$V_{\widehat\Z}\twoheadrightarrow V_{\widehat\Z/N\widehat\Z}\cong V_N$.
Let $G_\R^+:=\{g\in G_\R=\gl_2(\R)\mid \det g>0\}$ with Haar measure
$dg_\infty$ normalized as in \cite{wright}, and
let $\Gamma_\infty(\Phi_\infty,s)$ be the local zeta function 
defined in \eqref{eqn:vv_zeta}.

\begin{prop}\label{prop:unfold}
We have
\[
Z(\Phi_a,s,\chi)=|G_N|^{-1}
\Gamma_\infty(\Phi_\infty,s)\xi(s,\chi^{-1},a).
\]
If $\chi={\bf 1}$, $\xi(s,\chi,a)$ is holomorphic
except for simple poles
at $s=1$ and $5/6$. If $\chi\neq{\bf 1}, \chi^3={\bf 1}$,
then $\xi(s,\chi,a)$ is holomorphic except for a simple pole at $s=5/6$.
Otherwise it is entire.
\end{prop}

\begin{proof}
Let $\mathcal K_N:=\ker(G_{\widehat\Z}\rightarrow G_{\Z/N\Z})$.
This is an open subgroup of $G_{\A_\rf}$.
Note that
$G_\R^+\mathcal K_N\cap G_\Q=\Gamma(N)$.
For $t\in T_N\subset G_N=G_{\Z/N\Z}$, we denote by
$\tilde t\in G_{\widehat\Z}$ its (arbitrary) lift.
Then $\mathcal K_N\tilde t$ does not depend on the choice of $\tilde t$.
By strong approximation for $\spl_2$, it is known
(see e.g. \cite{gelbart}) that $G_\A=G_\R^+\mathcal KG_\Q$
for any subgroup $\mathcal K\subset G_{\widehat\Z}$ such that
the determinant map is surjective onto $\widehat\Z^\times$.
One checks that the union $\cup_{t\in T_N}\mathcal K_N\tilde t$
is such a $\mathcal K$,
and it is not difficult to check that the union
$G_\A=\cup_{t\in T_N}G_\R^+\mathcal K_N\tilde t G_\Q$
is disjoint.

Corresponding to this decomposition, for each $t\in T_N$ we put
\[
Z_t(\Phi_a,s,\widetilde\chi)
:=\int_{G_\R^+\mathcal K_N\tilde t G_\Q/G_\Q}|\det g|_\A^{2s}\widetilde\chi(\det g)\sum_{x\in V'_\Q}\Phi_a(gx)\, dg.
\]
Let $dg_\rf$ be the Haar measure on $G_{\A_\rf}$ normalized so that
the volume of $G_{\widehat\Z}$ is $1$. Then $dg=dg_\infty dg_\rf$.
For this integral, we have the following process of modification:
\begin{align*}
Z_t(\Phi_a,s,\widetilde\chi)
&=\int_{G_\R^+\mathcal K_NG_\Q/G_\Q}|\det \tilde tg|_\A^{2s}\widetilde\chi(\det \tilde tg)\sum_{x\in V'_\Q}\Phi_a(\tilde t gx)\, dg\\
&=\chi(\det t)\int_{G_\R^+\mathcal K_NG_\Q/G_\Q}|\det g|_\A^{2s}\widetilde\chi(\det g)\sum_{x\in V'_\Q}\Phi_{t^{-1}a}(gx)\, dg\\
&=\chi(\det t)\int_{G_\R^+\mathcal K_N/G_\R^+\mathcal K_N\cap G_\Q}|\det g|_\A^{2s}\widetilde\chi(\det g)\sum_{x\in V'_\Q}\Phi_{t^{-1}a}(gx)\, dg\\
&=\chi(\det t)\int_{G_\R^+/\Gamma(N)\times\mathcal K_N}
|\det g_\infty|_\infty^{2s}
\sum_{x\in V'_\Q}\Phi_\infty(g_\infty x)\Phi_{\rf,t^{-1}a}(g_\rf x)\, dg_\infty dg_\rf \\
&=\chi(\det t)\int_{\mathcal K_N}dg_\rf \int_{G_\R^+/\Gamma(N)}
|\det g_\infty|_\infty^{2s}
\sum_{x\in V'_\Q}\Phi_\infty(g_\infty x)\Phi_{\rf,t^{-1}a}(x)\, dg_\infty\\
&=\frac{\chi(\det t)}{|G_N|}\int_{G_\R^+/\Gamma(N)}
|\det g_\infty|_\infty^{2s}
\sum_{x\in V'_\Q\cap (t^{-1}a+NV_{\widehat\Z})}\Phi_\infty(g_\infty x)\, dg_\infty\\
&=\frac{\chi(\det t)}{|T_N|}\frac{1}{[\Gamma(1):\Gamma(N)]}
\int_{G_\R^+/\Gamma(N)}
|\det g_\infty|_\infty^{2s}
\sum_{x\in V'_\Q\cap (t^{-1}a+NV_\Z)}\Phi_\infty(g_\infty x)\, dg_\infty\\
&=\frac{\chi(\det t)}{|T_N|}
\Gamma_\infty(\Phi_\infty,s)\xi(s,t^{-1}a).
\end{align*}
By \eqref{eq:sum_TN},
the formula is obtained by summing this formula
over all $t\in T_N$.
We now explain the process of modifications above.

The first equality follows from
$G_\R^+\mathcal K_N\tilde t=\tilde t G_\R^+\mathcal K_N$.
Since $\widetilde\chi$ and $\tilde t$ are both lifts,
$\widetilde\chi(\det\tilde t)=\chi(\det t)$
and $\tilde t\in G_{\widehat\Z}$ implies $|\det\tilde t|_\A=1$.
Also by definition $\Phi_a(\tilde tx)=\Phi_{t^{-1}a}(x)$.
Hence the second equality follows.
The third equality is obvious.
Since
$|\det(\mathcal K_N)|_\A=\widetilde\chi(\det(G_\R^+\mathcal K_N)=1$,
we have the fourth. The fifth equality is because $\Phi_{\rf,t^{-1}a}$
is $\mathcal K_N$-invariant.
Since $\int_{\mathcal K_N}dg_\rf 
=[G_{\widehat\Z}:\mathcal K_N]^{-1}\cdot \int_{G_{\widehat\Z}}dg_\rf 
=|G_N|^{-1}$ and $\Phi_{\rf,t^{-1}a}\in\sS(V_{\A_f})$
is the characteristic function of
$t^{-1}a+NV_{\widehat\Z}$, we have the sixth equality.
By $V_\Q\cap(t^{-1}a+NV_{\widehat\Z})=t^{-1}a+NV_\Z$,
we have the seventh. For the last equality,
we divide $V'_\Q\cap (t^{-1}a+NV_{\widehat\Z})$ into
its $\Gamma(N)$-orbits and consider each integral
separately. Then this equality follows from the unfolding method.
Note that
\[
\int_{G_\R^+/\Gamma(N)_x}
|\det g_\infty|_\infty^{2s}\Phi_\infty(g_\infty x)\, dg_\infty
=
\frac{|\Gamma(N)_x|^{-1}}{|P(x)|^s}
\int_{G_\R^+}
|P(g_\infty x)|_\infty^{s}\Phi_\infty(g_\infty x)\, dg_\infty.
\]
The second statement of the proposition follows from Wright's study
\cite{wright} of $Z(\Phi,s,\omega)$ combined with
the standard argument treating $\Gamma_\infty(\Phi_\infty,s)$.
For this, see the proof of \cite[Theorem 2.1]{shintania}, for example.
\end{proof}

\subsection{Residues as integrals}
We now start to compute the residues of $\xi(s,\chi,a)$.
We introduce the following constants.
\begin{defn}\label{defn:albtgm}
We define
\begin{equation}\label{eq:albtgm}
\alpha	:=\frac{\pi^2}{36}\begin{pmatrix}1\\3\\\end{pmatrix},
\quad
\beta	:=\frac{\pi^2}{12}\begin{pmatrix}1\\1\\\end{pmatrix},
\quad
\gamma	:=\frac{2\pi^2}{9\Gamma(2/3)^3}\begin{pmatrix}1\\\sqrt3\\\end{pmatrix}.
\end{equation}
\end{defn}
Shintani \cite{shintania} proved the following.
\begin{thm}[Shintani]\label{thm:residueShintani}
The residues of $\xi(s)$ are
\[
\res_{s=1}\xi(s)=\alpha+\beta,
\qquad
\res_{s=5/6}\xi(s)=\zeta(1/3)\gamma.
\]
\end{thm}

We define an averaging operator
$\cM_{\rf,\chi}$ on
$\sS(V_{\A_\rf})$ by
\begin{equation}\label{eq:average_global}
\cM_{\rf,\chi}\Phi_\rf(x)
=\int_{G_{\widehat\Z}}\tilde\chi(\det g_\rf)\Phi_\rf( g_\rf x)d g_\rf
\qquad
\left(\Phi_\rf\in\sS(V_{\A_\rf})\right).
\end{equation}
If $\chi$ is trivial then we write $\cM_{\rf}$ as well.
For $\Phi_\rf\in\sS(V_{\A_\rf})$, let
\begin{align*}
\esB_\rf(\Phi_{\rf})
&:=\zeta(2)^{-1}
\int_{\A_\rf^\times\times\A_\rf^2}
|t|_\rf^{2}\cM_{\rf}\Phi_{\rf}(0,t,u_3,u_4)\md t du_3du_4,\\
\esC_\rf(\Phi_{\rf},\chi,s)
&:=\int_{\A_\rf^\times\times\A_\rf^3}
\tilde\chi(t)|t|_\rf^{s}\cM_{\rf,\chi^{-1}}\Phi_{\rf}(t,u_2,u_3,u_4)\md t du_2du_3du_4,\\
\esC_\rf(\Phi_{\rf},\chi)
&:=\esC_\rf(\Phi_{\rf},\chi,1/3).
\end{align*}
We note that $\esC_\rf(\Phi_{\rf},\chi,s)$ is defined for $\Re(s)\leq1$
by analytic continuation.
Then for $a\in V_N$, by \cite[Theorem 6.1]{wright}
and Proposition \ref{prop:unfold}
(see \cite[pp. 69-70]{dawra} also),
we have
\begin{equation}\label{eq:residue_distribution}
\underset{s=1}{\res}\ \frac{\xi(s,\chi,a)}{|G_N|}
=\delta(\chi)\left(\frac{1}{N^4}\alpha+\esB_\rf(\Phi_{\rf,a})\beta\right),
\qquad
\underset{s=5/6}{\res}\ \frac{\xi(s,\chi,a)}{|G_N|}
=\delta(\chi^3)\esC_\rf(\Phi_{\rf,a},\chi)\gamma.
\end{equation}

Hence we will compute these values explicitly.
We define a local averaging operator
$\cM_{p,\chi}$ on $\sS(V_{\Q_p})$ similarly to
\eqref{eq:average_global}:
\begin{equation}\label{eq:average_local}
\cM_{p,\chi}\Phi_p(x)
=\int_{G_{\Z_p}}\tilde\chi_p(\det g_p)\Phi_p( g_p x)d g_p
\qquad
\left(\Phi_p\in\sS(V_{\Q_p})\right).
\end{equation}
Here we normalize the measure $d g_p$
on $G_{\Q_p}$ such that the volume of $G_{\Z_p}$ is $1$.
For $a\in V_{p^e}$ or $a\in V_{\Z_p}$ and $e \geq 0$, let
$\Phi_{p,a}\in\sS(V_{\Z_p})$ be the characteristic
function of $a+p^eV_{\Z_p}$.
We put
\begin{align}
\esB_{p^e}(a)
	&:=p^{4e}(1-p^{-2})\int_{\Q_p^\times\times \Q_p^2}
		|t|_p^2\mathcal M_p\Phi_{p,a}(0,t,u_3,u_4)d^\times tdu_3du_4,
\label{eq:defb}\\
\esC_{p^e}(a,\chi)
&	:=p^{4e}L_p(\frac13,\chi^{-1})^{-1}
	\int_{\Q_p^\times\times \Q_p^3}
		\tilde\chi_p(t)|t|_p^{1/3}\mathcal M_{p,\chi^{-1}}\Phi_{p,a}(t,u_2,u_3,u_4)d^\times tdu_2du_3du_4.
\label{eq:defc}
\end{align}
Note that these are $1$ if $e=0$ and $\chi$ is unramified at $p$.
Let $N=\prod p^{e_p}$ be the prime decomposition of $N$.
For $a\in V_N$, we define
\begin{equation}\label{eq:residue_Eulerp}
\esB_{N}(a):=\prod_{p\mid N}\esB_{p^{e_p}}(a_p),
\quad
\esC_{N}(a,\chi):=\prod_{p\mid N}\esC_{p^{e_p}}(a_p,\chi),
\qquad
\text{where $a_p=(a\mod p^e)\in V_{p^e}$}.
\end{equation}
Then
\[
\esB_\rf(\Phi_{\rf,a})
=N^{-4}\esB_{N}(a),
\qquad
\esC_\rf(\Phi_{\rf,a},\chi)
=N^{-4}\esC_{N}(a,\chi)L(1/3,\chi^{-1}).
\]
Hence by \eqref{eq:residue_distribution}
we have the following generalization
of Theorem \ref{thm:residueShintani}.
\begin{thm}\label{thm:residue_formal}
We have
\[
\underset{s=1}{\res}\ \frac{\xi(s,\chi,a)}{|G_N|}
=\delta(\chi)\left(\frac{1}{N^4}\alpha+\frac{\esB_N(a)}{N^4}\beta\right),
\quad
\underset{s=5/6}{\res}\ \frac{\xi(s,\chi,a)}{|G_N|}
=\delta(\chi^3)\frac{\esC_{N}(a,\chi)}{N^4}L(\frac13,\chi^{-1})\gamma.
\]
\end{thm}

We now describe the residues of $\xi(s,f)$ for $f\in C(V_N,\chi)$.
For $f\in C(V_N)$
we define
\begin{equation}\label{eq:residue_f}
\esA_N(f)=\sum_{a\in V_N}\frac{f(a)}{N^4},
\quad
\esB_N(f)=\sum_{a\in V_N}\frac{f(a)\esB_N(a)}{N^4},
\quad
\esC_N(f,\chi)=\sum_{a\in V_N}\frac{f(a)\esC_N(a,\chi)}{N^4}.
\end{equation}
By definition these are multiplicative.
By Theorem \ref{thm:residue_formal} and
Proposition \ref{prop:f_inv_L_ogs},
we have:
\begin{prop}\label{prop:residue_f}
Let $f\in C(V_N,\chi)$. Then
\[
\underset{s=1}{\res}\ \xi(s,f)=
\delta(\chi)\left(\esA_N(f)\alpha+\esB_N(f)\beta\right),
\quad
\underset{s=5/6}{\res}\ \xi(s,f)
=\delta(\chi^3)\esC_{N}(f,\chi)L(\frac13,\chi^{-1})\gamma.
\]
\end{prop}
Since $\xi(s,\chi,a)$ and $\xi(s,f)$ for $f\in C(V_N,\chi)$
are holomorphic if $\chi^3$ is non-trivial, we assume
that $\chi^3$ is trivial for the rest of this section:
\begin{asmp}
The Dirichlet character $\chi$ is cubic.
\end{asmp}
This of course implies that $\tilde\chi$, $\tilde\chi_p$,
$\chi_p$ and $\chi_p'$ are cubic characters also.

\subsection{Preliminaries for explicit computation}
Now our aim is to compute
$\esB_{p^e}(a)$ and $\esC_{p^e}(a,\chi)$
explicitly.
In this subsection we prepare several lemmas for the computation.
We fix a prime $p$ and denote the $p$-part of
the conductor of $\chi$ by $p^c$.
We assume $e\geq c$.

We first describe some basic properties satisfied
by $\esB_{p^e}(a)$ and $\esC_{p^e}(a,\chi)$.
\begin{lem}\label{lem:residue_basic}
\begin{enumerate}[{\rm (1)}]
\item
For $g\in G_{p^e}$,
\begin{equation}\label{eq:bc_invariant}
\esB_{p^e}(ga)=\esB_{p^e}(a),
\qquad
\esC_{p^e}(ga,\chi)=\chi_p(\det g)^{-1}\esC_{p^e}(a,\chi).
\end{equation}
\item
If $\chi$ is unramified at $p$, then
\begin{equation}\label{eq:bc_sum}
p^{-4e}\sum_{a\in V_{p^e}}\esB_{p^e}(a)
=1,
\qquad
p^{-4e}\sum_{a\in V_{p^e}}\esC_{p^e}(a,\chi)
=1.
\end{equation}
\item
Let $m\leq e-c$. For $a\in V_{p^{e-m}}$, we regard
$p^ma\in V_{p^{e}}$.
Then
\begin{equation}\label{eq:bc_multiple}
\esB_{p^e}(p^ma)=\esB_{p^{e-m}}(a),
\qquad
\esC_{p^e}(p^ma,\chi)=\tilde\chi_p(p)^mp^{2m/3}\esC_{p^{e-m}}(a,\chi).
\end{equation}
\end{enumerate}
\end{lem}
\begin{proof}
These immediately follow from the definitions
\eqref{eq:defb}, \eqref{eq:defc}.
For (2), we note that if $\chi$ is unramified,
$\sum_{a\in V_{p^e}}\cM_{p,\chi^{-1}}\Phi_{p,a}
=\sum_{a\in V_{p^e}}\Phi_{p,a}$ is the characteristic function
of $V_{\Z_p}$.
\end{proof}
\begin{lem}\label{lem:residue_orbit}
Let $a\in V_{p^e}$.
Assume that
$G_{p^e}a+p^eV_{\Z_p}$ is a single $G_{\Z_p}$-orbit
and that $\chi_p\circ\det$ is trivial on $G_{p^e,a}$.
Then
\begin{equation}
\esB_{p^e}(a)=\esB_{p^{e'}}(a'),
\quad
\esC_{p^e}(a,\chi)=\esC_{p^{e'}}(a',\chi),
\end{equation}
for all $e'\geq e$ and $a'\in V_{p^{e'}}$ such that
$a'\mod p^e=a$.
\end{lem}
\begin{proof}
By \eqref{eq:defb} and \eqref{eq:defc}, it is enough to show that
\[
p^{4e}\mathcal M_{p,\chi^{-1}}\Phi_{p,a}=
p^{4e'}\mathcal M_{p,\chi^{-1}}\Phi_{p,a'}.
\]
Let $\tilde a\in V_{\Z_p}$ be a lift of $a$.
Then by assumption $G_{p^e}a+p^eV_{\Z_p}=G_{\Z_p}\tilde a$.
We claim that
\begin{equation}\label{eq:residue_orbit}
\mathcal M_{p,\chi^{-1}}\Phi_{p,a}(x)
=\begin{cases}
|G_{p^e,a}||G_{p^e}|^{-1}\chi_p(\det k)  & x=k\tilde a , k\in G_{\Z_p},\\
0	&x\notin G_{\Z_p}\tilde a.
\end{cases}
\end{equation}
By \eqref{eq:average_local} and Lemma \ref{lem:character},
\begin{equation}\label{eq:average_unfold}
\cM_{p,\chi^{-1}}\Phi_{p,a}
=
|G_{p^e}|^{-1}\sum_{g\in G_{p^e}}
	\chi_p(\det g)\Phi_{p,ga}.
\end{equation}
Hence $\mathcal M_{p,\chi^{-1}}\Phi_{p,a}(x)$
vanishes unless $x\in G_{p^e}a+p^eV_{\Z_p}$.
Let $x=k\tilde a, k\in G_{\Z_p}$.
Then
\begin{align*}
\mathcal M_{p,\chi^{-1}}\Phi_{p,a}(k\tilde a)
&=\chi_p(\det k)|G_{p^e}|^{-1}\sum_{g\in G_{p^e}}
	\chi_p(\det g)\Phi_{p,ga}(\tilde a)\\
&=\chi_p(\det k)|G_{p^e}|^{-1}\sum_{g\in G_{p^e,a}}\chi_p(\det g)
=|G_{p^e,a}||G_{p^e}|^{-1}\chi_p(\det k),
\end{align*}
where the last equality follows from the assumption
and hence we have \eqref{eq:residue_orbit}.
Now it is enough to show that $|G_{p^{e'},a'}|=|G_{p^{e},a}|$.
For this, note the identity
$G_{p^{e'}}a'+p^{e'}V_{\Z_p}=G_{p^e}a+p^eV_{\Z_p}$.
These volumes are
$p^{-4e'}|G_{p^{e'}}|/|G_{p^{e'},a'}|$ and
$p^{-4e}|G_{p^{e}}|/|G_{p^{e},a}|$ respectively,
and hence we have $|G_{p^{e'},a'}|=|G_{p^{e},a}|$.
This finishes the proof.
\end{proof}

To begin our computation of $\esB_{p^e}(a)$ and $\esC_{p^e}(a,\chi)$,
we introduce the following notation.
\begin{defn}
\begin{enumerate}[{\rm (1)}]
\item
For $a=(a_1,a_2,a_3,a_4)\in V_{\Z_p}$, we define
\begin{equation}\label{eq:b'}
\esB_{p^e}'(a)
:=\begin{cases}
0	& a_1\notin p^eV_{\Z_p},\\
p^e(1+p^{-1})|a_2|_p & a_1\in p^e V_{\Z_p},a_2\notin p^e V_{\Z_p},\\
1& a_1,a_2\in p^e V_{\Z_p}.\\
\end{cases}
\end{equation}
\item
For  $y\in\Z_p$, we define
$I_{p^e}(y,\chi)$ as follows.
If $c=0$, we put
\begin{equation}\label{eq:i_ur}
I_{p^e}(y,\chi)
:=\begin{cases}
p^{2e/3}\tilde\chi_p(p)^e& \ord_p(y)\geq e,\\
\dfrac{1-\tilde\chi_p(p)p^{-1/3}}{1-p^{-1}}{\tilde\chi_p(y)}|y|_p^{-2/3} & \ord_p(y)<e.
\end{cases}
\end{equation}
If $c\geq 1$, then we put
\begin{equation}\label{eq:i_rm}
I_{p^e}(y,\chi)
:=\begin{cases}
(1-p^{-1})^{-1}{\tilde\chi_p(y)}|y|_p^{-2/3} & \ord_p(y)\leq e-c,\\
0 & \ord_p(y)> e-c.\end{cases}
\end{equation}
\end{enumerate}
Since these quantities respectively depend only on $(a\mod p^e)\in V_{p^e}$ or
$(y\mod p^e)\in \Z/p^e\Z$, we use the same notation
for $a\in V_{p^e}$ or $y\in \Z/p^e\Z$ as well.
\end{defn}
An easy computation shows that
\begin{align}
\esB_{p^e}'(a)
&=p^{4e}(1-p^{-2})\int_{\Q_p^\times\times \Q_p^2}
		|t|_p^2\Phi_{p,a}(0,t,u_3,u_4)d^\times tdu_3du_4,
\label{eq:b'int}\\
I_{p^e}(y,\chi)
&=p^eL_p(\frac13,\chi^{-1})^{-1}\int_{y+p^e\Z_p}\tilde\chi_p(t)|t|_p^{1/3}\md t.
\label{eq:i_int}
\end{align}

Let $W=\aff^2$ be the affine space of 2-dimensional row vectors.
The $G$ naturally acts on $W$ from the right.
We put $W_{p^e}:=W_{\Z/p^e\Z}=(\Z/p^e\Z)^2$
and
\[
W'_{p^e}:=\{(u,v)\in (\Z/p^e\Z)^2\mid
\text{$u\in (\Z/p^e\Z)^\times$ or $v\in (\Z/p^e\Z)^\times$}\},
\]
which is the $G_{p^e}$-orbit of $(1,0)\in W_{p^e}$.
The following formulas for
$\esB_{p^e}(a)$ and $\esC_{p^e}(a,\chi)$ hold.
\begin{lem}\label{lem:residue_finiteorbit}
\begin{enumerate}[{\rm (1)}]
\item
We have
\begin{equation}\label{eq:bb'}
\esB_{p^e}(a)=|G_{p^e}|^{-1}\sum_{g\in G_{p^e}}\esB_{p^e}'(ga).
\end{equation}
\item
If $\chi$ is unramified at $p$ (i.e., if $c=0$), we have
\begin{equation}\label{eq:cg}
\esC_{p^e}(a,\chi)
=|G_{p^e}|^{-1}\sum_{g\in G_{p^e}}I_{p^e}((ga)_1,\chi),
\end{equation}
where $(ga)_1\in\Z/p^e\Z$ is the first entry of $ga\in V_{p^e}$.
Moreover for any $\chi$,
\begin{equation}\label{eq:cw}
\esC_{p^e}(a,\chi)
=|W'_{p^e}|^{-1}\sum_{(u,v)\in W'_{p^e}}I_{p^e}(a(u,v),\chi).
\end{equation}
Here we are plugging particular values of $u$ and $v$ into the
binary cubic form $a$.
\end{enumerate}
\end{lem}
\begin{proof}
(1) is obtained by \eqref{eq:average_unfold} with $\chi_p$ trivial
and \eqref{eq:b'int}.
We consider (2). Let
\begin{equation*}\label{eq:defc'}
\esC_{p^e}'(a,\chi)
	=p^{4e}L_p(\frac13,\chi^{-1})^{-1}\int_{\Q_p^\times\times \Q_p^3}
		\tilde\chi_p(t)|t|_p^{1/3}\Phi_{p,a}(t,u_2,u_3,u_4)d^\times tdu_2du_3du_4.
\end{equation*}
Then by \eqref{eq:i_int} we have
$\esC_{p^e}'(a,\chi)
=I_{p^e}(a_1,\chi)$.
Hence by \eqref{eq:average_unfold} we have
\[
\esC_{p^e}(a,\chi)
=|G_{p^e}|^{-1}\sum_{g\in G_{p^e}}\chi_p(\det g)\esC_{p^e}'(ga,\chi)
=|G_{p^e}|^{-1}\sum_{g\in G_{p^e}}\chi_p(\det g)I_{p^e}((ga)_1,\chi).
\]
In particular we have \eqref{eq:cg}. Moreover, note that
$(ga)_1=(\det g)^{-1}a((1,0)g)$.
Since $I_{p^e}(ty,\chi)=\chi_p(t)I_{p^e}(y,\chi)$ for
$t\in (\Z/p^e\Z)^\times$, we have
\begin{align*}
\esC_{p^e}(a,\chi)
&=|G_{p^e}|^{-1}\sum_{g\in G_{p^e}}I_{p^e}(a((1,0)g),\chi)
=|W'_{p^e}|^{-1}\sum_{(u,v)\in W'_{p^e}}I_{p^e}(a(u,v),\chi),
\end{align*}
as desired.
\end{proof}

\subsection{Unramified computation}
\label{subsec:ur}

In this subsection, we compute
$\esB_{p^e}(a)$ and $\esC_{p^e}(a,\chi)$ for
$p\nmid m$, that is, when $\chi$ is unramified at $p$.
By \eqref{eq:bc_invariant},
these depend only on
the $G_{p^e}$-orbit of $a$.
For $\esC_{p^e}$, we list the values of $(1-p^{-2})\esC_{p^e}(a,\chi)$
for convenience.

When $e=1$, we have the following.
\begin{prop}\label{prop:residue_ur_1}
For $e=1$, 
$\esB_p(a)$ and $(1-p^{-2})\esC_p(a,\chi)$
are given by:
\[
\begin{array}{c|c|l}
\hline
\text{\rm Type of $a$}
&\esB_p(a)
&(1-p^{-2})\esC_p(a,\chi)\\
\hline
\hline
(3)
&0
&(1-\chi(p)^{2}p^{-1/3})(1+p^{-1})\\
(21)
&1
&1-\chi(p)^{2}p^{-4/3}\\
(111)
&3
&(1-\chi(p)p^{-2/3})(1+\chi(p)^{2}p^{-1/3})^2\\
(1^21)
&\frac{p+2}{p+1}
&(1+\chi(p)^{2}p^{-1/3})(1-p^{-1})\\
(1^3)
&\frac{1}{p+1}
&1-\chi(p)^{2}p^{-4/3}\\
(0)
&1
&(1-p^{-2})\chi(p)^2p^{2/3}\\
\hline
\end{array}
\]
\end{prop}
\begin{proof}
For the computation of $\esC_{p^e}(a,\chi)$,
it is convenient to put
\begin{equation}\label{eq:x}
\ww:=\tilde\chi_p(p)p^{-1/3}=\chi(p)^{2}p^{-1/3}.
\end{equation}
We note that since $\chi$ is a cubic character, $\ww^3=p^{-1}$.
Let
$n_p^0(\sigma)=|\{a\in V_p(\sigma)\mid a_1\neq0\}|$,
$n_p^1(\sigma)=|\{a\in V_p(\sigma)\mid a_1=0,a_2\neq0\}|$, and
$n_p^2(\sigma)=|\{a\in V_p(\sigma)\mid a_1=a_2=0\}|$.
Note that $\sum_{0\leq i\leq 2} n_p^i(\sigma)=n_p(\sigma)$.
Then for $a\in V_p(\sigma)$,
respectively by \eqref{eq:bb'}, \eqref{eq:b'} and
by \eqref{eq:cg}, \eqref{eq:i_ur},
we have
\begin{align*}
\esB_p(a)
&=\frac{n_p^0(\sigma)}{n_p(\sigma)}\cdot 0
+\frac{n_p^1(\sigma)}{n_p(\sigma)}\cdot p(1+p^{-1})
+\frac{n_p^2(\sigma)}{n_p(\sigma)}\cdot 1,\\
\esC_p(a,\chi)
&=\frac{n_p^0(\sigma)}{n_p(\sigma)}\cdot\frac{1-\ww}{1-\ww^3}
+\frac{n_p^1(\sigma)+n_p^2(\sigma)}{n_p(\sigma)}\cdot\frac{1}{\ww^2}.
\end{align*}
For $(\sigma)=(1^21)$, we see that
the ratio $n_p^0(\sigma):n_p^1(\sigma):n_p^2(\sigma)$
is $(p-1):1:1$ and hence
\begin{align*}
\esB_p(a)
&=\frac{1}{p+1}\cdot p(1+p^{-1})+\frac{1}{p+1}\cdot 1=\frac{p+2}{p+1},\\
\esC_p(a,\chi)
&=\frac{p-1}{p+1}\cdot \frac{1-\ww}{1-\ww^3}+
\frac{2}{p+1}\cdot \frac{1}{\ww^2}
=\frac{1-\ww^3}{1+\ww^3}\cdot\frac{1-\ww}{1-\ww^3}
+\frac{2\ww^3}{1+\ww^3}\cdot\frac{1}{\ww^2}
=\frac{1+\ww}{1+\ww^3}.
\end{align*}
The other cases are computed similarly.
For $(\sigma)=(3), (21), (111), (1^21), (1^3)$, and $(0)$,
the ratio $n_p^0(\sigma):n_p^1(\sigma):n_p^2(\sigma)$
is $1:0:0$, $p:1:0$, $(p-2):3:0$, $(p-1):1:1$, $p:0:1$,
and $0:0:1$ respectively, and the result follows.
\end{proof}

We now consider the case $e\geq2$. In connection with counting cubic fields,
we mainly work for $a\in V_{p^e}^{\max}$, i.e.,
of type 
$(3), (21), (111), (1^21_{\max})$, or $(1^3_{\max})$.
This is a generalization of results of
Datskovsky-Wright \cite[Theorem 5.2 and Proposition 5.3]{dawra}
to unramified characters.

\begin{prop}\label{prop:residue_ur_max}
Let $e\geq2$
and $a\in V_{p^e}^{\max}$.
Then
$\esB_{p^e}(a)$ and $(1-p^{-2})\esC_{p^e}(a,\chi)$
are given by:
\[
\begin{array}{c|c|l}
\hline
\text{\rm Type of $a$}
&\esB_{p^e}(a)
&(1-p^{-2})\esC_{p^e}(a,\chi)\\
\hline
\hline
(3)
&0
&(1-\chi(p)^{2}p^{-1/3})(1+p^{-1})\\
(21)
&1
&1-\chi(p)^{2}p^{-4/3}\\
(111)
&3
&(1-\chi(p)p^{-2/3})(1+\chi(p)^{2}p^{-1/3})^2\\
(1^21_{\max})
&1
&(1+\chi(p)^{2}p^{-1/3})(1-\chi(p)p^{-2/3})\\
(1^3_{\max})
&0
&1-\chi(p)p^{-2/3}\\
\hline
\end{array}
\]
In particular, they depend only on the orbital type of $a$.
\end{prop}
\begin{proof}
For $(\sigma)=(3), (21), (111)$,
$V_{\Z_p}(\sigma)$ is a single $G_{\Z_p}$-orbit.
Hence,
if $a$ is of one of these types,
Lemma \ref{lem:residue_orbit} reduces our calculation to 
the case $e=1$, handled in Proposition \ref{prop:residue_ur_1}.
Therefore we consider the remaining cases. We put
$R=\Z/p^e\Z$, 
and again write $\ww = \chi(p)^2 p^{-1/3}$ as in \eqref{eq:x}.

First let $a$ be of type $(1^21_{\rm max})$.
By Theorem \ref{thm:defa}, each orbit in $V_{p^e}(1^21_{\max})$
contains some $a=(0,1,a_3,a_4)$
where $a_3\in pR$ and $a_4\in pR^\times$.
Let $g=\twtw qrst\in G_{p^e}$.
Then the first coordinate $(ga)_1$ of $ga$ is
$(ga)_1=(\det g)^{-1}r(q^2+a_3qr+a_4r^2)$.
Since $a_3\in pR$ and $a_4\in pR^\times$,
\[
\ord_p(q^2+a_3qr+a_4r^2)=
\begin{cases}
0 & \ord_p(q)=0,\\
1 & \ord_p(q)\geq1.
\end{cases}
\]
Hence $(ga)_1=0$ if and only if $r=0$
and in this case, $(ga)_2=q\in R^\times$.
Hence by \eqref{eq:b'}, 
$\esB_{p^e}'(ga)=p^e(1+p^{-1})$ if $r=0$ and $0$ otherwise.
Since
$|\{g\in G_{p^e}\mid r=0\}|=p^{3e}(1-p^{-1})^2$,
by \eqref{eq:bb'} we have
$\esB_{p^e}(a)=1$.

We compute $\esC_{p^e}(a,\chi)$ using
\eqref{eq:cw} and \eqref{eq:i_ur}.
Let $(u,v)\in W'_{p^e}$. Then $a(u,v)=v(u^2+a_3uv+a_4v^2)$.
If $u^2+a_3uv+a_4v^2\notin R^\times$, then
$u\notin R^\times$ and hence
$v\in R^\times$. Therefore,
\[
\ord_p(a(u,v))
=\begin{cases}
1 & \ord_p(u)\geq1,\\
\ord_p(v) & \ord_p(u)=0.\\
\end{cases}
\]
The cardinality of the subsets
$\{\ord_p(u)\geq1\}$,
$\{\ord_p(u)=0,\ord_p(v)=m<e\}$
and
$\{v=0\}$ of $W'_{p^e}$
are 
$p^{2e-1}(1-p^{-1})$,
$p^{2e-m}(1-p^{-1})^2$
and $p^{e}(1-p^{-1})$, respectively.
Hence by
\eqref{eq:cw} and \eqref{eq:i_ur},
\begin{align*}
\esC_{p^e}(a,\chi)
&=\frac{p^{2e-1}(1-p^{-1})}{p^{2e}(1-p^{-2})}\cdot\frac{(1-\ww)p\ww}{1-p^{-1}}
+\sum_{0\leq m<e}\frac{p^{2e-m}(1-p^{-1})^2}{p^{2e}(1-p^{-2})}\cdot\frac{(1-\ww)p^m\ww^m}{1-p^{-1}}\\
&\quad+\frac{p^{e}(1-p^{-1})}{p^{2e}(1-p^{-2})}\cdot p^e\ww^e
=\frac{(1+\ww)(1-\ww^2)}{1-p^{-2}}.
\end{align*}

Next let $a$ be of type $(1^3_{\rm max})$. We may assume
$a=(1,a_2,a_3,a_4)$ where $a_2,a_3\in pR$, $a_4\in pR^\times$.
Let $g\in G_{p^e}$ as above. Then
$(ga)_1=(\det g)^{-1}(q^3+a_2q^2r+a_3qr^2+a_4r^3)$.
Since $a_2,a_3\in pR$ and $a_4\in pR^\times$,
$\ord_p(q^3+a_2q^2r+a_3qr^2+a_4r^3)\leq 1$
and hence $(ga)_1$ is always nonzero. Hence $\esB_{p^e}(a)=0$.
Also the order of $\ord_p(a(u,v))$ is $0$ if $\ord_p(u)=0$
and $1$ otherwise. Hence
\[
\esC_{p^e}(a,\chi)
=\frac{p^{2e}(1-p^{-1})}{p^{2e}(1-p^{-2})}\cdot\frac{(1-\ww)}{1-p^{-1}}
+\frac{p^{2e-1}(1-p^{-1})}{p^{2e}(1-p^{-2})}\cdot\frac{(1-\ww)p\ww}{1-p^{-1}}
=\frac{1-\ww^2}{1-p^{-2}}.
\]
This finishes the proof.
\end{proof}

As corollaries we have the following.
Recall that we introduced
$f_p\in C(V_p)$, $\Phi_{p},\Phi_{p}'\in C(V_{p^2})$
in Definitions \ref{defn:f_p}, \ref{defn:f_p^2}
and the distributions $\esA_N,\esB_N,\esC_N$
in \eqref{eq:residue_f}.
\begin{cor}\label{cor:f_p}
Assume that $\chi$ is unramified at $p$. We have
\[
\esA_p(f_p)=\esB_p(f_p)
=\frac{1}{p}+\frac{1}{p^2}-\frac{1}{p^3},
\quad
\esC_p(f_p,\chi)
=\frac{1}{p}+\frac{\chi(p)^{2}}{p^{4/3}}-\frac{\chi(p)^{2}}{p^{7/3}}.
\]
\end{cor}
\begin{proof}
Let $\ww$ be as in \eqref{eq:x}.
By Lemma \ref{lem:n_p(sigma)} and Proposition \ref{prop:residue_ur_1},
we have
\begin{align*}
\esA_p(f_p)
&=\frac{p(p^2-1)}{p^4}+\frac{(p^2-1)}{p^4}+\frac{1}{p^4}
=\frac{1}{p}+\frac{1}{p^2}-\frac{1}{p^3},\\
\esB_p(f_p)
&=\frac{p(p^2-1)}{p^4}\frac{p+2}{p+1}+\frac{(p^2-1)}{p^4}\frac{1}{p+1}+\frac{1}{p^4}
=\frac{1}{p}+\frac{1}{p^2}-\frac{1}{p^3},\\
\esC_p(f_p,\chi)
&=\frac{p(p^2-1)}{p^4}\frac{1+\ww}{1+\ww^3}
+\frac{(p^2-1)}{p^4}\frac{1-\ww^4}{1-\ww^6}
+\frac{1}{p^4}\frac{1}{\ww^2}
=\ww^3+\ww^4-\ww^7,
\end{align*}
as desired. Again we note that $\ww^3=p^{-1}$.
\end{proof}
\begin{cor}\label{cor:f_p^2_residue}
Assume that $\chi$ is unramified at $p$. Then
\begin{align*}
\esA_{p^2}(\Phi_{p})&=\frac{1}{p^2}+\frac{1}{p^3}-\frac{1}{p^5},
&\esA_{p^2}(\Phi_{p}')&=\frac{2}{p^2}-\frac{1}{p^4},\\
\esB_{p^2}(\Phi_{p})&=\frac{2}{p^2}-\frac{1}{p^4},
&\esB_{p^2}(\Phi_{p}')&=\frac{2}{p^2}-\frac{1}{p^4},\\
\esC_{p^2}(\Phi_{p},\chi)&=\frac{\chi(p)}{p^{5/3}}+\frac{1}{p^{2}}-\frac{\chi(p)}{p^{11/3}},
&\esC_{p^2}(\Phi_{p}',\chi)&=\frac{\chi(p)}{p^{5/3}}+\frac{2}{p^{2}}-\frac{\chi(p)}{p^{8/3}}-\frac{1}{p^3}.
\end{align*}
\end{cor}
\begin{proof}
Let $\ww$ be as in \eqref{eq:x}.
By \eqref{eq:bc_sum}, we have
\[
\esC_{p^2}(\Phi_{p},\chi)
=p^{-8}\sum_{a\in V_{p^2}^{\rm nm}}\esC_{p^2}(a,\chi)
=1-p^{-8}\sum_{a\in V_{p^2}^{\rm max}}\esC_{p^2}(a,\chi).
\]
Hence by Lemma \ref{lem:n_p^2(sigma)} and Proposition \ref{prop:residue_ur_max},
\begin{align*}
\esC_{p^2}(\Phi_{p},\chi)
&=1-\frac{p^4(p^2-1)(p^2-p)}{p^8}\frac{1-\ww^2}{1-\ww^6}
\left(\frac{1-\ww+\ww^2}{3}+\frac{1+\ww^2}{2}+\frac{(1+\ww)^2}{6}+\frac{1+\ww}{p}+\frac{1}{p^2}\right)\\
&=1-(1-\ww^3)(1-\ww^2)(1+\ww^2+\ww^3+\ww^4+\ww^6)=\ww^5+\ww^6-\ww^{11},
\end{align*}
and the result follows.
By adding the contribution from $V_{p^2}(1^3_{\rm max})$, we also have
\[
\esC_{p^2}(\Phi_{p}',\chi)
=1-(1-\ww^3)(1-\ww^2)(1+\ww^2+\ww^3+\ww^4)
=\ww^5+2\ww^6-\ww^8-\ww^9,
\]
as desired.
The other formulas are proved similarly.
\end{proof}

For its own interest, when $e=2$ we compute
$\esB_{p^e}(a)$ and $\esC_{p^2}(a,\chi)$
for $a\in V_{p^2}^{\rm nm}$ also. By Lemma \ref{lem:residue_basic} (3),
we may assume $a\in V_{p^2}^{\rm nm}\setminus pV_{p^2}$.
\begin{prop}\label{prop:residue_ur_nm}
Let $a\in V_{p^2}$
be of type $(1^21_{\ast})$, $(1^3_{\ast})$ or $(1^3_{\ast\ast})$.
Then $\esB_{p^e}(a)$ and $(1-p^{-2})\esC_{p^2}(a,\chi)$
are given in the following table.
\[
\begin{array}{c|c|c}
\hline
\text{\rm Type of $a$} & \esB_{p^2}(a) & (1-p^{-2})\esC_{p^2}(a,\chi)\\
\hline\hline
& \frac{2p+1}{p+1}\ (p\neq2)\\
(1^21_{\ast}) & \frac43\ (p=2, a=(0,1,0,0)) & (1+\chi(p)p^{1/3})(1-p^{-1})\\
& \frac53\ (p=2, a=(0,1,2,0))\\
\hline
(1^3_\ast) & 1 & (1+\chi(p)p^{1/3})(1-\chi(p)p^{-2/3})\\
\hline
(1^3_{\ast\ast}) & \frac{1}{p+1} & (1+\chi(p)p^{1/3})(1-\chi(p)p^{-2/3})\\
\hline
\end{array}
\]
\end{prop}
\begin{proof}
Since the proof is quite similar to Proposition \ref{prop:residue_ur_max},
we shall be brief. Let $R=\Z/p^2\Z$.

Let $a=(0,1,a_2,0)\in\cD_{p^2}(1^21_\ast)$.
We first compute $\esC_{p^2}(a,\chi)$.
Let $(u,v)\in W'_{p^2}$. Then
\[
\ord_p(a(u,v))=\ord_p(uv(u+a_2v))=
\begin{cases}
2 & u\in pR,\\
\ord_p(v) & u\in R^\times.\\
\end{cases}\]
Hence by \eqref{eq:cw} and \eqref{eq:i_ur},
\[
\esC_{p^2}(a,\chi)
=\frac{(1-p^{-1})^2}{1-p^{-2}}\cdot\frac{1-x}{1-p^{-1}}
+\frac{p^{-1}(1-p^{-1})^2}{1-p^{-2}}\cdot\frac{(1-x)px}{1-p^{-1}}
+\frac{p^{-1}(1-p^{-2})}{1-p^{-2}}\cdot p^2x^2
=\frac{1+x^{-1}}{1+p^{-1}}.
\]
We consider $\esB_{p^2}(a)$.
Let $a=(0,1,0,0)$ and $g=\twtw qrst\in G_{p^2}$.
Then $(ga)_1=0$ if and only if $r=0$ or $q\in pR$.
Moreover,
if $r=0$ then $(ga)_2\in R^\times$ and
if $q\in pR$ then $(ga)_2\in 2qR^\times$.
Then by \eqref{eq:bb'} and \eqref{eq:b'}, if $p\neq2$
we have
\[
\esB_{p^2}(a)
=\frac{p^{-2}(1-p^{-1})}{1-p^{-2}}\cdot p^2(1+p^{-1})
+\frac{p^{-1}(1-p^{-1})^2}{1-p^{-2}}\cdot p(1+p^{-1})
+\frac{p^{-2}(1-p^{-1})}{1-p^{-2}}\cdot1
=\frac{2p+1}{p+1},
\]
and if $p=2$ we have
\[
\esB_{p^2}(a)
=\frac{p^{-2}(1-p^{-1})}{1-p^{-2}}\cdot p^2(1+p^{-1})
+\frac{p^{-1}(1-p^{-1})}{1-p^{-2}}\cdot1
=\frac{p+2}{p+1}=\frac43.
\]
When $p\neq2$, $V_{p^2}(1^21_\ast)$
is a single orbit by Proposition \ref{prop:singular_orbits}
and hence this is enough.
When $p=2$ we see numerically
that there is one other orbit represented by $a=(0,1,2,0)$
in $V_{p^2}(1^21_\ast)$, and by a similar consideration
we have $\esB_{p^2}(a)=\frac53$ for this $a$.

This finishes the proof for type $(1^21_\ast)$.
The arguments for types $(1^3_\ast)$ and $(1^3_{\ast\ast})$
are similar and easier, so we omit the detail.
\end{proof}

\subsection{Ramified computation}
\label{subsec:rm}

In this subsection, we compute $\esC_{p^e}(a,\chi)$
for $p\mid m$, that is, at the primes $p$ where $\chi$ is ramified.
Since $\chi$ is cubic,
either $p\equiv1\mod 3$ or $p=3$,
and the conductor $p^c$ of $\chi$ is
\[
p^c=\begin{cases}
p & p\equiv1\mod 3,\\
p^2 & p=3.
\end{cases}
\]
We assume $e\geq c$.
We mainly work for $p\equiv1\mod 3$.
For $p=3$, the computation seems to be more complicated theoretically.
Fortunately there are only finitely many cases
so we simply use PARI/GP \cite{pari} for evaluation.

We first treat the case $e=1$ and (hence) $p\equiv1\mod 3$. 
The following is a refinement
of \cite[Proposition 5.4]{dawra}.
As in \cite{dawra},
we will encounter a curious ``cubic character sum
of a cubic polynomial'' which was
evaluated by Wright \cite{wrightc}.
\begin{prop}\label{prop:residue_rm_1}
Assume $p\equiv1\mod 3$.
Let $e=1$. We have
\begin{equation*}
(1-p^{-2})\esC_{p}(a,\chi)
=\begin{cases}
p^{-2}\tau(\chi_p)^3\chi_p(P(a)) & \text{$a$ : of type $(3)$, $(111)$},\\
-p^{-2}\tau(\chi_p)^3\chi_p(P(a)) & \text{$a$ : of type $(2)$},\\
0 &\text{$a$ : of type $(1^21)$, $(0)$},\\
\chi_p(\det g)^2 &\text{$a$ : of type $(1^3)$, $a=g(1,0,0,0), g\in G_p$}.\\
\end{cases}
\end{equation*}
Here
$\tau(\chi_p)=\sum_{t\in\mathbb F_p^\times}\chi_p(t)\exp(2\pi it/p)$
is the usual Gauss sum.
\end{prop}
\begin{proof}
If $a$ is of type $(1^21)$ or $(0)$, then
$\chi_p\circ\det$ is non-trivial on ${G_{p,a}}$ and hence
$\esC_p(a,\chi)=0$. We consider the other cases.
By \eqref{eq:cw} and \eqref{eq:i_rm} we have
\[
\esC_{p}(a,\chi)
=\frac{1}{(p^2-1)(1-p^{-1})}
\sum_{\substack{(u,v)\in\mathbb F_p^2,\ a(u,v)\neq0}}\chi_p(a(u,v)).
\]
The sum in the right hand side was studied by Wright \cite{wrightc}.
Let $J(\chi_p,\chi_p)
=\sum_{t\in\mathbb F_p^\times, t\neq1}\chi_p(t)\chi_p(1-t)$
be the Jacobi sum. If $a$ is of type $(3)$, $(21)$ or $(111)$,
then by \cite[Theorem 1]{wrightc},
\[
\sum_{\substack{(u,v)\in\mathbb F_p^2,\ a(u,v)\neq0}}\chi_p(a(u,v))
=\pm(p-1)\chi_p(P(a))J(\chi_p,\chi_p),
\]
where the sign is $+$ if $a$ is of type $(3)$ or $(111)$
and $-$ if $a$ is of type $(2)$. 
Since $\chi_p^2=\overline\chi_p\neq{\bf 1}$, we have
\[
J(\chi_p,\chi_p)
=\tau(\chi_p)^2/\tau(\overline\chi_p)
=\tau(\chi_p)^2\cdot \chi_p(-1)\tau(\chi_p)/p
=\tau(\chi_p)^3/p,
\]
and the result follows.
Let $a$ be of type $(1^3)$.
We have $a=ga_0$ for some $g\in G_p$,
where $a_0=(1,0,0,0)$.
Then $\esC_{p}(a,\chi)
=\chi_p(\det g)^{2}\esC_{p}(a_0,\chi)$
and
\[
\sum_{\substack{(u,v)\in\mathbb F_p^2,\ a_0(u,v)\neq0}}\chi_p(a_0(u,v))
=\sum_{u\in\mathbb F_p^\times,v\in\mathbb F_p}\chi_p(u^3)
=\sum_{u\in\mathbb F_p^\times,v\in\mathbb F_p}1
=p(p-1).
\]
Hence we have the formula.
\end{proof}

We now study the case $e\geq2$.
When $e=2$ and $a\in V_{p^2}^{\rm nm}$,
this is fairly easy. As before we may assume
$a\notin pV_{p^2}$. We have the following.
\begin{prop}\label{prop:residue_rm_nm}
If $a$ is of type $(1^21_\ast)$, $\esC_{p^2}(a,\chi)=0$.
Let $a$ be of type $(1^3_{\ast})$ or $(1^3_{\ast\ast})$.
We may assume $a=(1,a_2,a_3,0)$ is in
$\cD_{p^2}(1^3_{\ast})$
or $\cD_{p^2}(1^3_{\ast\ast})$,
and for these $a$,
\[
(1-p^{-2})\esC_{p^2}(a,\chi)
=\begin{cases} 1 & p\equiv1\mod 3,\\
\frac13(1+\chi_p(1+a_2+a_3)+\chi_p(1-a_2+a_3)) & p=3.\\
\end{cases}
\]
\end{prop}
\begin{proof}
If $a$ is of type $(1^21_\ast)$,
then by Lemma \ref{lem:stabilizer_singular_p^2} (1), 
$\chi\circ\det$ is non-trivial on $G_{p^2,a}$.
Hence Lemma \ref{lem:residue_basic} (1)
implies that $\esC_{p^2}(a,\chi)$ must vanish.

We put $R=\Z/p^2\Z$.
Let $a=(1,a_2,a_3,0)$
with $a_2,a_3\in pR$.
Depending on the type of $a$,
$a_3$ is in $pR^\times$ or $0$.
For $(u,v)\in W_{p^2}'$,
$a(u,v)=0$ if $u\in pR$. Hence by
\eqref{eq:cw} and \eqref{eq:i_rm},
\begin{align*}
(1-p^{-2})\esC_{p^2}(a,\chi)
&=p^{-4}(1-p^{-1})^{-1}
\sum_{u\in R^\times,v\in R}\chi_p(u^3+a_2u^2v+a_3uv^2).
\end{align*}
By changing $v$ to $uv$ and using that $\chi_p$ is cubic,
\begin{align*}
(1-p^{-2})\esC_{p^2}(a,\chi)
&=p^{-2}\sum_{v\in R}\chi_p(1+a_2v+a_3v^2).
\end{align*}
If $p\equiv1\mod 3$, then $\chi_p(1+a_2v+a_3v^2)=1$
since the conductor of $\chi_p$ is $p$.
If $p=3$, then $\chi_p(1+a_2v+a_3v^2)$ is determined
by $(v\mod 3)\in\{0,\pm1\}$.
Hence we have the formula.
\end{proof}

For $e\geq2$, we now compute $\esC_{p^e}(a,\chi)$
for $a\in V_{\Z_p}^{\max}$ or $a\in V_{p^e}^{\max}$.
For stating our result as well as
applications to counting cubic fields \cite{scc},
it is convenient to instead compute a quantity closely related to $\esC_{p^e}(a,\chi)$.
Let $a\in V_{\Z_p}^{\rm max}$.
We choose $e\geq c$ such that
\begin{enumerate}[(i)]
\item $G_{p^e}a+p^eV_{\Z_p}$ is a single $G_{\Z_p}$-orbit,
\item the value $\tilde\chi_p(P(a)/p^{\ord_p(P(a))})$ depends
only on $a\mod p^e$,
\end{enumerate}
and define
\begin{equation}
\tilde \esC_{p^e}(a,\chi):=(1-p^{-2})
\frac{\esC_{p^e}(a,\chi)}{\tilde\chi_p(P(a)/p^{\ord_p(P(a))})}.
\end{equation}
This depends only on the $G_{\Z_p}$-orbit of $a$, and
only on $a\mod p^e$. Hence
this depends only on the $G_{p^e}$-orbit of $a\mod p^e$.
For each $a$, we can choose such $e\geq c$ satisfying (i) and (ii)
as follows.
\begin{lem}
\begin{enumerate}[{\rm (1)}]
\item Let $p\equiv1\mod3$. For $a$ of type $(3), (21), (111)$, $e=1$ is enough.
For $a$ of type $(1^21_{\max}), (1^3_{\max})$, $e=2$ is enough.
\item Let $p=3$. For $a$ of type $(3), (21), (111)$, $e=2$ is enough.
For $a$ of type $(1^21_{\max}), (1^3_{\max})$, $e=3$ is enough.
\end{enumerate}
\end{lem}
\begin{proof}
By Propositions \ref{prop:e_enough}, \ref{prop:e_enough_3},
these $e$ satisfy (i).
We check (ii)
for each orbit individually.
For elements of type $(3), (21), (111)$ this is trivial, so
we work for the remaining cases.
Assume $p=3$ and
$a\in V_{\Z_p}(1^3_{\rm max})$.
A set of representatives of the various $G_{\Z_p}$-orbits are given
in the second table of the next proposition.
If $a$ lies, say, in the orbit of $a_{\rm rep}=(1,3,0,3)$
and $a'\equiv a\mod p^3$,
then $\ord_p(P(a))=4$.
Let $a=ga_{\rm rep}$. Then $g^{-1}a'\equiv a_{\rm rep}\mod p^3$.
If we write $g^{-1}a'=(a_1,a_2,a_3,a_4)$,
then by the definition of the polynomial $P$, we have
\[
P(g^{-1}a')\equiv -4a_2^3a_4-27a_1^2a_4^2\equiv P(a_{\rm rep})
\mod p^6.
\]
This shows that $P(a')/p^4\equiv P(a)/p^4\ \mod p^2$.
Hence the values of $\tilde\chi_p$ for $a, a'$ coincide;
recall that the conductor of $\tilde\chi_p$ is $p^2$. This proves (ii)
for this $G_{\Z_p}$-orbit.

The elements $a$ in other orbits are treated similarly, and
we omit the detail.
\end{proof}

We now give the value of $\tilde \esC_{p^e}(a,\chi)$ for
each $G_{\Z_p}$-orbit in $V_{\Z_p}^{\max}$.
We also list the minimal $e$, $|P(a)|_p^{-1}$ and $|G_{\Z_p,a}|$
for convenience. The result for $p\equiv1\mod 3$ is
due to Datskovsky-Wright \cite[Proposition 5.4]{dawra}.

\begin{prop}\label{prop:residue_rm_max}
Let $a\in V_{\Z_p}^{\max}$.
If $p\equiv1\mod 3$, we have the following table:
\[
\begin{array}{c|ccccc}
\hline
\text{\rm Type of $a$} & a\in V_{\Z_p} & e\geq & \tilde \esC_{p^e}(a,\chi) & |P(a)|_p^{-1} & |G_{\Z_p,a}|\\
\hline
\hline
(3) & a & 1 & \tau(\chi_p)^3/p^2 & 1 & 3\\
\hline
(21) & a & 1 & -\tau(\chi_p)^3/p^2 & 1 & 2\\
\hline
(111) & a & 1 & \tau(\chi_p)^3/p^2 & 1 & 6\\
\hline
(1^21_{\max}) & (0,1,0,p\alpha), \alpha\in\Z_p^\times& 2
	& \chi_p(2)\chi_p'(p)^2p^{-1/3} & p & 2\\
\hline
(1^3_{\max}) & (1,0,0,p\alpha), \alpha\in\Z_p^\times& 2
	& \chi_p(\alpha)+\chi_p(\alpha)^2\chi_p'(p)^2p^{-1/3} & p^2 & 3\\
\hline
\end{array}
\]
If $p=3$, we have
\[
\begin{array}{c|ccccc}
\hline
\text{\rm Type of $a$} & a\in V_{\Z_p} & e\geq & \tilde \esC_{p^e}(a,\chi) & |P(a)|_p^{-1} & |G_{\Z_p,a}|\\
\hline
\hline
(3) & a & 2 & \tau(\chi_p)^3/p^4=\chi_p(2)/p & 1 & 3\\
\hline
(21) & a & 2 & \tau(\chi_p)^3/p^4=\chi_p(2)/p & 1 & 2\\
\hline
(111) & a & 2 & \tau(\chi_p)^3/p^4=\chi_p(2)/p & 1 & 6\\
\hline
(1^21_{\max}) & (0,1,0,\pm3) & 3
	& \pm(1-\chi_p(4))\chi_p'(p)^2p^{-4/3} & p & 2\\
\hline
& (1,0,3,3) & 3 & (\chi_p(2)-1)/p & p^3 & 1\\
& (1,0,6,3) & 3 & (2\chi_p(2)+1)/p & p^3 & 1\\
(1^3_{\max}) & (1,3,0,3) & 3 & \chi_p(4)\chi_p'(p)^2p^{-1/3} & p^4 & 1\\
& (1,-3,0,3\alpha), \alpha=1,4,7 & 3 & \chi_p(\alpha)^2+\chi_p'(p)^2p^{-1/3} & p^4 & 3\\
& (1,0,0,3\alpha), \alpha=1,4,7 & 3 & \chi_p(\alpha)^2\chi_p'(p)^2p^{-1/3} & p^5 & 1\\
\hline
\end{array}
\]
\end{prop}
\begin{proof}
Let $p\equiv1\mod 3$.
By Lemma \ref{lem:residue_orbit},
for orbits of type $(3),(21),(111)$
this follows from Proposition \ref{prop:residue_rm_1}.
We consider orbits of type $(\sigma)=(1^21_{\max})$ or $(1^3_{\max})$.
By Lemma \ref{lem:residue_orbit} we may assume $e=2$,
but for potential further applications of our argument
we let $e\geq2$ be arbitrary.
We use \eqref{eq:cw} and \eqref{eq:i_rm} for computation.
Let $R=\Z/p^e\Z$. For $0\neq x\in R$,
let $\tilde\chi_p(x)=\tilde\chi_p(\tilde x)$ where
$\tilde x\in \Z_p$ is an arbitrary lift of $x$.
This is well defined since the conductor of $\chi_p$ is $p$.
Also if $x\in p^mR^\times$ for some
$0\leq m<e$ and $y\in x+p^{m+1}R$,
then $\tilde\chi_p(y)=\tilde\chi_p(x)$.

(i)
Let $a\in V_{p^e}(1^21_{\max})$.
We may assume $a=(0,1,a_3,a_4)$ where $a_3\in pR$ and $a_4\in pR^\times$.
Let  $(u,v)\in W_{p^{e}}'$.
If $u\in pR$, then $v\in R^\times$ and hence
$a(u,v)\in a_4v^3+p^2R$. This implies
$\tilde\chi_p(a(u,v))=\tilde\chi_p(a_4v^3)=\tilde\chi_p(a_4)$.
If $u\in R^\times$, then $a(u,v)\neq0$ if and only if $v\neq0$,
and in this case $a(u,v)\in v(u^2+pR)$ and hence
$\tilde\chi_p(a(u,v))=\tilde\chi_p(u^2v)$.
Hence by \eqref{eq:cw}, \eqref{eq:i_rm}, 
\begin{align*}
\esC_{p^e}(a,\chi)
&=\frac{1}{p+1}\frac{\tilde\chi_p(a_4)p^{2/3}}{1-p^{-1}}
+\frac{1}{p^{2e}(1-p^{-2})}\sum_{u\in R^\times,\ v\in R\setminus\{0\}}
\frac{\tilde\chi_p(u^2v)|v|^{-2/3}}{1-p^{-1}}\\
&=\frac{\tilde\chi_p(a_4)p^{-1/3}}{1-p^{-2}}
+\frac{1}{p^{2e}(1-p^{-2})(1-p^{-1})}\sum_{u\in R^\times}\tilde\chi_p(u^2)
\sum_{0\leq m<e}p^{2m/3}\sum_{v\in p^mR^\times}\tilde\chi_p(v).
\end{align*}
Since $\tilde\chi_p$ induces a non-trivial character on each
$(\Z/p^{e-m}\Z)^\times$ for $m<e$, we have
\[
\sum_{v\in p^mR^\times}\tilde\chi_p(v)
=\tilde\chi_p(p)^m\sum_{v'\in(\Z/p^{e-m}\Z)^\times}\tilde\chi_p(v')=0.
\]
Hence $\esC_{p^e}(a,\chi)=(1-p^{-2})^{-1}\tilde\chi_p(a_4)p^{-1/3}$.
In particular for $a=(0,1,0,p\alpha)\in V_{\Z_p}$, $\alpha\in\Z_p^\times$,
\[
\tilde\esC_{p^e}(a,\chi)
=\frac{\tilde\chi_p(p\alpha)p^{-1/3}}{\tilde\chi_p(-4\alpha)}
=\tilde\chi_p(2p)p^{-1/3}
=\chi_p(2)\chi_p'(p^2)p^{-1/3}.
\]
Note that the last equality follows from Lemma \ref{lem:character}.

(ii)
Let $a\in V_{p^e}(1^3_{\max})$.
We may assume $a=(1,a_2,a_3,a_4)$ where $a_2,a_3\in pR$ and $a_4\in pR^\times$.
If $u\in R^\times$, then $a(u,v)\in u^3+pR$ and hence
$\tilde\chi_p(a(u,v))=1$.
If $u\in pR$, then $a(u,v)\in a_4v^3+p^2R$ and hence
$\tilde\chi_p(a(u,v))=\tilde\chi_p(a_4)$.
Hence by \eqref{eq:cw}, \eqref{eq:i_rm}, 
\[
\esC_{p^e}(a,\chi)
=\frac{p}{p+1}\frac{1}{1-p^{-1}}+\frac{1}{p+1}\frac{\tilde\chi_p(a_4)p^{2/3}}{1-p^{-1}}
=\frac{1+\tilde\chi_p(a_4)p^{-1/3}}{1-p^{-2}}.
\]
The result in the table follows from this.
This finishes the proof for $p\equiv1\mod 3$.

Let $p=3$.
Then
since $\Q_p^\times=p^\Z\times\Z_p^\times$
and $(\Z_p/p^2\Z_p)^\times\cong(\Z/9\Z)^\times$
is generated by $2\in\Z_p^\times$,
$\tilde\chi_p$ is determined uniquely
by $\tilde\chi_p(p)$ and $\tilde\chi_p(2)$.
Now the results in the second table are verified
by explicitly evaluating the sum \eqref{eq:cw} using PARI/GP \cite{pari}.
Note the identity $\chi_p(2)=p^{-3}\tau(\chi_p)^3$.
\end{proof}

\begin{rem}\label{rem:orbitalL}
We now explain how Theorem \ref{thm:introL}
follows from our arguments. The functional equations of $\xi(s,a)$
and $\xi(s,\chi,a)$ are obtained as special cases of
Theorem \ref{thm:FESato}, due to F. Sato, and for 
$\xi(s,\chi,a)$ we stated this as Proposition \ref{prop:FEorbital}.
By Proposition \ref{prop:partial_orbital},
the residues of $\xi(s,a)$ are obtained from
those of $\xi(s,\chi,a)$.
By Proposition \ref{prop:unfold},
$\xi(s,\chi,a)$ is entire if $\chi^3$ is nontrivial,
and Theorem \ref{thm:residue_formal} and
Lemma \ref{lem:residue_finiteorbit}
express the residues of $\xi(s,\chi,a)$ as a sum over $G_N$ for $\chi$ cubic.
When $a$ corresponds to a maximal cubic ring for all $p\mid N$,
explicit residue formulas are proved in Propositions \ref{prop:residue_ur_max}
and \ref{prop:residue_rm_max}. When $N$ is cube free,
explicit formulas
are proved in Propositions \ref{prop:residue_ur_1}, \ref{prop:residue_rm_1} for $p\mid\mid N$,
and Propositions \ref{prop:residue_ur_max}, \ref{prop:residue_ur_nm},
\ref{prop:residue_rm_nm}, \ref{prop:residue_rm_max}
for $p^2\mid\mid N$.
\end{rem}

\section{Examples;
bias of class numbers in arithmetic progressions}\label{sec:example}
Let $\chi$ be a primitive Dirichlet character of conductor $m$.
For each sign we define
\begin{equation}\label{eq:standardL}
\xi_\pm(s,\chi)
:=\sum_{\substack{x\in\spl_2(\Z)\backslash V_\Z^\pm\\(P(x),m)=1}}
	\frac{|{\rm Stab}(x)|^{-1}\chi(P(x))}{|P(x)|^s},
\end{equation}
where $V_\Z^\pm=\{x\in V_\Z\mid \pm P(x)>0\}$
and ${\rm Stab}(x)$ denotes the stabilizer group of $x$ in $\spl_2(\Z)$.
This is also a standard construction of $L$-functions from
Shintani's zeta functions $\xi_\pm(s)$.
In this section we apply our analysis to describe the residues of
these zeta functions and their relatives, and prove biases of
class numbers in arithmetic progressions.
We also discuss how these results relate to Theorem \ref{thm:introsccchi}.

Let $h\in C(V_m)$ be the function defined by
\[
h(a)
=\begin{cases}
\chi(P(a)) 	& P(a)\in(\Z/m\Z)^\times,\\
0		& \text{otherwise},
\end{cases}
\qquad
a\in V_m.
\]
Then $h\in C(V_m,\chi^2)$
and by Proposition \ref{prop:f_inv_splZ},
$\xi(s,h)={}^{t}(\xi_+(s,\chi),\xi_-(s,\chi))$.
Proposition \ref{prop:residue_f}
asserts that each of $\xi_\pm(s,\chi)$ is holomorphic
if $\chi^6$ is non-trivial.

Assume $\chi^6=\bf1$.
We consider the case where $m$ is a power of an odd prime $p$.
(Since $m$ is the conductor of $\chi$,
$\chi^6=\bf1$ implies that $m=p$ except for the case
$p=3$ and $\chi$ is not quadratic, where $m=p^2$.)
Let $\lambda_p\in C(V_p)$ be as follows:
$\lambda_p(a)=1$ if $a$ is of type $(3)$ or $(111)$,
$\lambda_p(a)=-1$ if $a$ is of type $(21)$,
and $\lambda_p(a)=0$ otherwise.

If $\chi$ is quadratic, then Proposition \ref{prop:residue_f} implies that
$\xi_\pm(s,\chi)$ has possible simple poles at $s=1$ and $5/6$.
We compute the quantity $\esC_p(h,{\bf1})$ defined in \eqref{eq:residue_f}.
In this case $h=\lambda_p$, and
by Proposition \ref{prop:residue_ur_1} with Lemma \ref{lem:n_p(sigma)},
\[
\esC_p(h,{\bf1})
=(1-p^{-1})
\left\{
	\frac{(1-p^{-1/3})(1+p^{-1})}{3}
	-\frac{1-p^{-4/3}}{2}
	+\frac{(1-p^{-2/3})(1+p^{-1/3})^2}{6}
\right\}=0.
\]
Similarly $\esA_p(h)=\esB_p(h)=0$.
Hence $\xi_\pm(s,h)$ is in fact entire.

Now assume that $\chi^2\neq\bf1$ but $\chi^6=\bf1$, i.e.,
$\chi$ is either cubic or sextic.
Then $\xi_\pm(s,\chi)$ has a possible simple pole at $s=5/6$
and is holomorphic elsewhere.
Let $m=p\neq3$.
By Proposition \ref{prop:residue_rm_max},
\begin{align*}
\esC_p(h,\chi^2)
&=p^{-4}\sum_{P(a)\neq0}
\chi(P(a))\esC_p(a,\chi^2)
=\frac{p^{-6}\tau(\chi^2)^3}{1-p^{-2}}\sum_{P(a)\neq0}\lambda_p(a)\chi(P(a))^3\\
&=\begin{cases}
0 & \text{$\chi$ is cubic},\\
p^{-2}(1-p^{-1})\tau(\chi^2)^3 & \text{$\chi$ is sextic}.\\
\end{cases}
\end{align*}
If $m=p^2=3^2$, we have
\[
\esC_{p^2}(h,\chi^2)=
\begin{cases}
p^{-4}(1-p^{-1})\tau(\chi^2)^3 & \text{$\chi$ is cubic},\\
0 & \text{$\chi$ is sextic}.\\
\end{cases}
\]
So $\xi_\pm(s,\chi)$ has a pole at $s=5/6$ when
$\esC_{p^e}(h, \chi^2)$ does not vanish.

Now let $m$ be an arbitrary odd integer. Since $\esA,\esB,\esC$
have Euler products, based on the computations above
we get the following residue formula.
Recall the decomposition
$\chi=\prod\chi_p$ we introduced at the beginning of
Section \ref{sec:residue}.
\begin{thm}\label{thm:standardL}
Assume that the conductor $m$ of $\chi$ is odd and $m\neq1$.
Then, the $\xi_\pm(s,\chi)$
are holomorphic except for a simple pole at $s=5/6$
which occurs if $\chi_p$ is of order $6$ for all $3\neq p\mid m$
and in addition $\chi_3$ is of order $3$ if $3\mid m$.
In this case the residues are
\begin{equation*}\label{eq:residue_standardL}
\underset{s=5/6}{\res}\ \xi_\pm(s,\chi)
=K_\pm\frac{2\pi^2\prod_{p\mid m}(1-p^{-1})}{9\Gamma(2/3)^3m^2}
\tau(\chi^2)^3L(1/3,\chi^{-2}).
\end{equation*}
Here $K_+=1, K_-=\sqrt3$,
$\tau(\chi^2)=\sum_{t\in(\Z/m\Z)^\times}\chi^2(t)\exp(2\pi it/m)$
is the Gauss sum, and $L(s,\chi)$ is the usual Dirichlet $L$-function.
If $\chi$ is of odd conductor $m>1$ but does not satisfy
the properties above, then the $\xi_\pm(s,\chi)$ are entire.
\end{thm}
\begin{proof}
We assume $\chi^6=\bf1$ and compute the residue at $s=5/6$;
the residue computation at $s=1$ is similar.
Since $\chi$ is primitive, $\chi^6=\bf1$ implies that
each $\chi_p$ is quadratic, cubic, or sextic.

We first consider the case $3\nmid m$. Then $m$ is square free.
Let us write $h=\prod_{p\mid m}h_p$, $h_p\in C(V_p)$,
so that
$h_p(a)=\chi_p(P(a))$ if $P(a)\in(\Z/p\Z)^\times$
and $h_p(a) = 0$ if $P(a)=0$.
Then by \eqref{eq:residue_Eulerp} and \eqref{eq:residue_f},
$\esC_N(h,\chi^2)=\prod_{p\mid m}\esC_p(h_p,\chi^2)$.
$\esC_p(h_p,\chi^2)$ is computed as above;
it is $p^{-2}(1-p^{-1})\tau(\chi_p^2)^3$ if $\chi_p$ is sextic
and $0$ if $\chi_p$ is quadratic or cubic.
Hence $\esC_N(h,\chi^2)=0$ if any $\chi_p$ is not sextic.
Thus assume that all $\chi_p$ are sextic.
Then by the decomposition formula for the classical Gauss sum
(recalled before Proposition \ref{prop:ogs_decomposition}),
\[
\frac{1}{m^2}\cdot{\tau(\chi^2)^3}=
\frac{1}{m^2}\prod_{p\mid m}\chi_p(m/p)^6\tau(\chi_p^2)^3=\prod_{p\mid m}\frac{\tau(\chi_p^2)^3}{p^2},
\]
and we conclude that
\[
\esC_N(h,\chi^2)=\frac{\tau(\chi^2)^3}{m^2}\prod_{p\mid m}(1-p^{-1}).
\]
Hence \eqref{eq:residue_standardL} follows from
Proposition \ref{prop:residue_f} and \eqref{eq:albtgm}.

The case $3\mid m$ is similarly done;
$\esC_N(h,\chi^2)=0$ unless $\chi_3$ is cubic
and $\chi_p$ is sextic for all $3\neq p\mid m$,
and in this case we have \eqref{eq:residue_standardL}
because of the identity
\[
\frac{1}{m^2}\cdot{\tau(\chi^2)^3}
=\frac{\tau(\chi_3^2)^3}{3^4}\prod_{p\mid m}\frac{\tau(\chi_p^2)^3}{p^2}.
\vspace{-.8cm}
\]
\end{proof}
From this result, we can prove the two main terms of
the function counting the class numbers of
integral binary cubic forms in arithmetic progressions.
This result implies that there is a bias
in the second main term if and only if
the odd modulus admits a character of order $6$.
\begin{thm}\label{thm:bias_classnumber}
Let $h_\pm(n)$ be the coefficients of Shintani's
original zeta function $\xi_\pm(s)$,i.e., $\xi_\pm(s)=\sum{h_\pm(n)}/n^s$.
Let $N$ be an odd integer and $a$ an integer coprime to $m$. Then
\begin{equation*}
\sum_{\substack{0<n<X\\ n\equiv a(\!\!\!\!\!\mod N)}}\!\!\!h_\pm(n)
=C_\pm'\frac{\pi^2\prod_{p\mid N}(1-p^{-2})}{9N}\cdot X
+K_1(N,a)\frac{2K_\pm\pi^2}{9\Gamma(2/3)^3N}\cdot \frac{X^{5/6}}{5/6}
+O_{N,\epsilon}(X^{3/5+\epsilon}),
\end{equation*}
where
$C_+'=1,C_-'=3/2, K_+=1,K_-=\sqrt3$, and
\begin{equation*}
K_1(N,a)=\sideset{}{'}\sum_{\chi^6={\bf1}}\chi(a)^{-1}
\frac{\tau(\chi^2)^3L(1/3,\chi^{-2})}{m_\chi^2}\prod_{\substack{p\mid N,\ p\nmid m_\chi}}(1-\chi(p)^{-2}p^{-4/3}).
\end{equation*}
Here the sum above
is over primitive characters $\chi$ whose conductor $m_\chi$
is a divisor of $N$
(including the trivial character modulo $1$),
such that if we write $\chi=\prod_{p\mid m_\chi}\chi_p$
then each $\chi_p$ has exact order $6$ for $p\neq3$ and
moreover $\chi_3$ has exact order $3$ if $3\mid m_\chi$.
\end{thm}
By the Delone-Faddeev correspondence, we can also state
Theorem \ref{thm:bias_classnumber} as a formula
counting discriminants of cubic rings in arithmetic progressions.
\begin{proof}
Let $\chi$ be a primitive Dirichlet character
whose conductor $m_\chi$ is a divisor of $N$.
We define
$\xi_\pm^N(s,\chi)$
by the formula \eqref{eq:standardL} with the sum
restricted to those $x$ with $P(x)$ coprime to $N$
(rather than $m_\chi$).
Then Proposition \ref{prop:residue_f} and Corollary \ref{cor:f_p} imply
\begin{align*}
\underset{s=1}{\res}\ \xi_\pm^N(s,{\bf 1})
&=\underset{s=1}{\res}\ \xi_\pm(s,{\bf 1})
\prod_{p\mid N}\left(1-p^{-1}\right)\left(1-p^{-2}\right),\\
\underset{s=5/6}{\res}\ \xi_\pm^N(s,\chi)
&=\underset{s=5/6}{\res}\ \xi_\pm(s,\chi)
\prod_{p\mid N,\ p\nmid m_\chi}
\left(1-p^{-1}\right)\left(1-\chi(p)^{-2}p^{-4/3}\right),
\end{align*}
where if $\chi^6\neq{\bf 1}$ the second formula means
the formal equality $0=0$.
On the other hand, Sato-Shintani's Tauberian theorem
\cite[Theorem 3]{sash} asserts
\[
\sum_{0<n<X, (n,N)=1}h_\pm(n)\chi(n)=
\underset{s=1}{\res}\ \xi_\pm^N(s,\chi)X
+\underset{s=5/6}{\res}\ \xi_\pm^N(s,\chi)\frac{X^{5/6}}{5/6}
+O_{N,\epsilon}(X^{3/5+\epsilon}).
\]
Now the theorem follows from the residue formulas in
Theorems \ref{thm:residueShintani} and \ref{thm:standardL}
with the orthogonality of characters.
Note that $\varphi(N)=N\prod_{p\mid N}(1-p^{-1})$.
\end{proof}

Let $\mathcal P$ be a finite set of primes.
We define the $\mathcal P$-maximal $L$-function
$\xi_\pm^{\mathcal P}(s,\chi)$
by the formula \eqref{eq:standardL} with the sum
restricted to those $x$ satisfying $(x\mod p^2)\in V_{p^2}^{\max}$
for all $p\in\mathcal P$. Note that
$(P(x),m)=1$ implies $(x\mod p^2)\in V_{p^2}^{\max}$
for $p\mid m$, hence only primes $p\in\mathcal P$ coprime
to $m$ are relevant for the definition.
Then $\xi_\pm^{\mathcal P}(s,\chi)$ again
has a pole under the same condition
for $\xi_\pm(s,\chi)$, and by Proposition \ref{prop:residue_f}
and Corollary \ref{cor:f_p^2_residue}
the residues are
\begin{equation}\label{eq:residue_standardLmax}
\underset{s=5/6}{\res}\ \xi_\pm^{\mathcal P}(s,\chi)
=\underset{s=5/6}{\res}\ \xi_\pm(s,\chi)
\prod_{p\in\mathcal P,\ p\nmid m}
\left(1-p^{-2}\right)\left(1-{\chi^2(p)p^{-5/3}}\right).
\end{equation}

The poles at $s=5/6$ of $\xi_\pm^{\mathcal P}(s,\chi)$,
as well as of $\xi_\pm(s,\chi)$, are the source of the biases
we described in Theorem \ref{thm:introsccchi}.
Indeed, for $\chi$ as in Theorem \ref{thm:standardL},
we prove in \cite{scc} that
\begin{equation}\label{eq:density_twist}
\sum_{\substack{[F:\Q]=3,\ 0<\pm\Disc(F)<X\\ (\Disc(F),m)=1}}\chi(\Disc(F))
=\frac{K_\pm(\chi)}{2}\frac{X^{5/6}}{5/6}+O(m^{8/9}X^{7/9+\epsilon}),
\end{equation}
where $K_\pm(\chi)$
is the limit of \eqref{eq:residue_standardLmax}
as $\mathcal P$ tends to the set of all primes:
\[
K_\pm(\chi):=
\frac{4K_\pm\tau(\chi^2)^3}{3\Gamma(2/3)^3m^2\prod_{p\mid m}(1+p^{-1})}
\frac{L(1/3,\chi^{-2})}{L(5/3,\chi^2)}.
\]
The $2$ in the denominator of $\frac{K_\pm(\chi)}{2}$
in \eqref{eq:density_twist} is the index $[\gl_2(\Z):\spl_2(\Z)]$.
This appears because the Shintani zeta functions count
$\spl_2(\Z)$-orbits, while cubic fields correspond to $\gl_2(\Z)$-orbits.

We briefly explain other variations as well.
Suppose first that the conductor $m$ is a power of $p=2$.
Then there are no cubic nor sextic characters, but
are three quadratic characters $\chi$. One is of conductor $4$,
and the two others are of conductor $8$.
To compute the residues, we note that
for $a\in V_{\Z_2}$ of type $(3)$, $(21)$ or $(111)$,
$P(a)\mod 8$ is given by
\begin{equation}\label{eq:P(a)mod8}
P(a)\equiv
\begin{cases}
1 \mod 8& \text{$a$ : of type $(3)$, $(111)$},\\
5 \mod 8& \text{$a$ : of type $(21)$}.
\end{cases}
\end{equation}
This is easily verified for the representatives
$(1,0,1,1)\in V_{\Z_2}(3)$, $(0,1,1,1)\in V_{\Z_2}(21)$
and $(0,1,1,0)\in V_{\Z_2}(111)$ and hence is true for any element $a$
because $P(ga)=(\det g)^2P(a)$ and
$(\det g)^2\in(\Z_2^\times)^2=1+8\Z_2$.
Let $\chi$ be of conductor $4=p^2$.
By \eqref{eq:P(a)mod8}
$h(a)=\chi(P(a))$ is always $1$ when $P(a)\in(\Z/4\Z)^\times$,
and we have
\[
\esA_{p^2}(h)=\esB_{p^2}(h)=(1-p^{-1})(1-p^{-2}),
\quad
\esC_{p^2}(h,{\bf1})=(1-p^{-1})(1-p^{-4/3}).
\]
Hence $\xi_\pm(s,\chi)$ has poles both at $s=1$ and $s=5/6$ for this $\chi$.
But this is fairly reasonable,
because $P(x)$ is always $\equiv 0,1\mod 4$
for $x\in V_\Z$, and so
$\xi_\pm(s,\chi)$ simply counts orbits with $P(x)\equiv1\mod4$
without a twist.
On the other hand, if $\chi$ is either character of conductor $8=p^3$,
by \eqref{eq:P(a)mod8}
$h(a)=1$ if $a$ is of type $(3)$ or $(111)$,
$h(a)=-1$ if $a$ is of type $(21)$,
and $h(a)=0$ otherwise.
Hence we have
$
\esA_{p^3}(h)=\esB_{p^3}(h)=\esC_{p^3}(h,{\bf1})=0
$
and the $\xi_\pm(s,\chi)$ are entire.

Second, this observation for $m=2^c$
allows us to extend Theorem \ref{thm:standardL}
to $m$ even. This consists of case by case
descriptions corresponding to conditions on $\chi_2$,
and we omit the detail.

Third, let $r$ be a positive integer. Then
\begin{equation}\label{eq:standardLr}
\xi_\pm(s,r,\chi)
:=\sum_{\substack{x\in\spl_2(\Z)\backslash V_\Z^\pm\\r\mid P(x),\ (P(x)/r,m)=1}}
	\frac{|{\rm Stab}(x)|^{-1}\chi(P(x)/r)}{|P(x)|^s}
\end{equation}
is also a natural $L$-function.
Since ${}^{t}(\xi_+(s,r,\chi),\xi_-(s,r,\chi))=\xi(s,h)$
for an appropriate $h\in C(V_N,\chi^2)$,
we can study these as well. In particular
it is entire if $\chi^6\neq\bf1$, and
we can describe their residues explicitly when $\chi^6=\bf1$
for the case we can apply the residual computations.
This includes the case when $r$ is cubefree and
$p\nmid m$ for all $p^2\mid r$.
As a simplest example, let $m=r=p\neq 2,3$.
Then $h\in C(V_{p^2},\chi^2)$ is given by
\[
h(a)
=\begin{cases}
\chi(P(a)/p) 	& a\in V_{p^2}(1^21_{\max}),\\
0		& \text{otherwise}.
\end{cases}
\]
For $\chi$ quadratic,
by Proposition \ref{prop:residue_ur_max} with
Proposition \ref{prop:singular_orbits} (1) (ii),
we see that
$\esA_{p^2}(h)=\esB_{p^2}(h)=\esC_{p^2}(h,{\bf1})=0$.
For $\chi$ cubic or sextic,
by Proposition \ref{prop:residue_rm_max} we have
\begin{align*}
\esC_{p^2}(h,\chi^2)
&=\sum_{a\in V_{p^2}(1^21_{\max})}\chi(P(a)/p)\esC_p(a,\chi^2)
=\frac{\chi(4)p^{-1/3}}{1-p^{-2}}\sum_{a\in V_{p^2}(1^21_{\max})}\chi^3(P(a)/p)\\
&=
\begin{cases}
\chi(4)(1-p^{-1})p^{-4/3} & \text{$\chi$ is cubic},\\
0& \text{$\chi$ is sextic}.\\
\end{cases}
\end{align*}
Hence the $\xi_\pm(s,p,\chi)$ are entire unless $\chi$ is cubic,
in which case their residues at $s=5/6$ are
\begin{equation*}\label{eq:residue_standardLr}
\underset{s=5/6}{\res}\ \xi_\pm(s,p,\chi)
=K_\pm\frac{2\pi^2\chi(4)(1-p^{-1})}{9\Gamma(2/3)^3p^{4/3}}
L(1/3,\chi^{-2}).
\end{equation*}
This again is a source of the bias in Theorem \ref{thm:introsccchi}
for $m=p^2$ and $(m,a)=p$.

Moreover, if $r=r(X_N)$
for a union of $G_N$-orbits $X_N$ in $V_N$ is well defined,
then we can define $\xi_\pm(s,X_N,\chi)$
by \eqref{eq:standardLr} with the sum
restricted to those $x$ satisfying $(x\mod N)\in X_N$.
This is in fact possible if $X_N$ detects certain
maximal cubic rings over $\Z_p$ for each $p\mid N$,
and as we computed the contributions to the residues
for all $a\in V_{p^e}^{\rm max}$
in Propositions \ref{prop:residue_ur_max}, \ref{prop:residue_rm_max},
we can describe the residues explicitly. This enables us to impose
local specifications while counting
cubic fields in arithmetic progressions.
For details, see \cite[Section 6.4]{scc}.

\end{document}